\documentclass[11pt,a4paper]{article}
\usepackage{tikz}
\usepackage[margin=1in]{geometry}
\usepackage{palatino}
\usepackage{mathpazo} %math font
\usepackage{subcaption}

\usepackage{enumitem}

\usepackage[
  colorlinks,
  linkcolor = blue,
  citecolor = blue,
  urlcolor = blue]{hyperref}

\usepackage{amsmath,amssymb,amsthm}
\usepackage{mathtools}
\mathtoolsset{centercolon}

\usepackage{thmtools, thm-restate}

\usepackage[affil-it]{authblk}

\newcommand{\be}{\begin{equation}}
\newcommand{\ee}{\end{equation}}

\renewcommand{\P}{\mathcal{P}}

\DeclareMathOperator{\tr}{Tr}

\DeclareMathOperator{\len}{len}

\newcommand{\Thm}[1]{\hyperref[thm:#1]{Theorem~\ref*{thm:#1}}}
\newcommand{\Lem}[1]{\hyperref[lem:#1]{Lemma~\ref*{lem:#1}}}
\newcommand{\Cor}[1]{\hyperref[cor:#1]{Corollary~\ref*{cor:#1}}}
\newcommand{\Def}[1]{\hyperref[def:#1]{Definition~\ref*{def:#1}}}
\newcommand{\Obs}[1]{\hyperref[obs:#1]{Observation~\ref*{obs:#1}}}
\newcommand{\Prop}[1]{\hyperref[prop:#1]{Proposition~\ref*{prop:#1}}}
\newcommand{\Rem}[1]{\hyperref[rem:#1]{Remark~\ref*{rem:#1}}}
\newcommand{\Ex}[1]{\hyperref[ex:#1]{Example~\ref*{ex:#1}}}

\newcommand{\Sec}[1]{\hyperref[sec:#1]{Section~\ref*{sec:#1}}}

\newcommand{\Fig}[1]{\hyperref[fig:#1]{Figure~\ref*{fig:#1}}}
\newcommand{\Tab}[1]{\hyperref[tab:#1]{Table~\ref*{tab:#1}}}
\newcommand{\EqRef}[1]{\hyperref[eq:#1]{(\ref*{eq:#1})}}
\newcommand{\Eq}[1]{Equation~\hyperref[eq:#1]{(\ref*{eq:#1})}}

\makeatletter
\newtheorem{theorem}{Theorem}[section]
\newtheorem*{theorem*}{Theorem}
\newtheorem{lemma}[theorem]{Lemma}
\newtheorem{cor}[theorem]{Corollary}

\theoremstyle{definition}
\newtheorem{definition}[theorem]{Definition}
\newtheorem{remark}[theorem]{Remark}
\newtheorem{example}[theorem]{Example}
\newtheorem{conjecture}{Conjecture}
\newtheorem{question}{Question}

\newcommand{\qeds}{\qed\vspace{.2cm}}
\newcommand{\oddto}{\xrightarrow{\mathrm{odd}}}
\newcommand{\woto}{\xrightarrow{\mathrm{w.o.}}}
\newcommand{\F}{\ensuremath{\mathcal{F}}}
\DeclareMathOperator{\Hom}{Hom}

\title{Oddomorphisms and homomorphism indistinguishability over graphs of bounded degree}

\author{David E.~Roberson} 

\affil{Department of Applied Mathematics and Computer Science, Technical University of Denmark%, DK-2800 Lyngby, Denmark
\\
QMATH, Department of Mathematical Sciences, University of Copenhagen}

\begin{document}

\maketitle

\begin{abstract}
We introduce (weak) oddomorphisms of graphs which are homomorphisms with additional constraints based on parity. These maps turn out to have interesting properties (e.g., they preserve planarity), particularly in relation to homomorphism indistinguishability. 

%, and weak oddomorphisms which are homomorphisms that can be made into oddomorphisms by restricting to a subgraph.

Graphs $G$ and $H$ are \emph{homomorphism indistinguishable} over a family $\F$ if $\hom(F,G) = \hom(F,H)$ for all $F \in \F$, where $\hom(F,G)$ is the number of homomorphisms from $F$ to $G$. A classical result of Lov\'{a}sz says that isomorphism is equivalent to homomorphism indistinguishability over the class of all graphs. In recent years it has been shown that many homomorphism indistinguishability relations have natural algebraic and/or logical formulations. Currently, much research in this area is focused on finding such reformulations. We aim to broaden the scope of current research on homomorphism indistinguishability by introducing new concepts/constructions and proposing several conjectures/questions. In particular, we conjecture that every family closed under disjoint unions and minors gives rise to a distinct homomorphism indistinguishability relation.

We also show that if $\F$ is a family of graphs closed under disjoint unions, restrictions to connected components, and weak oddomorphisms, then $\F$ satisfies a certain maximality or closure property: homomorphism indistinguishability over $\F$ of $G$ and $H$ does not imply $\hom(F,G) = \hom(F,H)$ for any $F \notin \F$. This allows us to answer a question raised over ten years ago, showing that homomorphism indistinguishability over graphs of bounded degree is not equivalent to isomorphism.%for any $G \notin \F$ there are $H$ and $H'$ that are homomorphism indistinguishable over $\F$, but $\hom(G,H) \ne \hom(G,H')$.

\end{abstract}

\section{Introduction}

In 1967 Lov\'{a}sz~\cite{lovasz} showed that graphs $G$ and $H$ are isomorphic, denoted $G \cong H$, if and only if they admit the same number of homomorphisms from any graph. In other words, any graph $G$ is determined (up to isomorphism) by the numbers $\hom(F,G)$ of homomorphisms from $F$ to $G$, if $F$ is allowed to vary over all graphs. A seemingly natural question to ask is whether the homomorphism counts from all graphs are necessary. In other words, is there some proper sub-family $\F$ of graphs such that $\hom(F,G) = \hom(F,H)$ for all $F \in \F$ implies that $G$ and $H$ are isomorphic?\footnote{It is well known that for given graphs $G$ and $H$ it suffices to count homomorphisms from graphs with at most $\min\{|V(G)|,|V(H)|\}$ vertices, but we are interested in families that do not depend on $G$ and $H$.} We say that such a family \emph{distinguishes} all graphs. As we will see, there are known families with this property~\cite{dvorak}, but much is still uncharted and it remains an interesting, and often difficult, question whether a given family distinguishes all graphs. We will answer this question (always in the negative) for several specific families of graphs. Moreover, the concepts and techniques that we introduce are clearly more widely applicable to such problems.

If a family $\F$ does not distinguish all graphs, then there are non-isomorphic graphs $G$ and $H$ such that $\hom(F,G) = \hom(F,H)$ for all $F \in \F$. We say such graphs $G$ and $H$ are \emph{homomorphism indistinguishable over $\F$} and write $G \cong_\F H$. Thus Lov\'{a}sz' result can be rephrased as stating that homomorphism indistinguishability over the family of all graphs is isomorphism. More generally, homomorphism indistinguishability over any given family is clearly an equivalence relation on graphs, and we can now begin to consider the landscape of these relations. Such relations have only really begun to be investigated in the past few years, with the exception of~\cite{dvorak} from 2010. So far, the main line of investigation has been in showing that $\cong_\F$ has a ``nice" characterization for some specific families of graphs. Examples include showing that homomorphism indistinguishability over trees is equivalent to \emph{fractional isomorphism}~\cite{DGR}, defined as the existence of a doubly stochastic matrix $D$ such that $A_GD = DA_H$ where $A_G$ is the adjacency matrix of $G$. This is further known to have several other equivalent formulations~\cite{tinhofer}, including in terms of logic~\cite{immerman1990}. The result for trees has been extended to graphs of bounded treewidth, showing that homomorphism indistinguishability over this class can be expressed in terms of the Sherali-Adams hierarchy of relaxations of the natural integer linear program for isomorphism~\cite{DGR}, and in terms of equivalence over first-order logic with counting using a bounded number of variables~\cite{dvorak}. Homomorphism indistinguishability over graphs of bounded tree\emph{depth} has also been characterize as equivalence over first-order logic with counting and bounded quantifier rank~\cite{grohetreedepth}. A somewhat different flavor of characterization was proven for homomorphism indistinguishability over planar graphs~\cite{qplanar}: this is equivalent to \emph{quantum isomorphism}~\cite{qiso1}, which can be defined as the existence of a \emph{quantum permutation matrix} $U$ satisfying $A_GU = UA_H$. This is similar to fractional isomorphism but a quantum permutation matrix has entries which are projections in some $C^*$-algebra and whose rows and columns sum to identity. As more and more characterizations of homomorphism indistinguishability relations are found, we begin to see more and more of a beautiful theory emerging. With this work however, we hope to broaden the scope of homomorphism indistinguishability research beyond the question of characterizations, and to focus more on the landscape of homomorphism indistinguishability relations as a whole. We do this through the introduction of useful new concepts and techniques, as well as by proposing some interesting conjectures and questions.

We remark that we do not mean to discount the work being done on characterizing homomorphism indistinguishability relations. Each such characterization is another step towards a full understanding of homomorphism indistinguishability. We particularly highlight the recent work of Grohe, Rattan, and Seppelt~\cite{homtensors}. They provide a very general framework for producing algebraic characterizations of homomorphism indistinguishability relations. They are able to reproduce several known results in a unified manner, as well as to prove new ones. But the real focus of their work is on understanding the relationship between a graph family and the corresponding homomorphism indistinguishability relation. Thus their work will clearly be important not only for coming up with specific characterizations, but for better understanding the larger picture of homomorphism indistinguishability.

%As can be seen by these examples, most characterizations of homomorphism indistinguishability relations are either algebraic, i.e., in terms of certain systems of equations, or in terms of logical equivalence.

%bounded treewidth graphs in terms of , e.g., bounded treewidth~\cite{dvorak,DGR}, bounded treedepth~\cite{grohetreedepth}, or planar graphs~\cite{qplanar}. These characterizations are usually either algebraic, i.e., in terms of certain systems of equations, or in terms of logical equivalence. With this work we hope broaden the scope of homomorphism indistinguishability research beyond the question characterizations. We do this through the introduction of useful new concepts and techniques, as well as by proposing some interesting conjectures and questions.

\subsection{Contributions}\label{sec:contributions}

Since $\hom(F_1 \cup F_2,G) = \hom(F_1,G)\hom(F_2,G)$ for all graphs $F_1$, $F_2$, and $G$, homomorphism indistinguishability over a family $\F$ is the same as homomorphism indistinguishability over the family of all disjoint unions of graphs in $\F$. Thus we may as well restrict our concern to families that are already closed under taking disjoint unions (we will call such families \emph{union-closed}). As can be seen from the examples given above, the vast majority of known characterizations of homomorphism indistinguishability relations correspond to minor-closed families of graphs. This may just be a coincidence, but it at least makes one wonder if all minor- and union-closed families give rise to ``nice" homomorphism indistinguishability relations\footnote{Note though that minor-closed families have the property that the relation $\cong_\F$ is preserved under taking graph complements, i.e., $G \cong_\F H$ if and only if $\overline{G} \cong_\F \overline{H}$~\cite[Remark~7.13]{qplanar}. So homomorphism indistinguishability over minor-closed classes at least have one ``nice" property.}. Though this is a vague and not purely mathematical question, it leads us to a more concrete one which we state in the form of a conjecture: 

\begin{conjecture}\label{conj:distinct}
Let $\F_1$ and $\F_2$ be two distinct minor- and union-closed families of graphs. Then the relations $\cong_{\F_1}$ and $\cong_{\F_2}$ are distinct.
\end{conjecture}

There are some natural union-closed families that give rise to the same relations. For instance, Dvo\v{r}\'{a}k showed that homomorphism counts from 2-degenerate graphs (graphs whose subgraphs all have minimum degree at most 2) determine a graph up to isomorphism~\cite{dvorak}. Thus all graphs and 2-degenerate graphs correspond to the same homomorphism indistinguishability relation. Notably, 2-degenerate graphs are not minor-closed.

For any families $\F_1$ and $\F_2$, it is immediate that if $\F_1 \supseteq \F_2$, then $G \cong_{F_1} H \Rightarrow G \cong_{\F_2} H$ for all graphs $G$ and $H$, i.e., that $\cong_{\F_1}$ is at least as fine an equivalence relation as $\cong_{\F_2}$ (we will denote this by $\cong_{\F_1} \ \Rightarrow \ \cong_{\F_2}$ for brevity). We may hope that for minor- and union-closed classes this implication can be made an if and only if:

\begin{conjecture}\label{conj:notimply}
If $\F_1$ and $\F_2$ are minor- and union-closed families of graphs, then $\cong_{\F_1} \ \Rightarrow \ \cong_{\F_2}$ if and only if $\F_1 \supseteq \F_2$.
\end{conjecture}

Note that since one direction is trivial, the above conjecture is equivalent to the statement that if $\F_1$ and $\F_2$ are minor- and union-closed families such that $\F_1 \not\supseteq \F_2$, then $\cong_{\F_1} \ \not\Rightarrow \ \cong_{\F_2}$. From this formulation it is immediate that Conjecture~\ref{conj:notimply} implies Conjecture~\ref{conj:distinct}. The other implication is not as apparent, but we will prove it in Section~\ref{sec:conjectures}.

In our view, proving either Conjecture~\ref{conj:distinct} or~\ref{conj:notimply} does not appear to be straightforward. Given some arbitrary minor- and union-closed families, how does one go about determining how their corresponding homomorphism indistinguishability relations compare? However, we will also show that these conjectures are equivalent to the following, which we believe is more approachable.

\begin{conjecture}\label{conj:graphs}
For any connected graph $G$, there exist graphs $H$ and $H'$ such that \begin{enumerate}
    \item $\hom(G,H) \ne \hom(G,H')$, and
    \item $\hom(F,H) \ne \hom(F,H')$ implies that $F$ contains $G$ as a minor.
\end{enumerate}
\end{conjecture}
Note that the second condition above is equivalent to saying that $H \cong_{\F_G} H'$ where $\F_G$ is the family of graphs not containing $G$ as a minor. Thus the above conjecture can be rephrased as saying that $H \cong_{\F_G} H'$ does not imply  $\hom(G,H) = \hom(G,H')$.

We will also show that the above conjectures hold if and only if every minor- and union-closed family $\F$ satisfies a kind of \emph{maximality} or \emph{closure} with respect to their corresponding homomorphism indistinguishability relation. By this we mean that for any graph $F \notin \F$, there exist graphs $G$ and $H$ such that $G \cong_\F H$ but $\hom(F,G) \ne \hom(F,H)$. In other words, adding any graph to $\F$ would change the relation. We say such families are \emph{homomorphism distinguishing closed}, or simply \emph{h.d.-closed}. Thus we conjecture the following:

\begin{conjecture}\label{conj:hdclosed}
Every minor- and union-closed class of graphs is homomorphism distinguishing closed.
\end{conjecture}

Note that for any family $\F$, we can take its \emph{homomorphism distinguishing closure}:
\[\mathrm{cl}(\mathcal{F}) := \{F : H \cong_{\mathcal{F}} H' \Rightarrow \hom(F,H) = \hom(F,H')\}.\]
It is clear from the definition that $\mathrm{cl}(\F)$ is h.d.-closed, $\mathrm{cl}(\mathrm{cl}(\F)) = \mathrm{cl}(\F)$, and the relations $\cong_\F$ and $\cong_{\mathrm{cl}(\F)}$ are the same. Moreover, any two distinct h.d.-closed families must give rise to different homomorphism indistinguishability relations. Therefore, since we are only interested in the relations $\cong_\F$, we are only interested in homomorphism distinguishing closed families, as these are in one-to-one correspondence. Also note that we can prove that $\F$ does not distinguish all graphs by showing that it is h.d.~closed and does not contain all graphs.

We remark that though we have made the above conjectures for minor- and union-closed families, our main goal is to understand what types of families are homomorphism distinguishing closed, or to identify some nice family of families of graphs that all give rise to distinct (and preferably nice) homomorphism indistinguishability relations. So ``minor-closed" may need to be replaced with some other property in order for the conjecture to hold. Fortunately, we prove that the conjectures are equivalent even when ``minor-closed" is replaced with some other property as long as it satisfies certain relatively mild conditions.

In Section~\ref{sec:hdclosed} we prove that several families are h.d.-closed and therefore their corresponding homomorphism indistinguishability relations are not isomorphism. We prove this for the family of graphs of any bounded maximum degree, graphs of any bounded from above circumference (length of longest cycle), graphs with upper bounded sized chordless cycles, forests, graphs of treewidth at most two, graphs with no odd holes, graphs with bounded sized induced stars, and others.

Notably, our proof that homomorphism indistinguishability over graphs of bounded maximum degree is not isomorphism answers a question first raised over 10 years ago by Dvo\v{r}\'{a}k~\cite{dvorak}. They observed that if \emph{contractors}\footnote{A contractor is a notion from graph limits. The full definition is outside the scope of this work, but essentially a contractor for a graph $G$ is a linear combination of graphs such that gluing it between two vertices in a graph $F$ acts the same as contracting those two vertices with regards to counting homomorphisms from $F$ to $G$.} of bounded degree exist, then graphs of bounded degree distinguish all graphs. Therefore our result implies that contractors of bounded degree do not exist. The question of whether homomorphism indistinguishability over graphs of bounded degree is isomorphism was also asked independently in more recent years~\cite{DGR}. Moreover, Grohe, Rattan, and Seppelt very recently addressed this question in their work~\cite{homtensors}. Though they were not able to resolve the question fully, they showed that homomorphism indistinguishability over \emph{trees} of bounded degree yields a strict hierarchy of increasingly finer equivalence relations. Moreover, they are able to give an algebraic characterization of these homomorphism indistinguishability relations, which we do not attempt for graphs of bounded degree. In fact we do not provide a characterization of any homomorphism indistinguishability relation in this work.

In section~\ref{sec:uncountability} we prove that there are uncountably many homomorphism indistinguishability relations by showing that every family of complete graphs gives rise to a distinct such relation. In Section~\ref{sec:loops} we consider allowing loops which in certain cases allows us to strengthen our results. In Section~\ref{sec:discussion} we present some further questions and directions raised by our work. In particular, we note that to prove Conjecture~\ref{conj:distinct} it suffices to show that a graph $F$ having an oddomorphism to $G$ implies that $F$ contains $G$ as a minor.

\subsection{Techniques}\label{sec:techniques}

%In Theorem~\ref{thm:homclosed1} we prove that if $\F$ is a family that is closed under restriction to connected components, disjoint unions, and weak oddomorphisms, then $\F$ is homomorphism distinguishing closed. This allows us to prove many families are h.d.-closed.
All of our techniques are elementary, only making use of linear algebra over $\mathbb{Z}_2$ and combinatorial arguments. Our main tool is a construction which takes a connected graph $G$ and produces two graphs $G_0$ and $G_1$. We show that the set of homomorphisms from any graph $F$ to $G_i$ can be partitioned in such a way that the homomorphisms in each part are in one-to-one correspondence with solutions to a system of linear equations over $\mathbb{Z}_2$. This partition is common to $G_0$ and $G_1$ and the coefficient matrices of the linear systems for each part are the same for $G_0$ and $G_1$. However, the systems for $G_0$ are always homogeneous whereas those for $G_1$ are not. Thus $\hom(F,G_0) \ge \hom(F,G_1)$ for any graph $G$ with equality if and only if every system for $G_1$ has a solution. By applying a certain type of duality to these linear systems we can obtain an algebraic certificate for when a given system for $G_1$ does not have a solution. Interpreting this combinatorially, we obtain a graph theoretic characterization of graphs $F$ satisfying $\hom(F,G_0) > \hom(F,G_1)$. This characterization is in terms of a certain type of homomorphism from $F$ to $G$ that we term an \emph{oddomorphism}. Specifically, $\hom(F,G_0) > \hom(F,G_1)$ if and only if $F$ has a \emph{weak} oddomorphism to $G$, i.e., a homomorphism which is an oddomorphism when restricted to some subgraph of $F$. In Theorem~\ref{thm:homclosed1} we prove that if $\F$ is a family that is closed under restriction to connected components, disjoint unions, and weak oddomorphisms, then $\F$ is homomorphism distinguishing closed. Using this, we are able to show that many families are homomorphism distinguishing closed by proving necessary conditions for a graph $F$ to have a weak oddomorphism to particular graph $G$. This is how we obtain many of our results. For example, it is immediate from the definition that if $F$ has a weak oddomorphism to $G$, then the maximum degree of $F$ is at least that of $G$, and this implies that the family of graphs of some bounded degree is homomorphism distinguishing closed. In Theorem~\ref{thm:chordlesscycles} we prove that if $F$ has a weak oddomorphism to a cycle of length $k$, then $F$ must contain a chordless cycle of length at least (and of the same parity as) $k$, but this is much more difficult. Additionally, in Section~\ref{sec:quantum} we use the theory of quantum isomorphisms to indirectly show that weak oddomorphisms preserve planarity, i.e., if $F$ is planar and has a weak oddomorphism to $G$ then $G$ is planar.

\subsection{Preliminaries}

Outside of Section~\ref{sec:loops} all of our graphs are finite and simple, i.e., no multiple edges nor loops. We use $V(G)$ and $E(G)$ to denote the vertex and edge sets of a graph $G$. We write $u \sim_G v$ to denote that vertices $u$ and $v$ are adjacent in $G$, and we will often drop the subscript if it is clear from context. A \emph{homomorphism} from $F$ to $G$ is an adjacency-preserving function $\varphi$ from $V(F)$ to $V(G)$, i.e., $u \sim_F v \Rightarrow \varphi(u) \sim_G \varphi(v)$. We write $F \to G$ to denote that there exists a homomorphism from $F$ to $G$. We use $G \cup H$ to denote the \emph{disjoint union} of the graphs $G$ and $H$, which is the graph consisting of a copy of $G$ and a disjoint copy of $H$. We will also use $nG$ to refer to the disjoint union of $n$ copies of $G$. We denote by $K_n$ the complete graph on $n$ vertices. We say that a sequence of vertices $v_0, \ldots, v_\ell$ is a walk (of length $\ell$) in a graph $G$ if $v_{i-1} \sim_G v_i$ for all $i = 1, \ldots, \ell$.

Note that we are constantly referring to homomorphism indistinguishability over ``families of graphs $\F$". To be precise we should really consider sets of isomorphism classes of graphs. However, we prefer to use the less verbose phrasing at the cost of being imprecise. In any case it should not cause any confusion at any point.

\section{Conjectures}\label{sec:conjectures}

As mentioned above, our main conjecture is Conjecture~\ref{conj:distinct}, which states that distinct minor- and union-closed families of graphs give rise to distinct homomorphism indistinguishability relations. Here we will show that this is equivalent to the three other conjectures presented in Section~\ref{sec:contributions}. In fact, we will prove something more general: if the property of being minor-closed in the conjectures is replaced by some other property $\P$ of graph families, then Conjectures (1)--(4) are equivalent under mild assumptions on the property $\P$. This will imply that the conjectures are equivalent for minor-closed families. First we need the following lemmas. Note that we say that a family $\F$ is \emph{closed under restriction to connected components} if $G \in \F$ implies that every connected component of $G$ is a member of $\F$.

\begin{lemma}\label{lem:takeunion}
Let $\F$ be a family of graphs closed under restriction to connected components. Then
\[H_1 \cong_{\F} H_1' \ \& \ H_2 \cong_{\F} H'_2 \ \ \Rightarrow \ \ \left(H_1 \cup H_2\right) \cong_{\F} \left(H'_1 \cup H'_2\right).\]
\end{lemma}
\proof
Since $\F$ is closed under restriction to connected components, any graph in $\F$ is a disjoint union of connected graphs in $\F$. It therefore suffices to prove that
\[\hom(F,H_1 \cup H_2) = \hom(F,H'_1 \cup H'_2)\]
for all \emph{connected} graphs $F \in \F$. So let $F \in \F$ be connected. Note that connectedness of $F$ implies that $\hom(F,G_1 \cup G_2) = \hom(F,G_1)+\hom(F,G_2)$ for any graphs $G_1$ and $G_2$. Thus we have that
\[\hom(F,H_1 \cup H_2) = \hom(F,H_1) + \hom(F,H_2) = \hom(F,H'_1) + \hom(F,H'_1) = \hom(F,H'_1 \cup H'_2).\]\qeds

\begin{lemma}\label{lem:findconnected}
If $\F$ is a family of graphs closed under restriction to connected components, then $\mathrm{cl}(\F)$ is closed under restriction to connected components.%If $G \in \mathrm{cl}(\F)$, then every connected component of $G$ is contained in $\mathrm{cl}(\F)$.
\end{lemma}
\proof
Suppose that $G \in \mathrm{cl}(\F)$ and let $G_1, \ldots, G_c$ be the connected components of $G$. Suppose that $G_1 \notin \mathrm{cl}(\F)$, i.e., that there exist graphs $H$ and $H'$ such that $H \cong_\F H'$ but $\hom(G_1,H) \ne \hom(G_1,H')$. We first consider the case where $\hom(G_i,H)$ and $\hom(G_i,H')$ are both strictly positive for all $i = 1, \ldots c$. For convenience let us set $a_i = \hom(G_i,H)$ and $b_i = \hom(G_i,H')$ for each $i = 1, \ldots, c$. By assumption we have that $a_1 \ne b_1$.

Since $H \cong_\F H'$ by assumption, and we trivially have $H \cong_\F H$ and $H' \cong_\F H'$, using Lemma~\ref{lem:takeunion} we can obtain that $xH \cup (k-x)H' \cong_\F kH'$ for any $k,x \in \mathbb{N}$ such that $x \le k$. Therefore, $\hom(G,xH \cup (k-x)H') = \hom(G,kH')$ and thus
\begin{equation}\label{eq:lincomb}
\prod_{i=1}^c \left(xa_i + (k-x)b_i\right) = \prod_{i=1}^c kb_i > 0
\end{equation}
for any $k,x \in \mathbb{N}$ such that $x \le k$. Notably, the righthand side of the equation above does not depend on $x$. For $k \in \mathbb{N}$, let $p_k(x)$ denote the expression on the lefthand side of Equation~\eqref{eq:lincomb}, viewed as a polynomial in $x$. Also let $r_k$ be the constant (for fixed $k$) on the righthand side. Remember that $r_k >0$. It is clear that the polynomial $p_k(x)$ is of degree at most $c$ for any $k$. However, our argument so far shows that $p_k(x) = r_k$ for all of the $k+1$ values $x = 0, \ldots, k$. Thus for $k = c$, the polynomial $p_c(x)$ is equal to $r_c$ at $c+1$ distinct values. Since $p_c(x)$ has degree at most $c$, this is only possible if $p_c(x) = r_c$ for all $x \in \mathbb{R}$. Now since $a_1 \ne b_1$, the factor $(xa_1 + (c-x)b_1) = (a_1-b_1)x +b_1c$ of $p_c(x)$ is equal to 0 for some $x^* \in \mathbb{R}$. But then we obtain $0 < r_c = p_c(x^*) = 0$, a contradiction. Thus we have shown that if $H \cong_\F H'$ and $\hom(G_i,H)$ and $\hom(G_i,H')$ are strictly positive for all $i$, then $\hom(G_1,H) = \hom(G_1,H')$.

% Since $0 < a_1 < b_1$, we have that $b_1/(b_1 - a_1)$ is a positive rational number. Thus we can choose $k \in \mathbb{N}$ such that $kb_1/(b_1-a_1)$ is a positive integer. If $x = kb_1/(b_1-a_1)$, then $xa_1 + (k-x)b_1 = 0$. Therefore plugging this value of $x$ into Equation~\eqref{eq:lincomb} we obtain that
% \[0 = \prod_{i=1}^c kb_i > 0,\]
% a contradiction. Thus we have shown that if $H \cong_\F H'$ and $\hom(G_i,H)$ and $\hom(G_i,H')$ are strictly positive for all $i$, then $\hom(G_1,H) = \hom(G_1,H')$.

Now let $H$ and $H'$ be arbitrary graphs such that $H \cong_\F H'$. As before let $a_i = \hom(G_i,H)$ and $b_i = \hom(G_i,H')$ for all $i$. Further, let $K$ be any graph such that $\hom(G_i,K) > 0$ for all $i$, e.g., we can take $K$ to be some very large complete graph. Let $d_i = \hom(G_i,K)$. By Lemma~\ref{lem:takeunion} we have that $H \cup K \cong_\F H' \cup K$. Moreover, $\hom(G_i,H \cup K) = \hom(G_i,H) + \hom(G_i,K) > 0$ and $\hom(G_i,H' \cup K) = \hom(G_i,H') + \hom(G_i,K) > 0$ for all $i$. Thus we are now in our previous case and so we obtain that
\[\hom(G_1,H) + \hom(G_1,K) = \hom(G_1,H \cup K) = \hom(G_1,H' \cup K) = \hom(G_1,H') + \hom(G_1,K).\]
Therefore $\hom(G_1,H) = \hom(G_1,H')$ as desired. As the choice of $1$ was arbitrary we have that for all $i = 1, \ldots, c$, if $H \cong_\F H'$ then $\hom(G_i,H) = \hom(G_i,H')$, i.e., if $G \in \mathrm{cl}(\F)$ then every connected component of $G$ is in $\mathrm{cl}(\F)$.\qeds

We remark that one can ask whether the converse of Lemma~\ref{lem:takeunion} holds for h.d.-closed families, i.e., if $\F$ is h.d.-closed such that $H_1 \cong_{\F} H_1'$ and $H_2 \cong_{\F} H'_2$ implies that $(H_1 \cup H_2) \cong_{\F} (H'_1 \cup H'_2)$, does this imply that $\F$ is closed under restrictions to connected components.

We can now show that Conjectures (1)--(4) are equivalent even when ``minor-closed" is replaced with a property $\P$ of graph families that satisfies certain conditions. For a given property $\P$ of graph families, we will refer to a family with the property as a $\P$-family.

\begin{theorem}\label{thm:absconjectures}
Let $\P$ be a property of families of graphs such that (a) the disjoint union closure of $\F_1 \cup \F_2$ is a $\P$-family whenever $\F_1$ and $\F_2$ are $\P$-families, (b) any $\P$-family is closed under restriction to connected components, (c) for any connected graph $G$ there is a unique maximal $\P$-family $\F_G$ not containing $G$, and (d) $\F_G \cup \{G\}$ is a $\P$-family. Then the following statements are equivalent:
\begin{enumerate}
    \item If $\F_1$ and $\F_2$ are distinct union-closed $\P$-families, then the relations $\cong_{\F_1}$ and $\cong_{\F_2}$ are distinct.
    \item If $\F_1$ and $\F_2$ are union-closed $\P$-families, then $\cong_{\F_1} \ \Rightarrow \ \cong_{\F_2}$ if and only if $\F_1 \supseteq \F_2$.
    \item For any connected graph $G$, there exist graphs $H$ and $H'$ such that $\hom(G,H) \ne \hom(G,H')$, and any $\P$-family that distinguishes $H$ from $H'$ must contain $G$.
    \item Every union-closed $\P$-family is homomorphism distinguishing closed.
\end{enumerate}
\end{theorem}
\proof
We will show that $(\ref{conj:distinct}) \Rightarrow (\ref{conj:notimply}) \Rightarrow (\ref{conj:graphs}) \Rightarrow (\ref{conj:hdclosed}) \Rightarrow (\ref{conj:distinct})$.\\

\noindent$\mathbf{(\ref{conj:distinct}) \Rightarrow (\ref{conj:notimply})}$\textbf{:} First note that that to prove (2) it is necessary and sufficient to show that $\cong_{\F_1} \ \not\Rightarrow \ \cong_{\F_2}$ for union-closed $\P$-families $\F_1$ and $\F_2$ such that $\F_1 \not\supseteq \F_2$. So suppose that $\F_1$ and $\F_2$ are two such classes. Note that this implies that $\F_1 \cup \F_2$ strictly contains $\F_1$. Let $\F$ be the disjoint union closure of $\F_1 \cup \F_2$, which is union-closed by construction and a $\P$-family by (a). It then follows by our assumption of (1) that $\cong_\F$ is a distinct relation from $\cong_{\F_1}$, since $\F_1$ is strictly contained in $\F_1 \cup \F_2$. Now assume that $\cong_{\F_1} \ \Rightarrow \ \cong_{\F_2}$ for contradiction. Then we would have $\cong_{\F_1} \ \Rightarrow \ \cong_{\F_1 \cup \F_2}$, and thus $\cong_{\F_1} \ \Rightarrow \ \cong_{\F}$ since closing under disjoint unions does not change homomorphism indistinguishability. However, we also have that $\F_1 \subseteq \F$ by definition and thus $\cong_{\F} \ \Rightarrow \ \cong_{\F_1}$. Therefore $\cong_{\F_1}$ and $\cong_{\F}$ are the same relation, a contradiction.\\

\noindent$\mathbf{(\ref{conj:notimply}) \Rightarrow (\ref{conj:graphs})}$\textbf{:} Let $G$ be a connected graph and let $\F_G$ be the unique maximal $\P$-family not containing $G$ which exists by (c). By (d) we also have that $\F_G \cup \{G\}$ is a $\P$-family.

Now let $\F$ be the closure of $\F_G \cup \{G\}$ under disjoint unions. Then $\F$ is disjoint union closed by construction and a $\P$-family by (a), where we take $\F_1 = \F_2 = \F_G \cup \{G\}$. Moreover, we have that $\F_G \not\supseteq \F$ as the former does not contain $G$. Thus by our assumption of (2), we have that $\cong_{\F_G} \ \not\Rightarrow \ \cong_{\F}$. In other words, there exist graphs $H$ and $H'$ such that $H \cong_{\F_G} H'$ but $H \not\cong_{\F} H'$. Now suppose that $\F'$ is any $\P$-family that does not contain $G$. Since $\F_G$ is the unique maximal $\P$-family not containing $G$, we have that $\F_G \supseteq \F'$ and thus $\cong_{\F_G} \Rightarrow \cong_{\F'}$. Therefore $H \cong_{\F'} H'$ by the above, and so we have shown that any $\P$-family that distinguishes $H$ and $H'$ must contain $G$. Lastly, if $\hom(G,H) = \hom(G,H')$, then we would have that $H \cong_{\F_G \cup \{G\}} H'$, but this would imply that $H \cong_\F H'$, a contradiction. Therefore we have that $\hom(G,H) \ne \hom(G,H')$ as desired.\\

%By the former condition, if $\F'$ is any $\P$-family that does not contain $G$The former condition simply says that if $\hom(F,H) \ne \hom(F,H')$, then $F \notin \F_G$. Since $\F_G$ is the unique maximal $\P$-family not containing $G$, the condition $\F \notin \F_G$ is equivalent to the statement that any $\P$-family that contains $F$ must contain $G$. Therefore we have shown that if $\hom(F,H) \ne \hom(F,H')$, then any $\P$-family containing $F$ must contain $G$, which is the second condition of the conclusion of (3). Finally, if $\hom(G,H) = \hom(G,H')$, then we would have that $H \cong_{\F_G \cup \{G\}} H'$, but this would imply that $H \cong_\F H'$, a contradiction. Therefore we have that $\hom(G,H) \ne \hom(G,H')$ as desired.\\

\noindent$\mathbf{(\ref{conj:graphs}) \Rightarrow (\ref{conj:hdclosed})}$\textbf{:} Suppose that $\F$ is a union-closed $\P$-family that is not homomorphism distinguishing closed. Then there exists a graph $G' \notin \F$ such that $G' \in \mathrm{cl}(\F)$. If every connected component of $G'$ were contained in $\F$, then $G'$ would also be contained in $\F$ since it is union-closed. Thus there is a connected component $G$ of $G'$ such that $G \notin \F$, but by Lemma~\ref{lem:findconnected} we have that $G \in \mathrm{cl}(\F)$. Since $G$ is connected, we can apply our assumption of (3) to the graph $G$. We obtain that there exist graphs $H$ and $H'$ such that $\hom(G,H) \ne \hom(G,H')$, and any $\P$-family that distinguishes $H$ and $H'$ must contain $G$. Since $G \in \mathrm{cl}(\F)$, the former condition implies that $H \not\cong_\F H'$, i.e., $\F$ distinguishes $H$ and $H'$. But then we must have $G \in \F$, a contradiction.\\

%In the above the only property of minor-closed families that we used was Lemma~\ref{lem:takeunion}, and the proof of this only uses that any minor-closed family is closed under restrictions to connected components, i.e., any connected component of a graph in the family is in the family. Obviously this holds for many other types of families of graphs.\\

\noindent$\mathbf{(\ref{conj:hdclosed}) \Rightarrow (\ref{conj:distinct})}$\textbf{:} This is trivial.\qeds

We now prove that Conjectures (1)--(4) as stated in Section~\ref{sec:contributions} are equivalent. Note that statement (3) in the above theorem is phrased slightly differently than Conjecture~\ref{conj:graphs}, but the latter is clearly equivalent to the former in the case where the property $\P$ is minor-closedness.

\begin{theorem}\label{thm:conjectures}
The following statements are equivalent:
\begin{enumerate}
    \item If $\F_1$ and $\F_2$ are distinct minor- and union-closed families of graphs, then the relations $\cong_{\F_1}$ and $\cong_{\F_2}$ are distinct.
    \item If $\F_1$ and $\F_2$ are minor- and union-closed families of graphs, then $\cong_{\F_1} \ \Rightarrow \ \cong_{\F_2}$ if and only if $\F_1 \supseteq \F_2$.%If $\F_1$ and $\F_2$ are minor- and union-closed families of graphs such that $\F_1 \not\subseteq \F_2$, then $\cong_{\F_2} \ \not\Rightarrow \ \cong_{\F_1}$.
    \item For any connected graph $G$, there exist graphs $H$ and $H'$ such that $\hom(G,H) \ne \hom(G,H')$, and $\hom(F,H) \ne \hom(F,H')$ implies that $F$ contains $G$ as a minor.
    \item Every minor- and union-closed class of graphs is homomorphism distinguishing closed.
\end{enumerate}
In other words, Conjectures (1)--(4) are equivalent.
\end{theorem}
\proof
We only need to show that conditions (a)--(d) from Theorem~\ref{thm:absconjectures} hold for minor-closed families, and this is straightforward.\qeds

\section{Our Construction}\label{sec:construction}

Here we present our construction of graphs $G_0$ and $G_1$ from a connected graph $G$ mentioned in Section~\ref{sec:techniques}. The construction is in fact almost the same as the construction used in~\cite{qiso1}\footnote{Which itself is very similar to the FGLSS reduction from the theory of hardness of approximation~\cite{FGLSS}.} to produce pairs of non-isomorphic yet quantum isomorphic graphs. In fact the only difference is our choice of edges. This construction takes a linear system $Mx=b$ and produces a graph $G(M,b)$. One can then show that $Mx=b$ has a solution if and only if $G(M,b)$ and $G(M,0)$ are isomorphic, and in~\cite{qiso1} they prove a quantum analog of this. An important special case of this construction is when $M$ is the incidence matrix of a connected graph $G$. It was shown by Arkhipov~\cite{arkhipov} that in this case the system $Mx=b$ for any odd weight $b$ never has a solution but has a \emph{quantum solution} if and only if the graph $G$ is not planar. It is this case, where $M$ is the incidence matrix of a graph, that we will focus on here. This allows us to describe the graphs $G(M,b)$ more directly in terms of the graph $G$, leaving the underlying connection to the linear system $Mx=b$ merely implicit rather than explicit. In fact, all of the analysis we do in this section can be done for general linear systems over $\mathbb{Z}_2$. However in that case the results are not quite as clean and we only need the special case for our results.

In the following definition and henceforth, we will use $E(v)$ to denote the set of edges incident to the vertex $v$. We will also use $\delta_{v,U}$ to denote $|\{v\} \cap U|$.
\begin{restatable}{definition}{construction}\label{def:construction}
Let $G$ be a graph and $U \subseteq V(G)$. Define $G_U$ as the graph with vertex set $\{(v,S) : v \in V(G), \ S \subseteq E(v), \ |S| \equiv \delta_{v,U} \ \mathrm{mod} \ 2\}$, where $(v,S)$ is adjacent to $(u,T)$ if $uv \in E(G)$ and $uv \notin S \triangle T$.
\end{restatable}

For a vertex $(v,S) \in G_U$, we will often refer to $v$ as its \emph{head} and $S$ as its \emph{tail}. We can think of the condition $uv \notin S \triangle T$ as the sets $S$ and $T$ ``agreeing" or being ``being consistent" on the edge $uv$. The difference between this construction and the construction used in ~\cite{qiso1} is that there they put edges between ``inconsistent" vertices. Specifically, they made vertices $(v,S)$ and $(u,T)$ adjacent if $u \sim v$ and $uv \in S \triangle T$, or if $u = v$ and $S \ne T$. The graph presented in the definition above is almost the complement of the graph of~\cite{qiso1}, except that we do not put edges between $(v,S)$ and $(u,T)$ when $u$ and $v$ are nonadjacent in $G$.

The first thing we do is show that the isomorphism class of $G_U$ only depends on the parity of $|U|$. This is essentially known from combining results of~\cite{arkhipov} and~\cite{qiso1}, but we give a direct proof.

\begin{lemma}\label{lem:onlytwo}
Let $G$ be a connected graph. If $U,U' \subseteq V(G)$ are such that $|U| \equiv |U'|\mod 2$, then $G_U \cong G_{U'}$.
\end{lemma}
\begin{proof}
Let $uv \in E(G)$. We will first show that $G_{U} \cong G_{U'}$ for $U' = U \triangle \{u,v\}$. Define $\varphi \colon V(G_U) \to V(G_{U'})$ as
\begin{equation*}
\varphi(w,S) = \begin{cases}
(w,S \triangle \{uv\}) & \text{if } w = u \text{ or } v \\
(w,S) & \text{otherwise}
\end{cases}
\end{equation*}
It is clear that this is a bijection. Suppose that $(x,S)$ and $(y,T)$ are adjacent in $G_{U}$, i.e., that $xy \in E(G)$ and $xy \notin S \triangle T$. Let $S'$ and $T'$ be such that $\varphi(x,S) = (x,S')$ and $\varphi(y,T) = (y,T')$. If $xy = uv$ (as an edge, not ordered pair), then $S' \triangle T' = (S \triangle \{uv\}) \triangle (T \triangle \{uv\}) = S \triangle T \not\ni uv = xy$, and thus $\varphi(x,S)$ and $\varphi(y,T)$ are adjacent in $G_{U'}$. If $xy \ne uv$, then $S'$ (respectively $T'$) contains $xy$ if and only if $S$ (respectively $T$) does. Thus $xy \notin S' \triangle T'$ and so $\varphi(x,S)$ is adjacent to $\varphi(y,T)$ in $G_{U'}$. So we have shown that $\varphi$ preserves adjacency. The proof that it preserves non-adjacency is similar, and so we have that $\varphi$ is an isomorphism from $G_U$ to $G_{U'}$.

We have shown that $G_U \cong G_{U \triangle \{u,v\}}$ for any edge $uv \in E(G)$. Since $G$ is connected, by applying this isomorphism iteratively, we have that $G_U \cong G_{U \triangle \{u,v\}}$ for any two vertices $u,v \in V(G)$. Finally, iteratively applying this we obtain that $G_U \cong G_{U'}$ for any $U,U'$ such that $|U \triangle U'|$ is even, i.e., any $U,U'$ with $|U| \equiv |U'| \mod 2$.
\end{proof}

The above lemma shows that there are, up to isomorphism, at most two graphs of the form $G_U$ for a connected graph $G$. We will therefore use $G_0$ to denote $G_\varnothing$ and $G_1$ to denote $G_{\{u\}}$ for some vertex $u$ which we may specify if needed. We will see later that these two graphs are indeed non-isomorphic.

\begin{example}\label{ex:cycles}
Let us consider the graphs $G_0$ and $G_1$ for $G = C_k$, the cycle of length $k$. Let us take $\{0, \ldots, k-1\}$ to be the vertex set such that $i \sim i+1$ (arguments taken modulo $k$). In $G_0$, for each $i \in V(G)$ we will have two vertices $(i, \varnothing)$ and $(i,\{\{i-1,i\},\{i,i+1\}\})$. It is easy to see that the former type of vertices form a $k$-cycle as do the latter, and there are no edges between the different types. Thus $G_0$ is just two disjoint $k$-cycles, i.e., $G_0 \cong 2C_k$. Now consider $G_{\{0\}}$. For each $i \in V(G) \setminus \{0\}$ we have the same two vertices as in $G_0$, but for $i=0$ we have $(0,\{\{0,1\}\})$ and $(0,\{0,k-1\}\})$. The former vertex is adjacent to $(1,\{\{0,1\},\{1,2\}\})$ and $(k-1,\varnothing)$, whereas the latter is adjacent to $(1,\varnothing)$ and $(k-1,\{\{0,k-1\},\{k-2,k-1\}\})$. It is not too difficult to see that this then forms a cycle of length $2k$, thus $G_1 \cong C_{2k}$.
\end{example}

\subsection{Counting homomorphisms to $G_i$}

We now turn to counting homomorphisms from some graph $F$ to the graphs $G_0$ and $G_1$. The key to our work is to partition the set of such homomorphisms in a particular way. To do this we use the fact that the map $\rho$ which takes $(v,S)$ to $v$ is a homomorphism from $G_i$ to $G$ for both $i=0$ and $i=1$. We sometimes refer to the map $\rho$ as the \emph{projection} to $G$. We will denote the \emph{set} of all homomorphisms from a graph $F$ to a graph $G$ as $\Hom(F,G)$. Since $\rho$ is a homomorphism from $G_i$ to $G$, if $\varphi$ is any homomorphism from a graph $F$ to $G_i$, then the composition $\rho \circ \varphi$ is a homomorphism from $F$ to $G$. We can thus partition the homomorphisms $\varphi \in \Hom(F,G_i)$ according to this composition. Formally, we partition $\Hom(F,G_i)$ into the sets $\Hom_\psi(F,G_i)$ for $\psi \in \Hom(F,G)$ defined as
\[\Hom_\psi(F,G_i) = \{\varphi \in \Hom(F,G_i) : \rho \circ \varphi = \psi\}.\]
The sets $\Hom_\psi(F,G_i)$ for $\psi \in \Hom(F,G)$ clearly partition $\Hom(F,G_i)$. Moreover, it allows us to partition $\Hom(F,G_0)$ and $\Hom(F,G_1)$ in the same way and compare the sizes of the parts $\Hom_\psi(F,G_0)$ and $\Hom_\psi(F,G_1)$ for a given $\psi \in \Hom(F,G)$. This will be crucial in our arguments later on. As with $\Hom$ and $\hom$, we will use $\hom_\psi(F,G_i)$ to denote $|\Hom_\psi(F,G_i)|$.

We now will show that for a given $\psi \in \Hom(F,G)$ the sets $\Hom_\psi(F,G_i)$ are in bijection with the solutions of a certain linear system.

\begin{lemma}\label{lem:homs2sols}
Let $G$ be a connected graph, let $U \subseteq V(G)$, and let $F$ be any graph. For a given $\psi \in \Hom(F,G)$ we define variables $x_e^a$ for all $a \in V(F)$ and $e \in E(\psi(a))$. Then the elements of $\Hom_\psi(F,G_U)$ are in bijection with solutions to the following equations over $\mathbb{Z}_2$:
\begin{align}
    \sum_{e \in E(\psi(a))} x_e^a &= \delta_{\psi(a),U} \text{ for all } a \in V(F);\label{eq:parity}\\
    x_e^a + x_e^b &= 0 \text{ for all } ab \in E(F), \text{ where } e = \psi(a)\psi(b) \in E(G).\label{eq:adjacency}
\end{align}
\end{lemma}
\proof
The idea is that for any potential $\varphi \in \Hom_\psi(F,G_U)$, the heads of the vertices $\varphi(a)$ are determined (they are equal to $\psi(a)$), and thus it is only left to choose the tails. Then $x^a_e$ is a variable that indicates whether $e \in E(\psi(a))$ is contained in the tail of $\varphi(a)$. Equation~\eqref{eq:parity} ensures that the tails have the correct parity to be vertices of $G_U$ and Equation~\eqref{eq:adjacency} ensures that $\varphi$ preserves adjacency.

Suppose that $x = (x^a_e)$ is a solution to \Eq{main}. We define a homomorphism $\varphi_x \in \Hom_\psi(F,G_U)$ as
\[\varphi_x(a) = (\psi(a),S_x(a)) \text{ where } S_x(a) := \{e \in E(\psi(a)) : x^a_e = 1\}.\]
We must first prove that $\varphi_x$ is actually an element of $\Hom_\psi(F,G_U)$. Note that $\varphi_x(a)$ is indeed a vertex of $G_U$ since Equation~\eqref{eq:parity} ensures that
\[|\{e \in E(\psi(a)) : x^a_e = 1\}| \equiv \delta_{\psi(a),U} \ \mathrm{mod} \ 2.\]
Next we show that $\varphi_x$ preserves adjacency. Suppose that $ab \in E(F)$. Since $\psi$ is a homomorphism from $F$ to $G$, we have that $\psi(a) \sim \psi(b)$ and these are the heads of $\varphi(a)$ and $\varphi(b)$. Now let $e = \psi(a)\psi(b)$. We must show that $e \notin S_x(a) \triangle S_x(b)$. By Equation~\eqref{eq:adjacency}, we have that $x_e^a = x_e^b$ and therefore either the edge $e$ is either contained in both $S_x(a)$ and $S_x(b)$ or neither. It follows that $(\psi(a),S_x(a)) \sim_{G_U} (\psi(b),S_x(b))$. Thus $\varphi$ is a homomorphism from $F$ to $G_U$ and $\varphi \in \Hom_\psi(F,G_U)$ by construction. It is easy to see that this construction is injective, as if $x \ne y$, then there is some $a \in V(F)$ and $e \in E_G(\psi(a))$ such that $x_e^a \ne y_e^a$ and therefore these two solutions will give rise to different $S_x(a)$ and thus different $\varphi$.

For surjectivity, suppose that $\varphi \in \Hom_\psi(F,G_U)$. For each $a \in V(F)$, let $S_a$ be such that $\varphi(a) = (\psi(a),S_a)$. Recall that this means that $S_a \subseteq E_G(\psi(a))$. Define $x_e^a = 1$ if $e \in S_a$ and $x_e^a = 0$ otherwise. This defines all $x_e^a$ and this satisfies all equations of~\eqref{eq:parity} by construction. Consider now the Equation~\eqref{eq:adjacency} for a particular $ab \in E(F)$. This requires that $x_{\psi(a)\psi(b)}^a = x_{\psi(a)\psi(b)}^b$. But this is ensured by the fact that $\varphi$ is a homomorphism from $\Hom_\psi(F,G_U)$ and thus $\psi(a)\psi(b)$ is either contained or not contained in both of $S_a$ and $S_b$. Thus the $x$ we have constructed is indeed a solution to the equations in~\eqref{eq:parity} and~\eqref{eq:adjacency}, and it is clear that this construction was the inverse of the construction of $\varphi_x$ from $x$ presented above.\qeds

Note that only the righthand side of the system of equations in~\eqref{eq:parity} and ~\eqref{eq:adjacency} depends on the subset $U \subseteq V(G)$, and the lefthand side depends only on $G$, $F$, and $\psi$. The following definition describes the coefficient matrices of the systems of equations in~\eqref{eq:parity} and ~\eqref{eq:adjacency}.

\begin{definition}\label{def:coeffmats}
Given a connected graph $G$, a graph $F$, and a homomorphism $\psi \in \Hom(F,G)$, let $R$ be the set of ordered pairs $(a, e)$ such that $a \in V(F)$ and $e \in E(\psi(a))$. Define the matrices $A^\psi \in \mathbb{Z}_2^{V(F) \times R}$ and $B^\psi \in \mathbb{Z}_2^{E(F) \times R}$ as follows\footnote{These matrices do depend on $G$ and $F$ as well, but we leave them out of the notation since it will always be clear from context which $G$ and $F$ are under consideration.}:
\begin{align}
    A^\psi_{b, (a,e)} &= \begin{cases} 1 & \text{if } b = a \\ 0 & \text{o.w.}\end{cases}\\
    B^\psi_{bc,(a,e)} &= \begin{cases}1 & \text{if } a \in \{b,c\} \ \& \ e = \psi(b)\psi(c)\\ 0 &\text{o.w.}\end{cases}
\end{align}
Additionally, for $U \subseteq V(G)$, we let $\chi_{\psi^{-1}(U)}$ denote the characteristic vector of the set $\psi^{-1}(U)$.
\end{definition}
Using the above notation, the system of equations given Lemma~\ref{lem:homs2sols} can be written as
\begin{equation}\label{eq:main}
    \begin{pmatrix}A^\psi \\ B^\psi\end{pmatrix}x = \begin{pmatrix} \chi_{\psi^{-1}(U)} \\ 0\end{pmatrix}
\end{equation}
Note that if $U = \varnothing$, then $\chi_{\psi^{-1}(U)} = 0$ and thus this system always has a solution, namely $x = 0$. The next corollary uses the fact that the number of solutions to a system of linear equations is either zero, or is equal to the number solutions to the homogenous version of the system, i.e. the same system but with the righthand side changed to the zero vector.

\begin{theorem}\label{thm:main}
Let $G$ be a connected graph, $U \subseteq V(G)$, and let $\psi \in \Hom(F,G)$ for some graph $F$. Then $\hom_\psi(F,G_0) > 0$ and 
\[\hom_\psi(F,G_U) = \begin{cases} \hom_\psi(F,G_0) & \text{if Equation~\eqref{eq:main} has a solution}\\0 & \text{if Equation~\eqref{eq:main} has no solution} \end{cases}\]
It follows that $\hom(F,G_U) \le \hom(F,G_0)$ with equality if and only if Equation~\eqref{eq:main} has a solution for all $\psi \in \Hom(F,G)$.
%\begin{equation*}
%    \begin{pmatrix}A^\psi \\ B^\psi\end{pmatrix}y = %\begin{pmatrix} b^\psi \\ 0\end{pmatrix}
%\end{equation*}
%has a solution for all $\psi \in \Hom(F,G_M)$.
\end{theorem}
\proof
By Lemma~\ref{lem:homs2sols}, the set $\Hom_\psi(F,G(M,b))$ is in bijection with the solutions to Equation~\eqref{eq:main}. As noted above, the number of such solutions is either 0 or equal to the number of solutions to the homogeneous version
\begin{equation*}
    \begin{pmatrix}A^\psi \\ B^\psi\end{pmatrix}x = \begin{pmatrix} 0 \\ 0\end{pmatrix}
\end{equation*}
which always has a strictly positive number of solutions and whose solutions are in bijection with $\Hom_\psi(F,G_0)$. Thus we have the claims in the second sentence of the theorem.
%for all $\psi \in \Hom(F,G_M)$, we have that $\hom_\psi(F,G(M,0)) > 0$ and
%\[\hom_\psi(F,G(M,b)) = \begin{cases} \hom_\psi(F,G(M,0)) & \text{if Equation~\eqref{eq:main} has a solution}\\0 & \text{if Equation~\eqref{eq:main} has no solution} \end{cases}\]
Now since the sets $\Hom_\psi(F,G_U)$ partition $\Hom(F,G_U)$ and similarly for $G_0$, the theorem follows.\qeds

As a corollary we obtain the following which in particular implies the converse to Lemma~\ref{lem:onlytwo}, i.e., that $G_0 \not\cong G_1$.
\begin{cor}\label{cor:0iso}
Let $G$ be a connected graph and $U \subseteq V(G)$. Then the following are equivalent.
\begin{enumerate}
    \item $|U|$ is even;
    \item $G_0 \cong G_U$;
    \item $\hom(G,G_0) = \hom(G,G_U)$;
    \item $\hom_{\mathrm{id}}(G,G_0) = \hom_{\mathrm{id}}(G,G_U)$.
\end{enumerate}
\end{cor}
\proof
The implication $(1) \rightarrow (2)$ follows from Lemma~\ref{lem:onlytwo}. Moreover, $(2) \rightarrow (3)$ is immediate and $(3) \rightarrow (4)$ follows from Theorem~\ref{thm:main}. Thus we only need to prove $(4) \rightarrow (1)$.

Assume that $\hom_{\mathrm{id}}(G,G_0) = \hom_{\mathrm{id}}(G,G_U)$. By Theorem~\ref{thm:main}, this implies that the system in Equation~\eqref{eq:main} has a solution for $\psi = \mathrm{id}$. Let $x$ be such a solution. Note that since we are considering $F = G$, the vertices of $F$ are the vertices of $G$. Let us first show that $x^a_e$ does not depend on $a$. Consider $x^a_e$ and $x^b_e$. Note that this means that $e \in E(\psi(a)) = E(a)$ and $e \in E(\psi(b)) = E(b)$, and therefore $e = ab$ and so $a$ and $b$ are adjacent in $F = G$. Thus by Equation~\eqref{eq:adjacency} we must have $x^a_e + x^b_e = 0$, i.e., $x^a_e = x^b_e$. Therefore, we can unambiguously define a graph $H$ with vertex set $V(G)$ and edges $e \in E(G)$ such that $x^a_e = 1$ for any $a \in V(F)=V(G)$ such that $e \in E(a)$. Since $x$ satisfies Equation~\eqref{eq:main}, it follows that every vertex contained in $U$ has odd degree in $H$ and every other vertex has even degree. Therefore $|U|$ must be even.\qeds

We will now use Theorem~\ref{thm:main} to obtain a \emph{combinatorial certificate} for when a graph $F$ has a different number of homomorphisms to $G_0$ and $G_1$. By this we mean some combinatorial condition on the graph $F$ that holds if and only if $\hom(F,G_0) \ne \hom(F,G_1)$. To do this we will make use the following lemma, known as the Fredholm Alternative, that provides an algebraic certificate for when a system of equations has no solutions.

\begin{lemma}[Fredholm Alternative]
Let $M \in \mathbb{Z}_2^{m \times n}$ and $b \in \mathbb{Z}_2^m$. Then the system $Mx = b$ has no solution if and only if the system
\begin{equation}\label{eq:fredholm}
    \begin{pmatrix}M^T \\ b^T\end{pmatrix}y = \begin{pmatrix}0 \\ 1\end{pmatrix}
\end{equation}
has a solution.
\end{lemma}
Note that the above result, which holds for any vector space, is usually stated as $Mx=b$ has no solution if and only if $M^Ty = 0$ has a solution $y$ such that $b^Ty \ne 0$, but over $\mathbb{Z}_2$ the latter is equivalent to $b^Ty = 1$ and thus we obtain the formulation given above.

In order to concisely state our combinatorial condition for when $\hom(F,G_0) \ne \hom(F,G_1)$, we first define the following:

\begin{definition}\label{def:oddomorphism}
Let $F$ and $G$ be graphs and $\psi$ a homomorphism from $F$ to $G$. We say that a vertex $a$ of $F$ is \emph{odd/even with respect to $\psi$} if $|N_F(a) \cap \psi^{-1}(v)|$ is odd/even for all $v \sim_G \psi(a)$. We define an \emph{oddomorphism} from $F$ to $G$ to be a homomorphism $\psi \in \Hom(F,G)$ such that
\begin{enumerate}
    \item each vertex of $F$ is either odd or even with respect to $\psi$;
    \item $\psi^{-1}(v)$ contains an odd number of odd vertices for all $v \in V(G)$.
\end{enumerate}
A homomorphism $\psi \in \Hom(F,G)$ is a \emph{weak oddomorphism} if there is a (not necessarily induced) subgraph $F'$ of $F$ such that $\psi|_{V(F')}$ is an oddomorphism from $F'$ to $G$. We will also sometimes use ``$\psi$-odd/even" instead of ``odd/even with respect to $\psi$".
\end{definition}

Note that, assuming Condition (1) above, Condition (2) is equivalent to there being an odd number of edges between $\psi^{-1}(u)$ and $\psi^{-1}(v)$ for all $uv \in E(G)$.

\begin{remark}\label{rem:oddo}
One way to think of an oddomorphism $\psi$ from $F$ to $G$ is in terms of the partition of $V(F)$ into the fibres $\psi^{-1}(v)$ for $v \in V(G)$. Since $\psi$ is a homomorphism, there can only be edges between fibres $\psi^{-1}(u)$ and $\psi^{-1}(v)$ such that $u \sim_G v$. Thus if $a \in V(F)$ is such that $\psi(a) = u \in V(G)$, then $a$ can only have neighbors in the fibres $\psi^{-1}(v)$ where $v \sim_G u$. The property of $a$ being odd with respect to $\psi$ then says that $a$ has an odd (and therefore nonzero) number of neighbors in every fibre that it can possibly have neighbors in. Similarly, $a$ is even with respect to $\psi$ if it has an even (thus possibly zero) number of neighbors in every fibre it can possibly have neighbors in. If $a$ has an even number of neighbors in some fibres and an odd number in some other fibres, then $a$ would be neither odd nor even with respect to $\psi$, and thus Condition (1) from Definition~\ref{def:oddomorphism} is nontrivial. Note that Condition (2) ensures that each fibre contains at least one odd vertex.
\end{remark}
\begin{remark}\label{rem:oddoissue}
Note that if $\psi$ is a homomorphism from $F$ to $G$ and $F'$ is a subgraph of $F$, then $\psi|_{V(F')}$ is a homomorphism from $F'$ to $G$. However, even if $V(F') = V(F)$ in which case $\psi|_{V(F')} = \psi$, the set of odd/even vertices with respect to $\psi$ will depend on whether we are considering $\psi$ as a homomorphism from $F$ or $F'$. This issue will not arise often, but when it does we will specify that a vertex is odd/even with respect to $\psi|_F$ (or $\psi|_{F'}$) to be clear. One particular place where this can come up is if $\psi$ is a weak oddomorphism from $F$ to $G$ and $F'$ is a subgraph of $F$ such that $\psi|_{V(F')}$ is an oddomorphism from $F'$ to $G$. In this case a vertex $a \in V(F')$ that is odd/even w.r.t.~$\psi|_{F'}$ may have the opposite parity w.r.t.~$\psi|_F$, and $a$ may be neither even nor odd w.r.t.~$\psi|_F$ if $F'$ is a proper subgraph of $F$.
\end{remark}

Before proving our main result, we need the following lemma:

\begin{lemma}\label{lem:oddoconnected}
Let $G$ be a connected graph and suppose that $\psi$ is a weak oddomorphism from a graph $F$ to $G$. Then there is a connected subgraph $\hat{F}$ of $F$ such that $\psi|_{V(\hat{F})}$ is an oddomorphism from $\hat{F}$ to $G$.
\end{lemma}
\proof
Let $F'$ be the subgraph of $F$ such that $\psi|_{V(F')}$ is an oddomorphism from $F'$ to $G$. Pick any vertex $v \in V(G)$ and let $\hat{F}$ be a connected component of $F'$ that intersects $\psi^{-1}(v)$ in an odd number of $\psi|_{F'}$-odd vertices. Note that this must exist since $\psi^{-1}(v)$ contains an odd number of such vertices. Consider now a vertex $u \in V(G)$ such that $u \sim v$. Since $\psi^{-1}(v) \cap V(\hat{F})$ contains an odd number of odd vertices, there must be an odd number of edges of $\hat{F}$ between $\psi^{-1}(u)$ and $\psi^{-1}(v)$. By the reverse of this argument, $\hat{F}$ must intersect $\psi^{-1}(u)$ in an odd number of odd vertices. Since $G$ is connected we can continue this argument to show that $\hat{F}$ intersects $\psi^{-1}(w)$ in an odd number of odd vertices for all $w \in V(G)$. Thus $\psi|_{V(\hat{F})}$ is an oddomorphism from $\hat{F}$ (which is connected) to $G$.\qeds

\begin{theorem}\label{thm:maindualG}
Let $G$ be a connected graph. Then $\hom(F,G_0) \ge \hom(F,G_1)$ with strict inequality if and only if there exists a weak oddomorphism from $F$ to $G$. Moreover, if such a weak oddomorphism $\psi$ does exist, then there is a \emph{connected} subgraph $F'$ of $F$ such that $\psi|_{V(F')}$ is an oddomorphism from $F'$ to $G$.
\end{theorem}
\proof
Note that we can assume that $V(F') = V(F)$ since adding isolated vertices to $F'$ does not affect whether it has an oddomorphism to $G$ (later we will replace $F'$ with one of its connected components via Lemma~\ref{lem:oddoconnected}). Pick $\hat{u} \in V(G)$ and take $G_1 = G_{\{\hat{u}\}}$. Suppose that $\hom(F,G_0) \ne \hom(F,G_1)$. By Theorem~\ref{thm:main} this is equivalent to the existence of $\psi \in \Hom(F,G)$ such that the system 
\[\begin{pmatrix}A^\psi \\ B^\psi\end{pmatrix} x = \begin{pmatrix} \chi_{\psi^{-1}(\hat{u})} \\ 0\end{pmatrix}\]
over $\mathbb{Z}_2$ does not have a solution. By the Fredholm Alternative this is further equivalent to there being a solution to
\begin{equation}\label{eq:cert}
\begin{pmatrix}(A^\psi)^T & (B^\psi)^T \\ \chi_{\psi^{-1}(\hat{u})}^T & 0\end{pmatrix}\begin{pmatrix}y \\ z\end{pmatrix} = \begin{pmatrix}0 \\ 1\end{pmatrix},
\end{equation}
where $y$ is indexed by $a \in V(F)$ and $z$ is indexed by $e \in E(F)$. 

Suppose that we have a solution $(y,z)^T$ to Equation~\eqref{eq:cert}. This must satisfy $(A^\psi)^T y + (B^\psi)^Tz = 0$ which is equivalent to $(A^\psi)^T y = (B^\psi)^Tz$. Let $O = \{a \in V(F) : y_a = 1\}$ and let $F'$ be the graph on vertex set $V(F)$ with edge set $E(F') = \{e \in E(F) : z_e = 1\}$. We will show that the elements of $O$ are odd vertices w.r.t.~$\psi|_{F'}$, and the rest are even. For $a \in V(F)$ and $e \in E_G(\psi(a))$ the $(a,e)$-entry of $(A^{\psi})^Ty$ is equal to $y_a$. The $(a,e)$-entry of $(B^{\psi})^Tz$ is equal to
\begin{equation}\label{eq:idunno}
    \sum_{f \in E(a), \ \psi(f) = e} z_f,
\end{equation}
where we are using $\psi(f)$ to denote the edge $\psi(a)\psi(b)$ where $a$ and $b$ are the ends of $f$. Letting $v = \psi(a)$, since $e \in E_G(\psi(a))$ we have that $e = uv$ for some $u \sim v$ in $G$. Then the expression in~\eqref{eq:idunno} is equal to the number of edges $f \in E(F')$ such that $f = ab$ where $\psi(b) = u$. In other words it is equal to the number of neighbors of $a$ in $F'$ that are contained in $\psi^{-1}(u)$, i.e., $|N_{F'}(a) \cap \psi^{-1}(u)|$. Thus the condition $(A^\psi)^T y = (B^\psi)^Tz$ says that for all $a \in V(F')$ and $u \in N_G(\psi(a))$ we have that $|N_{F'}(a) \cap \psi^{-1}(u)|$ is odd if and only if $y_a = 1$, i.e., $a \in O$. Therefore every vertex of $F'$ is either even or odd with respect to $\psi|_{F'}$.

The last equation of \Eq{cert} says that $\psi^{-1}(\hat{u})$ must contain an odd number of odd vertices of $F'$. Let $v \sim_G \hat{u}$, and let us consider the parity of the number of edges going between $\psi^{-1}(\hat{u})$ and $\psi^{-1}(v)$ in $F'$. This parity is equal to the parity of the number of odd vertices in $\psi^{-1}(\hat{u})$, which is odd, and also to the parity of the number of odd vertices in $\psi^{-1}(v)$. Thus $\psi^{-1}(v)$ must also contain an odd number of odd vertices, and since $G$ is connected this holds for all $v \in V(G)$. So we have constructed the subgraph $F'$ of $F$ such that $\psi|_{V(F')}$ is an oddomorphism to $G$. Therefore $\psi$ is a weak oddomorphism from $F$ to $G$ as desired. The converse direction is a straightforward reversal of this argument so we omit it.

Lastly, in the case where $F$ does have a weak oddomorphism $\psi$ to $G$, by Lemma~\ref{lem:oddoconnected} we can always find a \emph{connected} subgraph $F'$ of $F$ such that $\psi_{V(F')}$ is an oddomorphism from $F'$ to $G$.\qeds

Note that Corollary~\ref{cor:0iso} says that $\hom(G,G_0) \ne \hom(G,G_1)$ and therefore Theorem~\ref{thm:maindualG} shows that $G$ must admit a weak oddomorphism to itself. Moreover, the following slightly stronger statement can be proven easily (and thus we omit the proof).

\begin{lemma}\label{lem:idoddo}
If $G$ is a connected graph, then the identity map on $V(G)$ is an oddomorphism from $G$ to itself. Therefore $\hom(G,G_0) \ne \hom(G,G_1)$.
\end{lemma}

We remark that the above lemma can be used to give an alternative proof that $G_0 \not\cong G_1$.

Before moving on, we note that in Section~\ref{sec:construction2} we will give a construction of graphs $\tilde{G}_0$ and $\tilde{G}_1$ that is very similar to $G_0$ and $G_1$, and we will perform a similar analysis on when a graph $F$ satisfies $\hom(F,\tilde{G}_0) \ne \hom(F,\tilde{G}_1)$. Since we have already done this in detail for $G_0$ and $G_1$ here, we will be much briefer in Section~\ref{sec:construction2}. Thus some readers might find it easier to skip ahead and read Section~\ref{sec:construction2} and then come back.

\subsection{Examples of Oddomorphisms}

Here we will see that certain subdivisions and covers of a graph $G$ always have oddomorphisms to $G$.

Recall that a graph $F$ is a subdivision of $G$ if $F$ is obtained from $G$ by replacing each of its edges with a path of length at least one, such that these paths are internally disjoint from each other. An \emph{odd subdivision} is a subdivision in which each of these paths have odd length. Note that this means that $G$ is an odd subdivision of itself. We will refer to the vertices of a subdivision of $G$ that correspond to the original vertices of $G$ as the \emph{main vertices} of the subdivision. We will show that any odd subdivision $F$ of a graph $G$ has an oddomorphism to $G$. The idea is simple: the paths between two main vertices of $F$ are ``folded up" onto the edge between the two corresponding vertices of $G$.

\begin{lemma}\label{lem:subdivisions}
Let $F$ be an odd subdivision of $G$. Then $F$ admits an oddomorphism to $G$.
\end{lemma}
\proof
Let $\{1, \ldots, n\}$ be the vertex set of $G$ and let $V(F) \supseteq [n]$ such that vertex $i$ of $F$ is the main vertex of $F$ that corresponds to vertex $i$ of $G$. For each edge $ij$ of $G$, let $P_{ij}$ be the path from $i$ to $j$ in $F$ that replaced the edge $ij$ of $G$ and let $v^{ij}_{k}$ denote the vertex at distance $k$ from $i$ along this path (so $v^{ij}_0 = i$). Note that order matters in this notation and so $v^{ij}_k = v^{ji}_{\ell-k}$ where $\ell$ is the length of the path $P_{ij}$. Define a function $\psi: V(F) \to V(G)$ such that $\psi(i) = i$ for all $i \in [n]$ and
\[\psi(v^{ij}_k) = \begin{cases} i & \text{if } k \equiv 0 \ \mathrm{mod} \ 2 \\ j & \text{if } k \equiv 1 \ \mathrm{mod} \ 2
\end{cases}\]
for all edges $ij$ of $G$. Note that this is well defined since if $v^{ij}_k = v^{ji}_{k'}$ then $k' = \ell - k$ where $\ell$ is the length of the path $P_{ij}$ which is odd. Thus $k$ and $k'$ have opposite parity and thus the above definition of $\psi$ maps $v^{ij}_k$ and $v^{ji}_{k'}$ to the same vertex of $G$.

It is easy to see that $\psi$ is an oddomorphism. For $i \in [n]$ we have that $N_F(i) \cap \psi^{-1}(j) = \{v^{ij}_1\}$ for all $j \sim_G i$ and thus all vertices $i \in [n]$ in $F$ are odd with respect to $\psi$. On the other hand for any internal vertex $v^{ij}_k$ of $P_ij$ that is mapped to $i$ by $\psi$ we have that $N_F(v^{ij}_k) \cap \psi^{-1}(j) = \{v^{ij}_{k-1},v^{ij}_{k+1}\}$ and $N_F(v^{ij}_k) \cap \psi^{-1}(j') = \varnothing$ for all other $j' \in V(G)$, thus each of these vertices is even w.r.t.~$\psi$. Therefore every vertex of $F$ is either odd or even and each fibre of $\psi$ contains precisely one odd vertex. Thus $\psi$ is an oddomorphism as desired.\qeds

We immediately obtain that any cycle of length at least and of the same parity as $k$ has an oddomorphism to $C_k$:

\begin{cor}\label{cor:cycleoddos}
Let $k', k \ge 3$. Then $C_{k'}$ has an oddomorphism to $C_k$ if and only if $k' \ge k$ and $k' \equiv k \ \mathrm{mod} \ 2$. The same holds for weak oddomorphisms.
\end{cor}
\proof
Suppose that $k' \ge k$ and $k' \equiv k \ \mathrm{mod} \ 2$. Then $C_{k'}$ has an oddomorphism to $C_k$ by Lemma~\ref{lem:subdivisions} and the fact that $C_{k'}$ can be obtained from $C_k$ by replacing one of its edges with a path of length $k'-k+1$ and leaving all other edges as paths of length one. Since $C_{k'}$ has an oddomorphism to $C_k$ in this case, it also has a weak oddomorphism to $C_k$.

For the converse, we have three cases:

{\textbf{Case 1: $k' < k$.}} \ In this case $C_{k'}$ cannot have a weak oddomorphism to $C_k$ since weak oddomorphisms are always surjective.

{\textbf{Case 2: $k'$ odd, $k$ even.}} $C_{k'}$ cannot have a weak oddomorphism to $C_k$ because weak oddomorphisms are homomorphisms and odd cycles cannot have homomorphisms to even cycles (or to any bipartite graph).

{\textbf{Case 3: $k'$ even, $k$ odd.}} In this case we make use of Theorem~\ref{thm:maindualG}. Let $G = C_k$. From Example~\ref{ex:cycles} we see that $G_0 = 2C_k$ and $G_1 = C_{2k}$. We will show that $C_{k'}$ must have the same number of homomorphisms to $G_0$ and $G_1$ which will prove that $C_{k'}$ does not have a weak oddomorphism to $C_k$.

We make use of the \emph{categorical product of graphs}. Given graphs $H$ and $H'$ their categorical product, denoted $H \times H'$ has vertex set $V(H) \times V(H')$ with $(u,u') \sim (v,v')$ if $u \sim_H v$ and $u' \sim_{H'} v'$. It is well known that $\hom(F,H \times H') = \hom(F,H)\hom(F,H')$ for any graph $F$. Furthermore, it is known that $K_2 \times H \cong 2H$ for any bipartite graph $H$, and $K_2 \times C_k \cong C_{2k}$ for odd $k$. Therefore, we have that $K_2 \times G_0 \cong 2C_{2k} \cong K_2 \times G_1$ and thus
\[\hom(F,K_2)\hom(F,G_0) = \hom(F,K_2\times G_0) = \hom(F,K_2\times G_1) = \hom(F,K_2)\hom(F,G_1).\]
For bipartite $F$, the term $\hom(F,K_2)$ is nonzero and thus we obtain that $\hom(F,G_0) = \hom(F,G_1)$ for all bipartite $F$. Since $k'$ is even in this case we have that $\hom(C_{k'},G_0) = \hom(C_{k'},G_1)$ and therefore $C_{k'}$ does not have a weak oddomorphism to $G = C_k$.

Since in the above cases $C_{k'}$ does not have a weak oddomorphism to $C_k$, it also does not have an oddomorphism to $C_k$.\qeds

\begin{remark}\label{rem:bipartite}
Note that in Case 3 of the above proof we actually showed that if $k$ is odd, then $F$ does not have a weak oddomorphism to $C_k$ whenever $F$ is \emph{any} connected bipartite graph. Later, we will combine this with other results to show that no bipartite graph can have a weak oddomorphism to a non-bipartite graph.
\end{remark}

In light of Lemma~\ref{lem:subdivisions} one may wonder whether every graph admitting a weak oddomorphism to a graph $G$ must contain a subdivision of $G$. However we will see a counterexample to this in Section~\ref{sec:discussion}.

If $F$ and $G$ are graphs, then $\psi: V(F) \to V(G)$ is a \emph{covering map}, and $F$ is a cover of $G$, if $\psi$ is a surjection which is \emph{locally bijective}, i.e., the restriction of $\psi$ to $N_F(a)$ is a bijection to $N_G(\psi(a))$ for all $a \in V(F)$. Note that this implies that $\psi$ is a homomorphism. Equivalently, $\psi$ is a covering map if each fibre of $\psi$ is an independent set and two fibres $\psi^{-1}(u)$ and $\psi^{-1}(v)$ have no edges between them if $u \not\sim v$ and have a perfect matching between them if $u \sim v$. It is easy to see from this definition that if $G$ is connected then all of the fibres of $\psi$ have the same size. If all fibres of $\psi$ have size $k$ then $F$ is said to be a $k$-cover of $G$. We will say that $\psi$ is an \emph{odd covering map}, and $F$ is an \emph{odd cover}, if every fibre of $\psi$ has odd cardinality. From these definitions the following lemma is immediate:

\begin{lemma}\label{lem:covers}
If $\psi$ is an odd covering map from $F$ to $G$, then $\psi$ is an oddomorphism from $F$ to $G$.
\end{lemma}
\proof
For each $a \in V(F)$ we have that $|N_F(a) \cap \psi^{-1}(v)| = 1$ for all $v \sim_G \psi(a)$. Thus every vertex of $F$ is odd with respect to $\psi$, and since every fibre of $\psi$ contains an odd number of vertices we have that $\psi$ is an oddomorphism.\qeds

The above lemmas provide many examples of oddomorphisms. Moreover, in Section~\ref{sec:oddomorphisms} we will see that oddomorphisms can be composed, and thus we can essentially compose the above constructions to obtain even more examples of oddomorphisms.

\section{Homomorphism counts from bounded degree graphs}\label{sec:degree}

The lemma below, where $\Delta(G)$ denotes the maximum degree of $G$, follows directly from the definition of oddomorphism. 

\begin{lemma}\label{lem:degree}
Let $\psi$ be an oddomorphism from $F$ to $G$ and let $a \in V(F)$ be an odd vertex. Then $\deg_F(a) \ge \deg_G(\psi(a))$, and thus $\Delta(F) \ge \Delta(G)$ whenever $F$ admits a weak oddomorphism to $G$.
\end{lemma}
\proof
Let $v = \psi(a)$ and let $v_1, \ldots, v_d$ be the neighbors of $v$ in $G$. Since $a$ is an odd vertex in $\psi^{-1}(v)$, we have that $a$ has an odd, and thus strictly positive, number of neighbors in each of $\psi^{-1}(v_1), \ldots, \psi^{-1}(v_d)$ which are disjoint. Thus $\deg_F(a) \ge \deg_G(v)$.

If a graph $F$ has a weak oddomorphism to $G$ then there is some subgraph $F'$ of $F$ that has an oddomorphism to $G$. Since oddomorphisms are surjective by Condition (2) of Definition~\ref{def:oddomorphism}, the above implies that $\Delta(F') \ge \Delta(G)$ and thus $\Delta(F) \ge \Delta(G)$.\qeds

As an immediate consequence of Lemma~\ref{lem:degree} and Theorem~\ref{thm:maindualG}, we are able to show that homomorphism counts from bounded degree graphs do not determine a graph up to isomorphism.

\begin{cor}\label{cor:degree}
For all $d \in \mathbb{N}$ there exist graphs $H$ and $H'$ such that $\hom(F,H) = \hom(F,H')$ for any graph $F$ of maximum degree less than $d$, but $H \not\cong H'$.
\end{cor}
\proof
Let $G$ be any connected graph of maximum degree at least $d$ and let $v \in V(G)$ be a vertex of maximum degree. Suppose that $\hom(F,G_0) \ne \hom(F,G_1)$ for some graph $F$. We will show that $F$ must have maximum degree at least $d$. By Theorem~\ref{thm:maindualG} we have that $F$ must have a weak oddomorphism to $G$ and thus $\Delta(F) \ge \Delta(G) \ge d$ by Lemma~\ref{lem:degree}.\qeds %there is a subgraph $F'$ of $F$ and an oddomorphism $\psi$ from $F'$ to $G$. Furthermore, the set $\psi^{-1}(v)$ contains an odd vertex $a \in V(F')$ and by Lemma~\ref{lem:degree} $\deg_{F'}(a) \ge \deg_G(v) \ge d$. Since $F'$ is a subgraph of $F$, we have that the maximum degree of $F$ is at least $d$. Of course, we also have $G_0 \not\cong G_1$ and thus we have proven the corollary.\qeds

The simplest example of a graph with maximum degree $d$ is the star graph $K_{1,d}$. By considering this choice of $G$ and examining the structure of $G_0$ and $G_1$, we can strengthen the above result to show that homomorphism counts from graphs of bounded degree do not even distinguish connected graphs from those with isolated vertices. For this we will use a simplified description of the graphs $G_0$ and $G_1$. For $d \in \mathbb{N}$, define $G^d_i$ for $i = 0,1$ to be a graph with vertex set $[d] \cup \{S \subseteq [d] : |S| \equiv i \ \mathrm{mod} \ 2\}$ such that $j \in [d]$ is adjacent to $S$ if $j \notin S$. Then we have the following:

\begin{lemma}\label{lem:Gdiso}
For $G \cong K_{1,d}$, we have $G^d_i \cong G_i$ for $i = 0,1$.
\end{lemma}
\proof
Let $G$ have vertex set $\{0,1,\ldots, d\}$ and be isomorphic to $K_{1,d}$ such that vertex $0$ has degree $d$. Take $G_1 = G_{\{0\}}$. The isomorphism is given by mapping vertex $j \in [d]$ of $G_i^d$ to vertex $(j,\varnothing)$ of $G_i$ and mapping vertex $S \subseteq [d]$ of $G_i^d$ to vertex $(0,S)$ of $G_i$. It is straightforward to check that this is indeed an isomorphism for $i = 0,1$.\qeds

Next we consider the connectedness of the graphs $G_0^d$ and $G_1^d$.

\begin{lemma}\label{lem:Gdconnected}
For odd $d \in \mathbb{N}$, the graph $G_0^d$ is connected (with more than one vertex) and the graph $G_1^d$ contains an isolated vertex. For even $d \ge 4$ the graph $G_1^d$ is connected (with more than one vertex) and $G_0^d$ contains an isolated vertex.
\end{lemma}
\proof
If $d$ is odd, then $G^d_1$ contains the vertex $[d]$ which is not adjacent to any other vertices and therefore is an isolated vertex. Now we will show that $G^d_0$ is connected for $d$ odd. First note that $\varnothing$ is always a vertex of $G_0^d$ and it is adjacent to all vertices of the form $j \in [d]$. Now consider an arbitrary subset $S \subseteq [d]$ which is a vertex of $G_0^d$. This means that $|S|$ is even and thus $S \ne [d]$. Therefore there exists $j \in [d]$ such that $j \notin S$ and therefore $S$ and $j$ are adjacent in $G_0^d$. This shows that every vertex of $G_0^d$ has a path to $\varnothing$ and therefore $G_0^d$ is connected. Note that $\varnothing$ and $1$ are two distinct vertices of $G_0^d$ in this case and thus $G_0^d$ has more than one vertex.

If $d$ is even, then $G_0^d$ contains the vertex $[d]$ which is isolated. We will show that $G_1^d$ is connected for even $d \ge 4$. Suppose that $d \ge 4$. Suppose that $S \subseteq [d]$ is a vertex of $G_1^d$, i.e., $|S|$ is odd. Since $d$ is even we have that $S \ne [d]$ and thus there exists $j \in [d]$ with $j \notin S$ and therefore $S \sim j$. So every vertex of $G_1^d$ which is a subset of $[d]$ is adjacent to at least one vertex of $G_1^d$ which is an element of $[d]$. Thus to show $G_1^d$ is connected it suffices to show that any two vertices $i,j \in [d]$ of $G_1^d$ are connected by some path. Since $d \ge 4$, for any such choice of $i$ and $j$  there exists an element $\ell \in [d]$ such that $\ell \notin \{i,j\}$, and therefore the vertex $\{\ell\}$ is adjacent to both $i$ and $j$ and thus we have proven connectedness. Lastly, since $[d] \subseteq V(G_1^d)$ and $d \ge 4$, we have that $G_1^d$ contains more than one vertex.\qeds

\begin{remark}
 Note that $G^2_0$ does contain an isolated vertex (it is isomorphic to the disjoint union of $K_{1,2}$ and $K_1$), but the graph $G_1^2$ is not connected (it is isomorphic to the disjoint union of two $K_2$'s). Furthermore, the graphs $G_0^0$ and $G_1^0$ are the graphs with one and zero vertices respectively.
\end{remark}
 
Now we are able to prove the following which shows that homomorphism counts from graphs of bounded degree cannot even distinguish between connected graphs and graphs with isolated vertices:

\begin{theorem}\label{thm:degree}
For all $d \in \mathbb{N}$ there exist graphs $H$ and $H'$ such that $\hom(F,H) = \hom(F,H')$ for any graph $F$ of maximum degree less than $d$, but $H$ is connected and $H'$ is not connected and has an isolated vertex (so it is not just a single vertex) and moreover $\hom(K_{1,d},H) \ne \hom(K_{1,d},H')$.
\end{theorem}
\proof
For $d \ge 3$ this is an immediate consequence of Lemmas~\ref{lem:Gdiso}, \ref{lem:Gdconnected}, and~\ref{lem:idoddo}, where $H$ and $H'$ are $G^d_0$ and $G^d_1$ respectively. For each $d = 1,2$ we give a direct proof.

We first prove the theorem for $d = 2$. Note that if $F$ has maximum degree less than 2, then it is the disjoint union of $K_2$'s and $K_1$'s. Thus two graphs $H$ and $H'$ will satisfy $\hom(F,H) = \hom(F,H')$ for all $F$ of maximum degree less than 2 if and only if $H$ and $H'$ have the same number of vertices and edges. Take $H$ to be the path on four vertices and $H'$ to be the disjoint union of $K_3$ and $K_1$. Then $H$ is connected and both $H$ and $H'$ have four vertices and three edges. However, $\hom(K_{1,2},H) = 10$ and $\hom(K_{1,2},H') = 12$ since these numbers are just the sums of the squares of the degrees.

For $d = 1$, we can take $H$ to be any connected graph and $H'$ to be any disconnected graph that has an isolated vertex and is on the same number of vertices as $H$ but different number of edges, e.g. $H = K_2$ and $H' = \overline{K_2}$.\qeds

Let $\mathcal{T}$ denote the class of all trees, let $\mathcal{T}_d$ denote the class of all trees of degree at most $d$, and let $\mathcal{G}_d$ denote the class of all graphs of degree at most $d$. In~\cite{homtensors} they prove that for every $d \ge 1$ that there are graphs $H$ and $H'$ that are homomorphism indistinguishable over the class of trees of degree at most $d$ but are not homomorphism indistinguishable over the class of trees of degree $d+1$, i.e., that the relation $\cong_{\mathcal{T}_d}$ is strictly weaker than $\cong_{\mathcal{T}_{d+1}}$. Theorem~\ref{thm:degree} above strengthens this to say that even $H \cong_{\mathcal{G}_d} H'$ does not imply $H \cong_{\mathcal{T}_{d+1}} H'$. We also remark that the graphs $H$ and $H'$ in our proof are significantly more explicit (and simple) than the graphs arising in the proof in~\cite{homtensors}\footnote{In fact the graphs in~\cite{homtensors} are not fully explicit but could be made explicit without much effort.}. It is interesting that even the simplest graph of degree at least $d$, the star $K_{1,d}$, can be used to distinguish some graphs that have the same number of homomorphisms from any graph of degree less than $d$. A natural question to ask is whether there is \emph{any} graph $\hat{F} \notin \mathcal{G}_d$ such that $H \cong_{\mathcal{G}_d} H'$ implies that $\hom(\hat{F},H) = \hom(\hat{F},H')$. Our next theorem answers this question in the negative.

\begin{theorem}\label{thm:closed}
Let $d \in \mathbb{N}$ and let $\hat{F}$ be any graph of maximum degree greater than $d$. Then there exist graphs $H$ and $H'$ such that $\hom(F,H) = \hom(F,H')$ for all graphs $F$ of degree at most $d$ but $\hom(\hat{F},H) \ne \hom(\hat{F},H')$.
\end{theorem}
\proof
Let $F_1, \ldots, F_k$ be the connected components of $\hat{F}$ and suppose without loss of generality that $F_1$ has maximum degree greater than $d$ and set $G = F_1$. Take $H$ to be the disjoint union of $G_0$ and a clique $K_r$ of sufficient size such that $\hom(F_i,K_r) > 0$ for $i = 1, \ldots, k$, e.g., taking $r = \max\{\chi(F_i) : i \in [k]\}$ suffices. Similarly, take $H'$ to be the disjoint union of $G_1$ and $K_r$. This ensures that $\hom(F_i,H)$ and $\hom(F_i,H')$ are greater than zero for all $i$.

Now suppose that $F$ is a connected graph of degree at most $d$. By Lemma~\ref{lem:degree} we have that $\hom(F,G_0) = \hom(F,G_1)$. Moreover, since $F$ is connected, we have that
\[\hom(F,H) = \hom(F,G_0) + \hom(F,K_r) = \hom(F,G_1) + \hom(F,K_r) = \hom(F,H').\]
Since the equality holds for any connected graph of degree at most $d$, it further holds for any graph of degree at most $d$.

Now we must show that $\hom(\hat{F},H) \ne \hom(\hat{F},H')$. We will show that $\hom(F_i,H) \ge \hom(F_i,H') > 0$ for all $i = 1,\ldots, k$ with strict inequality for $i = 1$, thus proving the inequality. First, note that since $H'$ contains $K_r$ and $\hom(F_i,K_r) > 0$ by construction, we have that $\hom(F_i,H') > 0$ for all $i \in [k]$. Now, using Theorem~\ref{thm:main}, we have that
\[\hom(F_i,H) = \hom(F_i,G_0) + \hom(F_i,K_r) \ge \hom(F_i,G_1) + \hom(F_i,K_r) = \hom(F_i,H')\]
for all $i = 1, \ldots, k$. Furthermore, since $\hom(F_1,G_0) \ne \hom(F_1,G_1)$ by Lemma~\ref{lem:idoddo}, we have that $\hom(F_1,G_0) > \hom(F_1,G_1)$ and thus $\hom(F_1,H) > \hom(F_1,H')$. Finally, combining all of this, we have that
\[\hom(\hat{F},H) = \prod_{i=1}^k \hom(F_i,H) > \prod_{i=1}^k \hom(F_i,H') = \hom(\hat{F},H').\]\qeds

% Take $H$ to be the disjoint union of $G_0$ and $\hat{F}$ and take $H'$ to be the disjoint union of $G_1$ and $\hat{F}$. Suppose that $F$ is a connected graph of degree at most $d$. By Lemma~\ref{lem:degree} we have that $\hom(F,G_0) = \hom(F,G_1)$. Moreover, since $F$ is connected, we have that
% \[\hom(F,H) = \hom(F,G_0) + \hom(F,\hat{F}) = \hom(F,G_1) + \hom(F,\hat{F}) = \hom(F,H').\]
% Since the equality holds for any connected graph of degree at most $d$, it further holds for any graph of degree at most $d$.

% Now we must show that $\hom(\hat{F},H) \ne \hom(\hat{F},H')$. We will show that $\hom(F_i,H) \ge \hom(F_i,H') > 0$ for all $i = 1,\ldots, k$ with strict inequality for $i = 1$, thus proving the inequality. First, note that since $H'$ contains a copy of $\hat{F}$, it contains a copy of $F_i$ for all $i = 1,\ldots, k$ and therefore $\hom(F_i,H') > 0$ for all $i$. Now, using Theorem~\ref{thm:main}, we have that
% \[\hom(F_i,H) = \hom(F_i,G_0) + \hom(F_i,\hat{F}) \ge \hom(F_i,G_1) + \hom(F_i,\hat{F}) = \hom(F_i,H')\]
% for all $i = 1, \ldots, k$. Furthermore, since $\hom(F_1,G_0) \ne \hom(F_1,G_1)$ by Lemma~\ref{lem:idoddo}, we have that $\hom(F_1,G_0) > \hom(F_1,G_1)$ and thus $\hom(F_1,H) > \hom(F_1,H')$. Finally, combining all of this, we have that
% \[\hom(\hat{F},H) = \prod_{i=1}^k \hom(F_i,H) > \prod_{i=1}^k \hom(F_i,H') = \hom(\hat{F},H').\]\qeds

The above result says that the class $\mathcal{G}_d$ of graphs of degree at most $d$ is homomorphism distinguishing closed. The key property of $\mathcal{G}_d$ used in the proof above is that it is closed under weak oddomorphisms, i.e., if $F \in \mathcal{G}_d$ and $F$ has a weak oddomorphism to $G$, then $G \in \mathcal{G}_d$. In Section~\ref{sec:hdclosed} we will see that any family with this property that is also closed under disjoint unions and restrictions to connected components is homomorphism distinguishing closed.

We end this section by giving a complete characterization of the graphs admitting a weak oddomorphism to $K_{1,d}$:

\begin{lemma}
A graph $F$ has a weak oddomorphism to $K_{1,d}$ if and only if $F$ is bipartite with maximum degree at least $d$.
\end{lemma}
\proof
If $F$ is not bipartite then $F$ does not admit a homomorphism to $K_{1,d}$ and thus does not admit a weak oddomorphism either. If $F$ has maximum degree less than $d$, then it does not admit a weak oddomorphism to $K_{1,d}$ by Lemma~\ref{lem:degree}. This proves one direction of the lemma.

Now suppose that $F$ is bipartite and has maximum degree at least $d$. Let $a \in V(F)$ be a vertex of degree at least $d$ and let $a_1, \ldots, a_d$ be $d$ distinct neighbors of $a$. Let $F'$ be the subgraph of $F$ induced by $\{a, a_1, \ldots, a_d\}$ (note there are no edges among the $a_i$ as $F$ is bipartite), and let $F''$ be the connected component of $F$ containing $a$. We will define a homomorphism $\psi$ from $F$ to $K_{1,d}$, whose vertex set we take to be $\{0, 1, \ldots, d\}$ with $0$ being the vertex of degree $d$. First, we set $\psi(a) = 0$ and $\psi(a_i) = i$ for $i = 1, \ldots, d$. We extend this to the rest of $F''$ by mapping any vertex in the same part of the bipartition as $a$ to $0$ and mapping all remaining vertices to $1$. We then extend this to a homomorphism from $F$ by mapping every other connected component of $F$ to the edge 01 according to some 2-coloring. It is easy to see that this is a homomorphism from $F$ to $K_{1,d}$ and moreover that the restriction to $F'$ is an oddomorphism where every vertex is odd.\qeds

\section{Further properties of oddomorphisms}\label{sec:oddomorphisms}

In this section we will prove some properties of oddomorphisms and weak oddomorphisms, such as the fact that they are composable and induce a partial order on the class of all (isomorphism classes of) graphs.

\begin{lemma}\label{lem:oddcomp}
Suppose that $\psi_1$ is an oddomorphism from $F$ to $G$ and that $\psi_2$ is an oddomorphism from $G$ to $H$. Then the composition $\psi_2 \circ \psi_1$ is an oddomorphism from $F$ to $H$.
\end{lemma}
\proof
First note that since $\psi_1$ and $\psi_2$ are homomorphisms, their composition $\psi = \psi_2 \circ \psi_1$ is a homomorphism.

Now let $V_1^0$ and $V_1^1$ be respectively the sets of even and odd vertices with respect to $\psi_1$ and similarly define $V_2^0$ and $V_2^1$. We will first show that every vertex of $F$ is either odd or even with respect to $\psi$ and that its set of odd vertices is $\{a \in V(F): a \in V_1^1 \ \& \ \psi_1(a) \in V_2^1\}$. Let $a \in V(F)$ be such that $a \in V_1^i$ and $u = \psi_1(a) \in V_2^j$ for some $i,j \in \{0,1\}$. Let $x = \psi(a) \in V(H)$ and suppose that $y \sim_H x$. We will show that $|N_F(a) \cap \psi^{-1}(y)| \equiv ij \ \mathrm{mod} \ 2$ which proves that $a$ is either even or odd and is the latter if and only if $i = j = 1$. We have that $\psi^{-1}(y) = \psi_1^{-1}(\psi_2^{-1}(y))$ and thus
\[\psi^{-1}(y) = \bigcup_{v \in \psi_2^{-1}(y)} \psi_1^{-1}(v).\]
Therefore
\[N_F(a) \cap \psi^{-1}(y) = \bigcup_{v \in \psi_2^{-1}(y)} \left(N_F(a) \cap \psi_1^{-1}(v)\right).\]
Since all these sets are disjoint we have that
\[|N_F(a) \cap \psi^{-1}(y)| = \sum_{v \in \psi_2^{-1}(y)} |N_F(a) \cap \psi_1^{-1}(v)|.\]
Since $\psi_1$ is a homomorphism, the set $N_F(a) \cap \psi_1^{-1}(v)$ is empty unless $v \sim_G \psi_1(a))$, i.e., $v \in N_G(\psi_1(a))$. Therefore,
\[|N_F(a) \cap \psi^{-1}(y)| = \sum_{v \in N_G(\psi_1(a)) \cap \psi_2^{-1}(y)} |N_F(a) \cap \psi_1^{-1}(v)|.\]
Now since $\psi_1$ is an oddomorphism, the set $N_F(a) \cap \psi_1^{-1}(v)$ has parity $i$ for all $v \in N_G(\psi_1(a))$. Moreover, the set being summed over,  $N_G(\psi_1(a)) \cap \psi_2^{-1}(y)$, has parity $j$ since we assumed that $\psi_1(a) \in V_2^j$. Thus the parity of $|N_F(a) \cap \psi^{-1}(y)|$ is $ij$ for all $y \sim_H \psi(a)$. As this does not depend on $y$, we have shown that $a$ is odd if and only if $ij = 1$ (i.e., $i = j = 1$) and is otherwise even. Moreover this shows that the set of odd vertices with respect to $\psi$ is $\{a \in V(F): a \in V_1^1 \ \& \ \psi_1(a) \in V_2^1\}$.

Now we must show that $\psi^{-1}(x)$ contains an odd number of odd vertices for all $x \in V(H)$. We have that 
\[\psi^{-1}(x) = \bigcup_{v \in \psi_2^{-1}(x)} \psi_1^{-1}(v).\]
Note that the union above is in fact a disjoint union. Also, if $v \in V^2_0$, then our above arguments show that $\psi_1^{-1}(v)$ contains no $\psi$-odd vertices. Therefore, we only need to count the number of $\psi$-odd vertices in 
\[\bigcup_{v \in V_1^2 \cap \psi_2^{-1}(x)} \psi_1^{-1}(v).\]
Again by the above, a vertex $a$ in this set will be $\psi$-odd if and only if it is $\psi_1$-odd. Since $\psi_1$ is an oddomorphism, each $\psi_1^{-1}(v)$ will contribute an odd number of such vertices. Moreover, the set $V_1^2 \cap \psi_2^{-1}(x)$ we are taking the union over is odd since $\psi_2$ is an oddomorphism. Therefore, the above union contains an odd number of $\psi$-odd vertices and we are done.\qeds

% We first consider the number of odd vertices (with respect to $\psi$) contained in each $\psi_1^{-1}(v)$. Since $\psi_1$ is an oddomorphism, $\psi_1^{-1}(v)$ contains an odd number of $\psi_1$-odd vertices, i.e., it contains an odd number of elements of $V_1^1$. If $v$ is $\psi_2$-odd (i.e., $v \in V_2^1$), then each element of $\psi_1^{-1}(v) \cap V_1^1$ is odd w.r.t.~$\psi$ and the remaining elements of $\psi_1^{-1}(v)$ are even by the above. On the other hand, if $v$ is $\psi_2$-even (i.e., $v \in V_2^0$), then the above shows that no element of $\psi^{-1}(v)$ is odd w.r.t.~$\psi$. Therefore, each $v \in V_2^1 \cap \psi_2^{-1}(x)$ contributes an odd number of odd (w.r.t.~$\psi$) vertices to $\psi^{-1}(x)$. As $\psi_2$ is an oddomorphism, we have that $\psi_2^{-1}(x)$ contains an odd number of $\psi_2$-odd vertices. Thus $\psi^{-1}(x)$ contains an odd number of vertices that are odd with respect to $\psi$. Therefore $\psi$ is an oddomorphism from $F$ to $G$.\qeds

To state the following lemma, we will need to introduce some notation. Given a homomorphism $\psi$ from $F$ to $G$ and a subgraph $G'$ of $G$, we denote by $\psi^{-1}(G')$ the subgraph $F'$ of $F$ with vertex set $\psi^{-1}(V(G'))$ and edge set $\{ab \in E(F) : a,b \in V(F') \ \& \ \psi(a)\psi(b) \in E(G')\}$.

\begin{lemma}\label{lem:suboddo}
Let $\psi$ be an oddomorphism from $F$ to $G$. Suppose that $G'$ is a subgraph of $G$ and let $F' = \psi^{-1}(G')$. Then $\psi|_{V(F')}$ is an oddomorphism from $F'$ to $G'$ such that the parity of $v \in V(F')$ w.r.t.~$\psi|_{V(F')}$ is the same as its parity w.r.t.~$\psi$. It follows that if $\psi$ is a weak oddomorphism from $F$ to $G$ and $F' = \psi^{-1}(G')$, then $\psi|_{V(F')}$ is a weak oddomorphism from $F'$ to $G'$.
\end{lemma}
\proof
Let $\varphi = \psi|_{V(F')}$. We must first show that every vertex of $F'$ is either even or odd with respect to $\varphi$.

% Note that since $V(F') = \psi^{-1}(V(G'))$, we have that $\varphi^{-1}(v) = \psi^{-1}(v)$ for all $v \in V(G')$. Now if $b \in N_F(a) \cap \psi^{-1}(v)$ for some $a \in V(F')$ and $v \in N_{G'}(\psi(a))$, then $v \in V(G')$ and $b \in \varphi^{-1}(v)$, and thus $b \in V(F')$. Furthermore, since $v \in N_{G'}(\psi(a))$, we have that $\psi(a)\psi(b) = \psi(a)v$ is an edge of $G'$ and thus $ab \in E(F')$, i.e., $b \in N_{F'}(a)$. Therefore $N_F(a) \cap \psi^{-1}(v) = N_{F}(a) \cap \varphi^{-1}(v) \subseteq N_{F'}(a) \cap \varphi^{-1}(v)$ for all $a \in V(F')$ and $v \in N_{G'}(\varphi(a))$. As the other containment is obvious, we have 
It is clear from the definition of $F'$ and $\varphi$ that $N_F(a) \cap \psi^{-1}(v) = N_{F'}(a) \cap \varphi^{-1}(v)$ for all $a \in V(F')$ and $v \in N_{G'}(\varphi(a))$. It follows that every vertex of $F'$ is even or odd with respect to $\varphi$, and moreover its odd vertices are precisely the odd vertices of $\psi$. From the latter it follows that $\varphi^{-1}(v)$ contains an odd number of odd vertices for all $v \in V(G')$. Therefore, $\varphi = \psi|_{V(F')}$ is an oddomorphism from $F'$ to $G'$.

Lastly, suppose that $\psi$ is a weak oddomorphism from $F$ to $G$ and let $F' = \psi^{-1}(G')$. Then by definition there is a subgraph $\hat{F}$ of $F$ such that $\hat{\psi} := \psi|_{V(\hat{F})}$ is an oddomorphism from $\hat{F}$ to $G$. By the above it then follows that if $\hat{F}' := \hat{\psi}^{-1}(G')$ (which is a subgraph of both $\hat{F}$ and $F'$), then $\hat{\psi}|_{V(\hat{F}')}$ is an oddomorphism from $\hat{F}'$ to $G'$. Since $\hat{\psi}|_{V(\hat{F}')} = \psi|_{V(\hat{F}')} = \left(\psi|_{V(F')}\right)|_{V(\hat{F}')}$, it follows that $\psi|_{V(F')}$ is a weak oddomorphism from $F'$ to $G'$ as desired.\qeds

We can now combine the above two lemmas to prove that weak oddomorphisms can be composed.

\begin{lemma}\label{lem:woddcomp}
Suppose that $\psi_1$ is a weak oddomorphism from $F$ to $G$ and that $\psi_2$ is a weak oddomorphism from $G$ to $H$. Then the composition $\psi_2 \circ \psi_1$ is a weak oddomorphism from $F$ to $H$.
\end{lemma}
\proof
By definition of weak oddomorphism, there exist subgraphs $F'$ and $G'$ of $F$ and $G$ respectively such that $\varphi_1 := \psi_1|_{V(F')}$ and $\varphi_2 := \psi_2|_{V(G')}$ are oddomorphisms from $F'$ to $G$ and from $G'$ to $H$ respectively. Let $F'' = {\varphi_1}^{-1}(G')$. By Lemma~\ref{lem:suboddo}, we have that $\varphi_1|_{V(F'')}$ is an oddomorphism from $F''$ to $G'$ and thus by Lemma~\ref{lem:oddcomp} the map $\varphi_2 \circ \varphi_1|_{V(F'')}$ is an oddomorphism from $F''$ to $H$. Since $\varphi_2 \circ \varphi_1|_{V(F'')} = \left(\psi_2 \circ \psi_1\right)|_{V(F'')}$, we have that $\psi_2 \circ \psi_1$ is a weak oddomorphism from $F$ to $H$.\qeds

We will write $F \oddto G$ if $F$ admits an oddomorphism to $G$ and $F \woto G$ if $F$ admits a weak oddomorphism to $G$. Lemmas~\ref{lem:oddcomp} and~\ref{lem:woddcomp} show that $\oddto$ and $\woto$ are transitive relations. Moreover, Lemma~\ref{lem:idoddo} proves that $\oddto$ and $\woto$ are reflexive. We will now show that $G \woto H$ and $H \woto G$ implies that $G \cong H$, and thus both $\oddto$ and $\woto$ are antisymmetric and therefore both are partial orders on the set of isomorphism classes of graphs.

\begin{lemma}\label{lem:antisym}
Let $G$ and $H$ be graphs. If $G \woto H$ and $H \woto G$, then $G \cong H$.
\end{lemma}
\proof
Suppose that $\psi$ is a weak oddomorphism from $G$ to $H$ and $G'$ is a subgraph of $G$ such that $\psi|_{V(G'})$ is an oddomorphism from $G'$ to $H$. Since $\psi^{-1}(v)$ contains an odd number of odd vertices of $G'$ for every $v \in V(H)$, we have that $\psi$ is a surjective homomorphism from $G$ to $H$. Similarly, there must be a surjective homomorphism $\varphi$ from $H$ to $G$. This implies that $|V(G)| = |V(H)|$, and moreover the adjacency preserving property of homomorphisms implies that $|E(G)| = |E(H)|$. Therefore, $\psi$ is a bijection from $V(G)$ to $V(H)$ that preserves adjacency and since $|E(G)| = |E(H)|$ this implies that $\psi$ also preserves non-adjacency, i.e., $\psi$ is an isomorphism.\qeds

Thus we have the following:

\begin{theorem}\label{thm:partialorder}
The relations $\oddto$ and $\woto$ are partial orders on the set of (isomorphism classes) of graphs.
\end{theorem}

In Lemma~\ref{lem:suboddo} we showed that if $F$ has a weak oddomorphism to $G$ and $G'$ is a subgraph of $G$, then there is a subgraph $F'$ of $F$ that has an oddomorphism to $G'$. We now extend this from subgraphs to minors.

\begin{lemma}\label{lem:minoroddo}
Let $F$ and $G$ be graphs such that $F$ has a weak oddomorphism to $G$. If $G'$ is a minor of $G$, then there is a minor $F'$ of $F$ such that $F'$ admits an oddomorphism to $G'$.
\end{lemma}
\proof
As with Lemma~\ref{lem:suboddo}, it suffices to consider the case where $F$ has an oddomorphism to $G$. Moreover, 
since we have already proven the lemma for subgraphs of $G$, it suffices to prove the case where $G'$ is obtained from $G$ by contracting a single edge. So let $uv \in E(G)$ and let $G' = G/uv$. We will refer the vertex formed by contracting $uv$ as $\hat{w}$ and thus $V(G') = \left(V(G) \setminus \{u,v\}\right) \cup \{\hat{w}\}$.

Let $\psi$ be an oddomorphism from $F$ to $G$. We will build a graph $F'$ and map $\psi'$ such that $\psi'$ is an oddomorphism from $F'$ to $G'$. The construction of $F'$ will be in several steps. First, we will define $G''$ and $F''$. For each $x \in V(G)$ that is adjacent to both $u$ and $v$, remove the edge between $x$ and $v$ (we could have chosen $u$ instead). We refer to the graph obtained after performing this for all relevant $x$ as $G''$. Note that $G''/uv = G/uv$, since we do not allow multiple edges. We define $F''$ by removing all edges between $\psi^{-1}(x)$ and $\psi^{-1}(v)$ for each $x \in V(G)$ adjacent to both $u$ and $v$. Note that $F'' = \psi^{-1}(G'')$ and thus $\psi$ is an oddomorphism from $F''$ to $G''$ by Lemma~\ref{lem:suboddo}. Now to obtain $F'$ we first contract all edges between $\psi^{-1}(u)$ and $\psi^{-1}(v)$. Let us refer to the vertex sets of the connected components of the subgraph of $F$ induced by $\psi^{-1}(u) \cup \psi^{-1}(v)$ as $C_1, \ldots, C_k$ and let $c_i$ be the vertex the elements of $C_i$ are contracted to when performing the above described contractions. Thus
\[V(F') = \left(V(F) \setminus \left(\bigcup_{i=1}^k C_i\right)\right) \cup \{c_1, \ldots, c_k\}.\]
Finally, for each $c_i$ and $a \in V(F) \setminus \left(\cup_{i=1}^k C_i\right)$ we remove the edge between $c_i$ and $a$ (if there is any) if $|N_{F''}(a) \cap C_i|$ was even. Thus a vertex $a \in V(F) \setminus \left(\cup_{i=1}^k C_i\right)$ is adjacent to $c_i$ in $F'$ if and only if $|N_{F''}(a) \cap C_i|$ was odd. Note that there are no edges among the $c_i$. By construction we see that $F'$ is a minor of $F$. Also note that the edges between $\psi^{-1}(x)$ and $\psi^{-1}(y)$ in $F'$ are the same as those in $F''$ as long as $x,y \in V(G') \setminus \{\hat{w}\}$.

Now define $\psi' : V(F') \to V(G')$ as 
\[\psi'(a) = \begin{cases} \psi(a) & \text{if } a \notin \{c_1, \ldots, c_k\} \\ \hat{w} &\text{o.w.} \end{cases}\]
We first must show that this is a homomorphism from $F'$ to $G'$. If $a,b \in V(F')$ are not among $c_1, \ldots, c_k$ and $a \sim_{F'} b$, then $a \sim_{F''} b$ and thus  $\psi'(a) \sim_{G'} \psi'(b)$ since $\psi$ is a homomorphism from $F''$ to $G''$ and in this case $\psi'(a) = \psi(a)$ and $\psi'(b) = \psi(b)$. If $a \in V(F')$ is adjacent to $c_i$ then $a \notin\{c_1, \ldots, c_k\}$ and $a$ was adjacent to some element of $C_i \subseteq \psi^{-1}(u) \cup \psi^{-1}(v)$. Therefore, $\psi'(a) = \psi(a)$ was adjacent to either $u$ or $v$ in $G$ and thus $\psi'(a)$ is adjacent to $\hat{w} = \psi'(c_i)$ in $G'$. These are all of the adjacencies of $F'$ and thus we have shown that $\psi'$ is a homomorphism from $F'$ to $G'$.

Now we must show that every vertex of $F'$ is either even or odd with respect to $\psi'$. We will show that if $a \in V(F') \setminus \{c_1, \ldots, c_k\}$, then the parity of $a$ w.r.t.~$\psi'$ is the same as its parity w.r.t~$\psi$. Consider $a \in V(F') \setminus \{c_1, \ldots, c_k\}$ and let $x = \psi'(a) = \psi(a)$. Suppose that $a$ is odd w.r.t.~$\psi$ (the even case is similar). If $y \in V(G') \setminus \{\hat{w}\}$ is such that $y \sim_{G'} x$, then $N_{F'}(a) \cap {\psi'}^{-1}(y) = N_{F''}(a) \cap {\psi}^{-1}(y)$ and thus this set has odd size as desired. Now we consider $N_{F'}(a) \cap {\psi'}^{-1}(\hat{w}) = N_{F'}(a) \cap \{c_1, \ldots, c_k\}$ in the case where $x \sim_{G'} \hat{w}$. We must show that this set as odd size. By construction this is equivalent to there being an odd number of $i \in [k]$ such that $|N_{F''}(a) \cap C_i|$ is odd. Since $x \sim_{G'} \hat{w}$, we have that $x$ is adjacent to precisely one of $u$ and $v$ in $G''$. Suppose without loss of generality that $x \sim_{G''} u$. Thus $N_{F''}(a) \cap C_i = N_{F''}(a) \cap \left(C_i \cap \psi^{-1}(u)\right)$ for each $i$. Since $a$ is odd w.r.t.~$\psi$, it is adjacent (in $F'')$ to an odd number of elements of $\psi^{-1}(u)$ and it follows that there are an odd number of $i \in [k]$ such that $|N_{F''}(a) \cap C_i|$ is odd. Thus $a$ is adjacent to an odd number of the $c_i$ as desired. The case where $a$ is even w.r.t.~$\psi$ is analogous.

We must now consider the $c_i$ themselves. First note that by an edge counting argument the number of $\psi$-odd vertices in $C_i \cap \psi^{-1}(u)$ has the same parity as the number of odd vertices in $C_i \cap \psi^{-1}(v)$. We will show that $c_i$ is odd w.r.t.~$\psi'$ if this number is odd and even otherwise. Suppose that $x \sim_{G'} \hat{w}$. By construction, a vertex $b \in {\psi'}^{-1}(x)$ is adjacent to $c_i$ in $F'$ if and only if it is adjacent to an odd number of elements of $C_i$ in $F''$. Therefore,
\[N_{F'}(c_i) \cap {\psi'}^{-1}(x) = \triangle_{c \in C_i} \left(N_{F''}(c) \cap \psi^{-1}(x)\right)\]
where $\triangle$ denotes symmetric difference. Recall that since $x \sim_{G'} \hat{w}$ we have that $x$ is adjacent to precisely one of $u$ and $v$ in $G''$. Suppose without loss of generality that $x \sim_{G''} u$. Then
\[N_{F'}(c_i) \cap {\psi'}^{-1}(x) = \triangle_{c \in C_i \cap \psi^{-1}(u)} \left(N_{F''}(c) \cap \psi^{-1}(x)\right).\]
As $\psi$ is an oddomorphism from $F''$ to $G''$, the size of this symmetric difference is odd if $C_i \cap \psi^{-1}(u)$ contains an odd number of $\psi$-odd vertices and is even otherwise. Since this does not depend on the particular $x \sim_{G'} \hat{w}$, we have proven our claim.

Now that we have shown that every vertex of $F'$ is either even or odd with respect to $\psi'$, we must show that ${\psi'}^{-1}(x)$ contains an odd number of odd vertices for all $x \in V(G')$. We have shown that if $x \ne \hat{w}$ then $a \in {\psi'}^{-1}(x)$ is odd w.r.t.~$\psi'$ if and only if it is odd w.r.t.~$\psi$. Therefore ${\psi'}^{-1}(x)$ contains an odd number of odd vertices for any $x \ne \hat{w}$. Now consider ${\psi'}^{-1}(\hat{w}) = \{c_1, \ldots, c_k\}$. We have shown that $c_i$ is odd w.r.t.~$\psi'$ if and only if $C_i \cap \psi^{-1}(u)$ contains an odd number of $\psi$-odd vertices. Since $\psi^{-1}(u)$ contains an odd number of $\psi$-odd vertices, it follows that there are an odd number of $i \in [k]$ such that $C_i \cap \psi^{-1}(u)$ contains an odd number of $\psi$-odd vertices. Therefore an odd number of the $c_i$ are odd w.r.t.~$\psi'$.\qeds

Using the above, we are able to prove the following:

\begin{cor}
Let $\mathcal{F}$ be a minor closed family of graphs and suppose that $G$ is a connected graph such that $G_0 \cong_{\mathcal{F}} G_1$. If $H$ is a connected graph that contains $G$ as a minor, then $H_0 \cong_{\mathcal{F}} H_1$.
\end{cor}
\proof
We prove the contrapositive. Suppose that $H_0 \not\cong_{\mathcal{F}} H_1$, i.e., that there exists a graph $F \in \mathcal{F}$ such that $\hom(F,H_0) \ne \hom(F,H_1)$. By Theorem~\ref{thm:maindualG} this is equivalent to $F$ having an oddomorphism to $H$. Since $G$ is a minor of $H$, by Lemma~\ref{lem:minoroddo} there exists a minor $F'$ of $F$ that has an oddomorphism to $G$ and therefore $\hom(F',G_0) \ne \hom(F',G_1)$. Since $\mathcal{F}$ is a minor closed family, $F' \in \mathcal{F}$ and thus $G_0 \not\cong_{\mathcal{F}} G_1$.\qeds

In the proof of Corollary~\ref{cor:cycleoddos} (see Remark~\ref{rem:bipartite}), we showed that a connected bipartite graph cannot have a weak oddomorphism to an odd cycle. Here we generalize this to show that no bipartite graph admits a weak oddomorphism to any non-bipartite graph, and this allows us to show that if $G$ is a connected non-bipartite graph then $G_0$ and $G_1$ have isomorphic bipartite double covers\footnote{A direct proof of this latter fact was first found jointly with Eva Rotenberg.}. The bipartite double cover of a graph $H$ is the graph $K_2 \times H$.

\begin{lemma}
Let $G$ be any non-bipartite graph. If $F$ is bipartite, then $F$ does not admit a weak oddomorphism to $G$. As a consequence, if $G$ is a connected non-bipartite graph, then $K_2 \times G_0 \cong K_2 \times G_1$.
\end{lemma}
\proof
Since $G$ is non-bipartite, it contains a subgraph $G'$ isomorphic to an odd cycle. Suppose that $F$ is bipartite and has a weak oddomorphism to $G$. Then by Lemma~\ref{lem:suboddo}, $F$ has a subgraph $F'$ (which is also bipartite) that has a weak oddomorphism to the odd cycle $G'$. Furthermore, by Lemma~\ref{lem:oddoconnected} $F'$ must have a connected component $F''$ (which is also bipartite) that has a weak oddomorphism to the odd cycle $G'$. But this contradicts Remark~\ref{rem:bipartite}, and thus we have that $F$ could not have had a weak oddomorphism to $G$, as desired.

Now consider a connected non-bipartite graph $G$. By the above we have that any bipartite graph $F$ does not have a weak oddomorphism to $G$, and thus by Theorem~\ref{thm:maindualG} we have that $\hom(F,G_0) = \hom(F,G_1)$ for all bipartite graphs $G$. Using known properties of the categorical product of graphs, we have that for bipartite $F$:
\[\hom(F,K_2 \times G_0) = \hom(F,K_2)\hom(F,G_0) = \hom(F,K_2)\hom(F,G_1) = \hom(F,K_2 \times G_1).\]
On the other hand, if $F$ is not bipartite, then $\hom(F,K_2 \times G_0) = \hom(F,K_2 \times G_1) = 0$ since both $K_2 \times G_0$ and $K_2 \times G_1$ are bipartite. Therefore, $\hom(F,K_2 \times G_0) = \hom(F,K_2 \times G_1)$ for all graphs $F$ and thus by Lovasz' theorem we have that $K_2 \times G_0 \cong K_2 \times G_1$.\qeds

\subsection{Cycles, paths, and induced stars}\label{sec:cycles}

We have seen in Lemma~\ref{lem:degree} that if $G$ contains a vertex of degree $d$, then any graph $F$ that has a (weak) oddomorphism to $G$ must contain a vertex of degree at least $d$. Here we will prove an analogous result for (chordless) cycles.

To prove this result we first introduce the following notion:

\begin{definition}
Let $C_k$ be the $k$-cycle graph with vertex set $\{0,1, \ldots, k-1\}$ such that $i \sim i+1 \ \mathrm{mod} \ k$, and let $\psi$ be a homomorphism from a graph $F$ to $C_k$. If $ab \in E(F)$, then we define the $\psi$-length of the \emph{ordered} edge $ab$, denoted $\len_\psi(ab)$, as
\begin{equation}
\len_\psi(ab) = \begin{cases}+1 & \text{if } \psi(b) \equiv \psi(a) + 1 \ \mathrm{mod} \ k \\ -1 & \text{if } \psi(b) \equiv \psi(a) - 1 \ \mathrm{mod} \ k\end{cases}
\end{equation}
Note that as $\psi$ is a homomorphism, one of the above cases always holds. More generally, for a walk $W = a_0, \ldots, a_\ell$ in $F$, we define the $\psi$-length of $W$, denoted $\len_\psi(W)$, as
\begin{equation}
    \len_\psi(W) = \sum_{i=1}^\ell \len_\psi(a_{i-1}a_i).
\end{equation}
Note that the values of $\len_\psi$ and the sum above are not taken modulo $k$. Thus the $\psi$-length of a walk can be any integer (even negative).
\end{definition}

We now make some simple observations:
\begin{enumerate}
    \item If $W_1$ and $W_2$ are walks in $F$ such that the last vertex of $W_1$ is the first vertex of $W_2$, then we denote the concatenation of these walks as $W_1W_2$. It is clear from the definition that
    \begin{equation}\label{eq:psilen1}
        \len_\psi(W_1W_2) = \len_\psi(W_1)+\len_\psi(W_2).
    \end{equation}
    \item If $W = a_0, \ldots, a_\ell$ is a walk in $F$ and $W' = a_\ell, \ldots, a_0$ is the \emph{reversal} of $W$, then
    \begin{equation}\label{eq:psilen2}
        \len_\psi(W') = -\len_\psi(W).
    \end{equation}
    \item If $W = a_0, \ldots, a_\ell$ is a walk in $F$, then 
    \begin{equation}\label{eq:psilen3}
        \len_\psi(W) \equiv r \ \mathrm{mod} \ k \ \Leftrightarrow \ \psi(a_\ell) \equiv \psi(a_0) + r \ \mathrm{mod} \ k.
    \end{equation}
    In particular, $\len_\psi(W)$ is a multiple of $k$ if and only if $\psi(a_0) = \psi(a_\ell)$.
\end{enumerate}

From the latter observation we see that if $W$ is a \emph{closed} walk (i.e., the first and last vertices of $W$ are the same), then $\len_\psi(W) = mk$ for some integer $m$. We refer to this number $m$ as the \emph{$\psi$-winding number} (or simply winding number) of the closed walk $W$. Note that this number can be any integer. A cycle is a special type of closed walk and we will use our notion of winding number to prove that if $F$ has a weak oddomorphism to $C_k$ then it must contain a chordless cycle of length at least and of the same parity as $k$. We do this by proving that there is a chordless cycle with odd winding number. For this we will consider \emph{tours}: closed walks that do not repeat edges. Given a tour $W = a_0, a_1, \ldots, a_\ell = a_0$ in a graph $F$ with homomorphism $\psi$ to $C_k$, we denote by $\vec{E}(W)$ the set $\{(a_j,a_{j+1}) : j = 0, \ldots, \ell-1\}$. We also define for each $i \in V(C_k)$
\begin{align*}
    W^+_i &= \{(a,b) \in \vec{E}(W) : \psi(a) = i, \ \psi(b) = i+1\} \\
    W^-_i &= \{(a,b) \in \vec{E}(W) : \psi(a) = i+1, \ \psi(b) = i\},
\end{align*}
where addition is modulo $k$. We first prove the following lemma:

\begin{lemma}\label{lem:windingnum}
Let $F$ be a graph with a homomorphism $\psi$ to $C_k$. If $W$ is a tour in $F$ then the $\psi$-winding number of $W$ is equal to $|W^+_i|-|W^-_i|$ for all $i \in \{0, \ldots, k-1\}$.
\end{lemma}
\proof
We first show that $|W^+_i|-|W^-_i|$ is independent of $i$ by showing that $|W^+_i|-|W^-_i| = |W^+_{i-1}|-|W^-_{i-1}|$ for arbitrary $i$. Consider the digraph $D$ with $V(D) = V(F)$ whose arcs are the elements of $\vec{E}(W)$. Since $W$ is a tour, the in-degree of any vertex of $D$ is equal to its out-degree. The sum of the out-degrees of the vertices in $\psi^{-1}(i)$ is equal to $|W^+_i|+|W^-_{i-1}|$ whereas the sum of the in-degrees is equal to $|W^-_i|+|W^+_{i-1}|$. Therefore $|W^+_i|+|W^-_{i-1}| = |W^-_i|+|W^+_{i-1}|$ which is equivalent to the desired equality. Thus $|W^+_i|-|W^-_i|$ is some constant $m$ for all $i \in V(C_k)$.

By definition, the $\psi$-length of $W$ is equal to
\[\sum_{i \in V(C_k)} |W^+_i|-|W^-_i|\]
which is equal to $mk$ by the above. Therefore the winding number of $W$ is $m = |W^+_i|-|W^-_i|$.\qeds

We remark that the above lemma holds more generally for closed walks via the same proof, only one needs to take into account the repetition of edges (so $W^\pm_i$ will be a multiset and the digraph $D$ will have multiple edges).

For a tour $W = a_0, \ldots, a_\ell = a_0$ in a graph $F$, let $E(W)$ be the edges of $F$ of the form $a_ja_{j+1}$ for some $j = 0, \ldots, \ell-1$, i.e., $E(W)$ is an undirected version of $\vec{E}(W)$. We have the following:

\begin{lemma}\label{lem:eventour}
Let $F$ be a graph with an oddomorphism $\psi$ to $C_k$, and let $W$ be a tour in $F$ with even winding number. Let $F'$ be the graph with $V(F') = V(F)$ and $E(F') = E(F) \setminus E(W)$. Then $\psi|_{F'}$ is an oddomorphism from $F'$ to $C_k$.
\end{lemma}
\proof
We first show that every vertex of $F'$ is even or odd with respect to $\psi|_{F'}$. Let $a \in \psi^{-1}(i)$. Since $W$ is a tour, the number of edges in $E(W)$ incident to $a$ is even. Thus the number of edges of $E(W)$ incident to $a$ and going to a vertex of $\psi^{-1}(i+1)$ has the same parity as the number of edges of $E(W)$ incident to $a$ and going to $\psi^{-1}(i-1)$. Therefore, after removing the edges of $E(W)$ from $F$ the vertex $a$ has the same parity of incident edges going to $\psi^{-1}(i+1)$ as those going to $\psi^{-1}(i-1)$. This proves the above claim.

Now we must show that there are an odd number of odd vertices with respect to $\psi|_{F'}$ in each fibre. Since every vertex is either odd or even by the above, it suffices to show that there is an odd number of edges of $F'$ between $\psi^{-1}(i)$ and $\psi^{-1}(i+1)$ for a given $i \in V(C_k)$. By Lemma~\ref{lem:windingnum} and the assumption that $W$ has even winding number, the value of $|W^+_i|-|W^-_i|$ is even and therefore $|W^+_i|+|W^-_i|$ is even. The latter is just the number of edges of $E(W)$ between $\psi^{-1}(i)$ and $\psi^{-1}(i+1)$. Since $E(F') = E(F) \setminus E(W)$, we have that the parity of the number of edges between $\psi^{-1}(i)$ and $\psi^{-1}(i+1)$ is the same for $F'$ and $F$. The parity is odd for the latter since $\psi$ is an oddomorphism and so we are done.\qeds

The following lemma is due to Carsten Thomassen.

\begin{lemma}\label{lem:tour}
Let $F$ be a graph with an oddomorphism $\psi$ to $C_k$. Then $F$ contains a tour of nonzero length.
\end{lemma}
\proof
Pick any \emph{odd} vertex $a_0 \in \psi^{-1}(0)$. We will describe how to construct a tour $W$ starting at $a_0$. We will construct the walk one edge at a time and during the construction we may update the ``status" of some vertices. Initially the status of any vertex in $F$ is odd/even if it is odd/even with respect to $\psi$. We begin the walk by moving along an edge from $a_0$ to some vertex $a_1 \in \psi^{-1}(1)$ (this exists since $a_0$ is odd) and updating the status of $a_0$ to even. If $a_1$ is odd then we ``continue in the same direction" along an edge from $a_1$ to some $a_2 \in \psi^{-1}(2)$ and update the status of $a_1$ to even. If $a_1$ is even then we ``turn around" and move along an edge (that we have not used yet) to a vertex $a_2 \in \psi^{-1}(0)$ and the status of $a_1$ remains even. We continue in this way, continuing in the same direction when we reach a vertex whose status is odd and then updating its status to even, but turning around whenever we reach a vertex whose status is even and not changing its status. At each step we use an edge that we have not yet used in our walk. We do this until we reach $a_0$ along an edge between $\psi^{-1}(k-1)$ and $\psi^{-1}(0)$ and then we stop. This is possible because whenever we reach a vertex $a$ whose status is odd, then $a$ is odd with respect to $\psi$ and this is the first time we have visited $a$ in our walk (or else we would have updated its status to even). Thus there must be an edge incident to $a$ that allows us to continue in the same direction we were walking by the definition of oddomorphism. When we reach a vertex $a \in \psi^{-1}(i)$ from $\psi^{-1}(i \pm 1)$ whose status was even then either $a$ is even w.r.t.~$\psi$ and we have so far used an odd number of edges incident to $a$ between $\psi^{-1}(i)$ and $\psi^{-1}(i \pm 1)$ or $a$ is odd w.r.t.~$\psi$ and we have used an even number of such edges. In either case there must be another edge incident to $a$ that we can use to turn around. The only exception to this is when we reach $a_0$ along an edge between $\psi^{-1}(k-1)$ and $\psi^{-1}(0)$ which is when we stop.\qeds

Now we are able to prove our main result about oddomorphisms to cycles:

\begin{theorem}\label{thm:chordlesscycles}
Suppose that $F$ has a weak oddomorphism $\psi$ to the cycle $C_k$ of length $k$. Then $F$ contains a chordless (i.e., induced) cycle with odd $\psi$-winding number. It follows that this cycle has length at least and the same parity as $k$. Moreover, the cycle contains an edge between $\psi^{-1}(u)$ and $\psi^{-1}(v)$ for each edge $uv$ of $C_k$.
\end{theorem}
\proof\hspace{-.18cm}\footnote{In the first part of this proof, it is shown that $F$ must contain a closed walk of odd winding number. This is proven by induction using Lemma~\ref{lem:tour}, and this approach is due to Carsten Thomassen. The original proof of the author was longer and more complicated. The part of the proof after showing that $F$ must contain a closed walk of odd winding number remains unchanged.} We first show that if $F$ has a weak oddomorphism to $C_k$ then it must contain a tour of odd winding number. Note that it suffices to prove this for $F$ having an oddomorphism to $C_k$, so let $\psi$ be such an oddomorphism.

The proof is by induction on the number of edges. Since there are an odd number of edges of $F$ between $\psi^{-1}(i)$ and $\psi^{-1}(i+1)$ (mod $k$) for all $i \in \{0, \ldots, k-1\}$, the graph $F$ must contain at least $k$ edges. It is easy to see that if $F$ contains precisely $k$ edges then each $\psi^{-1}(i)$ contains precisely one odd vertex and $F$ is a $k$-cycle formed from the odd vertices plus possibly some isolated even vertices. The $k$-cycle of odd vertices then clearly has odd winding number and we are done.

Now suppose that $F$ has more than $k$ edges. By Lemma~\ref{lem:tour}, the graph $F$ contains a tour $W$ of nonzero length (i.e., $|E(W)| > 0$). If $W$ has odd winding number then we are done. Otherwise $W$ has even winding number and by Lemma~\ref{lem:eventour} the graph $F'$ obtained from $F$ by removing the edges in $E(W)$ has an oddomorphism to $C_k$. By induction $F'$ contains a tour of odd winding number and thus so does $F$.

So we have shown that if $F$ has a weak oddomorphism to $C_k$, then $F$ contains a tour (and thus a closed walk) of odd winding number. Now let $W = a_0, \ldots, a_\ell = a_0$ be a closed walk in $F$ of odd winding number with shortest possible length (actual length, not $\psi$-length). We will show that $W$ must be a chordless cycle.

First we show that $W$ is a cycle. If not then there are two indices $i<j \in \{0, \ldots, \ell-1\}$ such that $a_i = a_j$. Let $W_1 = a_0, \ldots, a_i$, $W_2 = a_i, \ldots, a_j$, and $W_3 = a_j, \ldots, a_\ell$. Then
\[\len_\psi(W) = \len_\psi(W_1) + \len_\psi(W_2) + \len_\psi(W_3) = \len_\psi(W_1W_3) + \len_\psi(W_2).\]
Note that both $W_1W_3$ and $W_2$ are closed walks. Since $W$ has odd winding number, one of $W_1W_3$, $W_2$ has odd winding number and is shorter than $W$, a contradiction. Therefore $W$ is a cycle.

Now suppose that $W$ has a chord $a_ia_j$ for some $i<j \in \{0, \ldots, \ell-1\}$. Let $W_1 = a_0, \ldots, a_i$, $W_2 = a_i, \ldots, a_j$, $W_3 = {a_j, \ldots, a_\ell}$, $\hat{W} = a_i,a_j$, and $\hat{W}^{\text{rev}} = a_j,a_i$. Then both $W_1\hat{W}W_3$ and $W_2\hat{W}^{\text{rev}}$ are closed walks and
\begin{align*}
    \len_\psi(W_1\hat{W}W_3) + \len_\psi(W_2\hat{W}^{\text{rev}}) &= \len_\psi(W_1) + \len_\psi(W_2) + \len_\psi(W_3) + \len_\psi(\hat{W}) + \len_\psi(\hat{W}^{\text{rev}}) \\
    &= \len_\psi(W_1) + \len_\psi(W_2) + \len_\psi(W_3) \\
    &= \len_\psi(W)
\end{align*}
since $\len_\psi(\hat{W}^{\text{rev}}) = -\len_\psi(\hat{W})$. It follows that one of $W_1\hat{W}W_3$ and $W_2\hat{W}^{\text{rev}}$ has odd winding number. Moreover, both $W_2$ and $W_3W_1$ have (actual) length at least two since $a_ia_j$ is a chord, and thus they both have length at most the length of $W$ minus two (since the sum of their lengths is the length of $W$. Therefore both of $W_1\hat{W}W_3$ and $W_2\hat{W}^{\text{rev}}$ have length strictly less than $W$ and so we have found a closed walk of odd winding number with length strictly smaller than $W$, a contradiction. So $W$ must have been a chordless cycle of odd winding number.

Now we must show that a closed walk of odd winding number has the other properties stated in the lemma. First note that if $W$ is a closed walk that does not use any edges between $\psi^{-1}(i)$ and $\psi^{-1}(i+1)$ for some $i \in V(C_k)$, then its winding number must be zero by Lemma~\ref{lem:windingnum}. Thus a closed walk of odd winding number must contain an edge between $\psi^{-1}(u)$ and $\psi^{-1}(v)$ for every edge $uv$ of $C_k$. From this it follows that a closed walk of odd winding number must have length at least $k$. Finally, note that every edge used in a walk contributes either $1$ or $-1$ to its $\psi$-length. Thus the parity of the $\psi$-length of a walk is equal to the parity of its actual length. Since a closed walk of odd winding number has $\psi$-length $mk$ for some odd $m$, we have that its length has the same parity as $k$.\qeds

Using the above theorem we can show that if $F$ has a weak oddomorphism to a graph $G$ which contains a (chordless) cycle of length $k$, then $F$ must also have a (chordless) cycle of length at least $k$. In fact we can prove something stronger, but first we need some terminology.

Let $R$ be a graph and let $c : V(R) \to \{0,1\}$ and $\ell : V(R) \to \{i \in \mathbb{N} : i \ge 3\}$. We say that a graph $F$ contains an \emph{$(R,c,\ell)$ cycle structure} if $F$ contains disctinct cycles $C_u$ for $u \in V(R)$ such that $C_u$ has length at least and of the same parity as $\ell(u)$, $C_u$ is chordless if $c(u) = 1$, and cycles $C_u$ and $C_v$ are vertex disjoint if $u \sim_R v$.

\begin{lemma}\label{lem:cyclestructures}
Let $R$ be a graph and let $c : V(R) \to \{0,1\}$ and $\ell : V(R) \to \{i \in \mathbb{N} : i \ge 3\}$. Let $G$ be a graph that contains an $(R,c,\ell)$ cycle structure. If $F$ has a weak oddomorphism to $G$, then $F$ contains an $(R,c, \ell)$ cycle structure.
\end{lemma}
\proof
Let $C_u$ for $u \in V(R)$ be the cycles of $G$ forming an $(R,c,\ell)$ cycle structure. Let $\psi$ be a weak oddomorphism from $F$ to $G$. For $u \in V(R)$ let $F_u = \psi^{-1}(C_u)$. By Lemma~\ref{lem:suboddo}, $\psi|_{V(C_u)}$ is a weak oddomorphism from $F_u$ to $C_u$ for all $u \in V(R)$. Thus by Theorem~\ref{thm:chordlesscycles}, $F_u$ contains a chordless cycle $D_u$ of length at least and the same parity as $\ell(u)$ for $u \in V(R)$. Obviously the cycle $D_u$ of $F_u$ will still be a cycle of length at least and the same parity as $\ell(u)$ in $F$, but we must show that it is chordless in $F$ when $c(u) = 1$.

Let $u \in V(R)$ be such that $c(u) = 1$ and let $x_1, \ldots, x_{k_u}$ be the vertices of $C_u$ such that $x_i \sim x_{i +1}$ with indices taken modulo $k_u$. Thus $V(F_u) = \cup_{i=1}^{k_u} \psi^{-1}(x_i)$. If $e \in E(F)$ is a chord of $D_u$, then each end of $e$ is contained in $V(D_u) \subseteq V(F_u)$. Therefore the ends $a$ and $b$ of $e$ must be contained in $\psi^{-1}(x_s)$ and $\psi^{-1}(x_t)$ for some $s,t \in [k_u]$. Since $\psi$ is a homomorphism, this implies that there is an edge between $x_s$ and $x_t$. Since $C_u$ is chordless in $G$, this implies that $x_sx_t$ is an edge of $C_u$. But then $e$ is contained in $V(F_u)$ and this contradicts the fact that $D_u$ is chordless in $F_u$.

We must show that the cycles $D_u$ for $u \in V(R)$ are distinct. Let $u$ and $v$ be any two distinct vertices in $R$. Since the cycles $C_u$ and $C_v$ are distinct by assumption, there is an edge $e$ of $C_u$ that is not contained in $C_v$. Let $e = xy \in V(G)$. By Theorem~\ref{thm:chordlesscycles} the cycle $D_u$ must contain an edge between $\psi^{-1}(x)$ and $\psi^{-1}(y)$. But $D_v$ cannot contain any such edge since $e = xy$ is not contained in $C_v$. Thus the cycles $D_u$ for $u \in V(R)$ are distinct.

Finally, suppose that $u \sim_R v$. By definition of $(R,c,\ell)$ cycle structures this means that $C_u$ and $C_v$ are vertex disjoint. This further implies the sets $\psi^{-1}(V(C_u))$ and $\psi^{-1}(V(C_v))$ are disjoint and these two sets respectively contain $V(D_u)$ and $V(D_v)$ and thus $D_u$ and $D_v$ are vertex disjoint.\qeds

The length of a longest cycle in $G$ is known as the circumference of $G$. We have several corollaries to the above lemma:

\begin{cor}\label{cor:circumference}
If $G$ has circumference $k$ and $F$ has a weak oddomorphism to $G$, then $F$ has circumference at least $k$.
\end{cor}

\begin{cor}\label{cor:forests}
If $G$ is not a forest and $F$ has a weak oddomorphism to $G$, then $F$ is not a forest.
\end{cor}

\begin{cor}\label{cor:oddholes}
If $G$ contains an odd hole (induced odd cycle of length at least 5) and $F$ has a weak oddomorphism to $G$, then $F$ contains an odd hole.
\end{cor}

We remark that the analog of Theorem~\ref{thm:chordlesscycles} for induced paths can be proven easily in comparison and from this one can prove a path version of Lemma~\ref{lem:cyclestructures}. Next we consider induced stars. Lemma~\ref{lem:degree} can be viewed as saying that if $F$ has a weak oddomorphism to $G$ and $G$ contains $K_{1,d}$ as a subgraph, then $F$ must also contain $K_{1,d}$ as a subgraph. We now prove the same thing but for induced subgraphs.

\begin{lemma}\label{lem:inducedstar}
Suppose that $G$ contains $K_{1,d}$ as an induced subgraph. If $F$ has a weak oddomorphism to $G$, then $F$ also contains $K_{1,d}$ as an induced subgraph.
\end{lemma}
\proof
Let $\psi$ be a weak oddomorphism from $F$ to $G$. Suppose that $G'$ is an induced subgraph of $G$ isomorphic to $K_{1,d}$ and let $u_0, u_1, \ldots, u_d$ be the vertices of $G'$ such that $u_0 \sim u_i$ for $i = 1, \ldots, d$. Let $\hat{F}$ be a subgraph of $F$ such that $\psi|_{V(\hat{F})}$ is an oddomorphism from $\hat{F}$ to $G$. Let $a_0 \in V(\hat{F})$ be an odd vertex w.r.t.~$\psi|_{\hat{F}}$ contained $\psi^{-1}(u_0)$. Since $a_0$ is odd, it must have a neighbor $a_i \in \psi^{-1}(u_i)$ for all $i = 1, \ldots, d$. Furthermore, since $G'$ is induced there are no edges between any two distinct $u_i$ with $i \ge 1$, and since $\psi$ is a homomorphism this implies there are no edges between the sets $\psi^{-1}(u_i)$ and $\psi^{-1}(u_j)$ for $1 \le i < j \le d$. Therefore there is no edge between $a_i$ and $a_j$ for $1 \le i < j \le d$ and thus the set $\{a_0, \ldots, a_d\}$ induces a $K_{1,d}$ in $F$.\qeds

We remark that we can also prove an analog of Lemma~\ref{lem:cyclestructures} showing that if $F$ has a weak oddomorphism to a graph with several distinct/vertex-disjoint stars of various orders (which may be induced or not) then $F$ must contain a corresponding collection of stars. But the proof would not be significantly different from that of Lemma~\ref{lem:cyclestructures}, so we will not bother.

Another structure to consider would be cuts, and it is straightforward to see that if $G$ has a cut of size $k$ and $F$ has a weak oddomorphism to $G$, then $F$ must have a cut of size at least $k$.

\section{Homomorphism distinguishing closed families}\label{sec:hdclosed}

Here we provide some machinery for producing homomorphism distinguishing closed families of graphs. We begin by showing the the intersection of an arbitrary collection of closed families is closed.

\begin{lemma}\label{lem:intersection}
Let $\mathbb{A}$ be some index set and let $\F_\alpha$ be a homomorphism distinguishing closed family for all $\alpha \in \mathbb{A}$. Then
\[\mathcal{F} = \bigcap_{\alpha \in \mathbb{A}} \F_\alpha\]
is homomorphism distinguishing closed.
\end{lemma}
\proof
Suppose that $F \notin \F$. Then there exists $\alpha \in \mathbb{A}$ such that $F \notin \F_\alpha$. Since $\F_\alpha$ is closed, there exist graphs $G$ and $H$ such that $G \cong_{\F_\alpha} H$ but $\hom(F,G) \ne \hom(F,H)$. However, $\F \subseteq \F_\alpha$ and therefore we also have $G \cong_\F H$.\qeds.

We remark that the same does not hold for unions, since the union of two homomorphism distinguishing closed families need not even be closed under disjoint unions. Though one can ask whether the disjoint union closure of the union of homomorphism distinguishing closed families is homomorphism distinguishing closed.

In the following theorem we show that families of graphs that are closed under images of weak oddomorphisms, restriction to connected components, and disjoint unions are homomorphism distinguishing closed. This result can be applied to prove that many families of graphs $\mathcal{F}$ are homomorphism distinguishing closed, and thus in particular the relation $\cong_{\mathcal{F}}$ is not isomorphism. %The conditions in the following theorem may seem odd at first glance, but they can be applied to prove that many families of graphs $\mathcal{F}$ are homomorphism distinguishing closed, and thus in particular the relation $\cong_{\mathcal{F}}$ is not isomorphism. %Here, for a family $\mathcal{F}$ of (isomorphism classes of) graphs we denote by $\overline{\mathcal{F}}$ the family of (isomorphism classes of) graphs not contained in $\mathcal{F}$.

\begin{theorem}\label{thm:homclosed1}
Let $\mathcal{F}$ be a family of graphs such that
\begin{enumerate}
    \item if $F \in \mathcal{F}$ and $F \woto G$, then $G \in \mathcal{F}$, and
    \item $F \in \mathcal{F}$ if and only if every connected component of $F$ is an element of $\F$, i.e., $\mathcal{F}$ is closed under disjoint unions and restriction to connected components.%the graphs of $\F$ are precisely the disjoint unions of connected graphs from $\F$.
\end{enumerate}
Then the family $\mathcal{F}$ is homomorphism distinguishing closed.
\end{theorem}
\proof
Let $G \notin \mathcal{F}$. We must show that there are graphs $H$ and $H'$ such that $H \cong_{\mathcal{F}} H'$ but $\hom(G,H) \ne \hom(G,H')$. Note that Condition (2) is equivalent to saying that a graph is not in $\F$ if and only if it contains a connected component that is not in $\F$. Therefore there is a connected component $G'$ of $G$ such that $G' \notin \F$. Let $H$ be the disjoint union of $G'_0$ and a clique $K_r$ such that $\hom(G,K_r) > 0$, and let $H'$ be the disjoint union of $G'_1$ and $K_r$.

Suppose that $F \in \F$ is connected. By Condition (1), we have that $F \not\woto G'$ and therefore $\hom(F,G'_0) = \hom(F,G'_1)$ by Theorem~\ref{thm:maindualG}. Since $F$ is connected,
\[\hom(F,H) = \hom(F,G'_0) + \hom(F,K_r) = \hom(F,G'_1) + \hom(F,K_r) = \hom(F,H').\]
Since the equality holds for all connected $F \in \F$, by Condition (2) it holds for all graphs in $\F$ and thus $H \cong_\F H'$.

Now we must show that $\hom(G,G'_0) \ne \hom(G, G'_1)$. However this part of the proof is identical to how we showed that $\hom(\hat{F},H) \ne \hom(\hat{F},H')$ in the proof of Theorem~\ref{thm:closed}, and thus we omit it.\qeds.

\begin{remark}
We note that if $\P$ is the property of satisfying conditions (1) and (2) from Theorem~\ref{thm:homclosed1}, then $\P$ satisfies the conditions of Theorem~\ref{thm:absconjectures} and thus the four statements in that theorem are equivalent for union-closed $\P$-families. However we do not need this for any of our results so we omit the (easy) proof.
\end{remark}

The above allows us to produce many homomorphism distinguishing closed classes of graphs via the following theorem.

\begin{theorem}\label{thm:homclosed2}
Let $\mathcal{G}$ be any family of connected graphs. If
\[\mathcal{F} = \{F : \forall \text{ connected components } F' \text{ of } F, \ \forall \ G \in \mathcal{G},  F' \not\woto G\},\]
then $\mathcal{F}$ is homomorphism distinguishing closed.
\end{theorem}
\proof
Note that $\mathcal{F}$ satisfies Condition (2) of Theorem~\ref{thm:homclosed1} by definition. Thus it is only left to check Condition (1).

Suppose that $F \woto H \notin \F$. By definition of $\F$, we have that $H$ contains a connected component $H'$ that admits a weak oddomorphism to some $G \in \mathcal{G}$. By composition, we have that there is a weak oddomorphism $\psi$ from $F$ to $G$. Since $G$ is connected, by Theorem~\ref{thm:maindualG} we have that there is a connected subgraph $F'$ of $F$ such that $\psi|_{V(F')}$ is an oddomorphism from $F'$ to $G$. If $F''$ is the connected component of $F$ that contains $F'$, then $\psi|_{V(F'')}$ is a weak oddomorphism from $F''$ to $G$, and therefore $F \notin \F$. Thus Condition~(1) of Theorem~\ref{thm:homclosed1} holds and we are done.\qeds

\begin{remark}
The above proof of Theorem~\ref{thm:homclosed2} shows that any family $\F$ that can be described as in the theorem statement meets the conditions of Theorem~\ref{thm:homclosed1}. We remark that this can be reversed, i.e., any family $\F$ that satisfies the conditions of Thereom~\ref{thm:homclosed1} can be described as in Theorem~\ref{thm:homclosed2} by taking $\mathcal{G}$ to be the family of connected graphs not in $\F$. A natural question to ask is whether it is always possible to use a finite family $\mathcal{G}$ to describe families $\F$ meeting the conditions of Theorem~\ref{thm:homclosed1}, similar to how minor-closed families can always be described by a finite family of forbidden minors. However, we will show in Section~\ref{sec:uncountability} that there are uncountably many such families and therefore one cannot always use a finite $\mathcal{G}$.
\end{remark}
Applying the above theorem along with our results from Section~\ref{sec:cycles}, we obtain several homomorphism distiguishing closed families. 

\begin{cor}\label{cor:cyclestructure}
Let $R$ be a graph and let $c : V(R) \to \{0,1\}$ and $\ell : V(R) \to \{i \in \mathbb{N} : i \ge 3\}$. Let $\F$ be the family of graphs $F$ such that no connected component of $F$ contains an $(R,c,\ell)$ cycle structure. Then $\F$ is homomorphism distinguishing closed.
\end{cor}
\proof
Let $\mathcal{G}$ be the family of connected graphs that contain an $(R,c,\ell)$ cycle structure. We will show that $\F$ is the family of graphs $F$ such that no connected component $F'$ of $F$ has a weak oddomorphism to an element of $G$.

Suppose that $F$ has a connected component $F'$ such that $F' \woto G \in \mathcal{G}$. Since $G$ contains an $(R,c,\ell)$ cycle structure, so does $F'$ by Lemma~\ref{lem:cyclestructures}. Therefore $F \notin \F$.

On the other hand, if $F \notin \F$, then some connected component $F'$ of $F$ contains a $(R,c,\ell)$ cycle structure, i.e., $F' \in \mathcal{G}$. But $F'$ admits a weak oddomorphism to itself, and therefore $F$ contains a connected component that has a weak oddomorphism to some element of $\mathcal{G}$.\qeds

Recall that the \emph{circumference} of a graph $G$ is the maximum length of a cycle in $G$. Using Theorem~\ref{thm:homclosed2} and Corollary~\ref{cor:cyclestructure} we can show that the family of graphs of some bounded circumference is homomorphism distinguishing closed.

\begin{cor}
For any $k \in \mathbb{N}$, $k \ge 3$, the family of graphs with circumference at most $k$ is homomorphism distinguishing closed.
\end{cor}
\proof
Let $\mathcal{G}$ be the family of all \emph{connected} graphs with circumference strictly greater than $k$. It follows from Corollary~\ref{cor:circumference} that the family of graphs with circumference at most $k$ is equal to
\[\mathcal{F} = \{F : \forall \text{ connected components } F' \text{ of } F, \ \forall \ G \in \mathcal{G},  F' \not\woto G\}.\]
Therefore by Theorem~\ref{thm:homclosed2} the family of graphs with circumference at most $k$ is homomorphism distiguishing closed.\qeds

We also have the following:

\begin{cor}
The family of all forests is homomorphism distinguishing closed.
\end{cor}
\proof
This family clearly satisfies Condition~(2) of Theorem~\ref{thm:homclosed1}, and Corollary~\ref{cor:forests} shows that it also satisfies Condition~(1).\qeds

Notably, graphs $G$ and $H$ are homomorphism indistiguishable over the family of all trees if and only if they are indistinguishable by the 1-dimensional Weisfeiler-Leman algorithm, also known as color refinement~\cite{DGR,dvorak}. Thus the above says that indistinguishability of $G$ and $H$ by color refinement does not imply that $\hom(F,G) = \hom(F,H)$ for any graph $F$ that is not a forest. We can obtain the same result for 2-dimensional Weisfeiler-Leman, which is equivalent to homomorphism indistiguishability over the family of graphs of treewidth at most 2~\cite{dvorak,DGR}.

\begin{cor}
The family of graphs with treewidth at most 2 is homomorphism distinguishing closed.
\end{cor}
\proof
Let $\F$ be the family of graphs of treewidth 2. These are precisely the graphs with no $K_4$ minor. It is easy to see that Condition~(2) of Theorem~\ref{thm:homclosed1} is satisfied by $\F$. Suppose that $F \in \F$ is such that $F \woto G \notin \F$. This means that $G$ has a $K_4$ minor, and therefore by Lemma~\ref{lem:minoroddo} the graph $F$ has a minor $F'$ that admits an oddomorphism to $K_4$. This implies that $F'$ has a different number of homomorphisms to $G_0$ and $G_1$ for $G = K_4$. However, in this case $G_0$ and $G_1$ are the $4 \times 4$ rook graph and the Shrikhande graph respectively. These are two strongly regular graphs with the same parameters and it is known that such pairs are not distinguished by the 2-dimensional Weisfeiler-Leman method. Since $\hom(F',G_0) \ne \hom(F',G_1)$, we see that $F'$ must have treewidth greater than 2 and thus must contain a $K_4$ minor. Since $F'$ is a minor of $F$, we obtain that $F$ has a $K_4$ minor, i.e., $F \notin \F$. So we have shown that $\F$ satisfies Condition~(1) of Theorem~\ref{thm:homclosed1} and we are done.\qeds

Using Corollary~\ref{cor:oddholes} we obtain the following:

\begin{cor}
The family of graphs containing no odd holes is homomorphism distinguishing closed.
\end{cor}

Since perfect graphs do not contain odd holes, we have the following:

\begin{cor}
Homomorphism indistinguishability over the family of perfect graphs is not equivalent to isomorphism.
\end{cor}

%It is perhaps an interesting question whether perfect graphs are homomorphism distinguishing closed, or if they are closed under weak oddomorphisms.

Lastly we can apply Lemma~\ref{lem:inducedstar} to obtain the following:

\begin{cor}
Let $\F$ be the family of graphs not containing an induced copy of $K_{1,d}$. Then $\F$ is homomorphism distinguishing closed and therefore the relation $\cong_\F$ is not isomorphism.
\end{cor}

Notably, the class of graphs not containing an induced $K_{1,d}$ is much larger than the class of graphs with maximum degree less than $d$.

% Lastly, we show that the family of disjoint unions of cycles (plus $K_1$ and $K_2$ is h.d.~closed). 

% \begin{lemma}\label{lem:cycles}
% Let $\F$ be the family of graphs that are disjoint unions of cycles and $K_1$ and $K_2$. Then $\F$ is homomorphism distinguishing closed.
% \end{lemma}
% \proof
% Suppose that $\F$ is not h.d.~closed, and let $F \notin \F$ be a graph such that $G \cong_\F H$ implies that $\hom(F,G) = \hom(F,H)$. Since $\F$ is closed under restriction to connected components, by Lemma~\ref{lem:findconnected} we may assume that $F$ is connected. If $F$ has maximum degree greater than two, then $F_0$ and $F_1$ have the same number of homomorphisms from every graph in $\F$ by Theorem~\ref{thm:maindualG} and Lemma~\ref{lem:degree}, but of course $\hom(F,F_0) \ne \hom(F,F_1)$, a contradiction. Thus we must have that $F$ has maximum degree at most two, and therefore $F$ is a path of length at least two, since it is not contained in $\F$.

% Suppose that $F$ is the path of 

It is well known that the corresponding homomorphism indistinguishability relation $\cong_\F$ is cospectrality of adjacency matrices.

\begin{lemma}
The family of all forests is homomorphism distinguishing closed.
\end{lemma}

\subsection{Uncountability}\label{sec:uncountability}

Since any homomorphism indistinguishability relation corresponds to some family of finite graphs, the number of such relations can be at most the cardinality of the continuum. It is clear from previous results that there are an infinite number of homomorphism indistinguishability relations. It thus remains to decide whether or not there are countably many homomorphism indistinguishability relations, i.e., whether there are countably many homomorphism distinguishing closed families of graphs. Here we answer this question in the negative by showing that there are uncountably many families meeting the conditions of Theorem~\ref{thm:homclosed1}.

First we prove the following two lemmas.%at a complete graph does not admit a weak oddomorphism to any other graph.

\begin{lemma}\label{lem:cliques}
If $n \in \mathbb{N}$ and $K_n \woto G$, then $G \cong K_n$.
\end{lemma}
\proof
Since a weak oddomorphism is surjective, we have that $|V(G)| \le n$. On the other hand, since an oddomorphism is a homomorphism we must have that $G$ contains a complete subgraph of size $n$. It follows that $G \cong K_n$.\qeds

\begin{lemma}\label{lem:connected2connected}
Let $F$ and $G$ be graphs. If $F \woto G$, then for every connected component $G'$ of $G$, there is a connected component $F'$ of $F$ such that $F' \woto G'$.
\end{lemma}
\proof
Let $\psi$ be a weak oddomorphism from $F$ to $G$ and let $G'$ be a connected component of $G$. For $F' = \psi^{-1}(G')$, the map $\psi|_{V(F')}$ is a weak oddomorphism from $F'$ to $G'$ by Lemma~\ref{lem:suboddo}. Since $\psi$ is a homomorphism, the subgraph $F'$ must be a disjoint union of connected components of $F$. By Lemma~\ref{lem:oddoconnected} there is a connected subgraph $\hat{F}$ of $F'$ such that $\psi|_{V(\hat{F})}$ is an oddomorphism from $\hat{F}$ to $G'$. If $F''$ is the connected component of $F'$ containing $\hat{F}$, then $\psi|_{V(F'')}$ is a weak oddomorphism from $F''$ to $G'$ and we are done.\qeds

\begin{theorem}\label{thm:uncountable}
There are uncountably many homomorphism indistinguishability relations.
\end{theorem}
\proof
Let $S \subseteq \mathbb{N}$, and define $\F_S$ to be the family graphs which are disjoint unions of complete graphs of size some element of $S$, i.e.,
\[\F_S = \{\cup_{i = 1}^k K_{n_i}: k \in \mathbb{N}, \ n_i \in S \ \forall i \in [k]\}.\]
It is clear that each choice of $S \subseteq \mathbb{N}$ results in a distinct familiy $\F_S$ and thus there are uncountably many such families. We now show that $\F_S$ satisfies the conditions of Theorem~\ref{thm:homclosed1}.

It is immediate that $\F_S$ satisfies Condition (2) of Theorem~\ref{thm:homclosed1}, i.e., it is closed under disjoint unions and restrictions to connected components. Thus it is only left to check Condition (1). Suppose that $F \in \F_S$ and $F \woto G$. Let $G'$ be a connected component of $G$. By Lemma~\ref{lem:connected2connected} there is a connected component $F'$ of $F$ such that $F'$ has a weak oddomorphism to $G'$. By definition of $\F_S$, the graph $F'$ is isomorphic to $K_n$ for some $n \in S$ and thus by Lemma~\ref{lem:cliques} we have that $G' \cong K_n$. Therefore $G$ is a disjoint union of complete graphs of size some element of $S$, i.e., $G \in \F_S$. Thus Condition (2) of Theorem~\ref{thm:homclosed1} holds for $\F_S$ and therefore it is homomorphism distinguishing closed.

Since the families $\F_S$ are all homomorphism distinguishing closed and are distinct for distinct choices of $S \subseteq \mathbb{N}$, we have that the corresponding relations $\cong_{\F_S}$ are all distinct. Clearly there are uncountably many choices for $S$ and thus we have proven the theorem.\qeds

\section{Allowing loops}\label{sec:loops}

When considering the \emph{existence} of homomorphisms between graphs, loops lead to trivialities since any graph has a homomorphism to any graph containing a loop and no graph with a loop has a homomorphism to any graph without one. However, when counting homomorphisms between graphs, it is not unusual to allow loops. So far we have only considered loopless graphs, as this seemed the more natural setting for our results. But in this section we will consider what happens when loops are allowed.

A natural question to consider here is how allowing loops affects homomorphism distinguishing closedness. For instance, if $\F$ is a family of loopless graphs that is h.d.~closed in the loopless setting, then does it remain h.d.~closed when loops are allowed? In other words, is there some graph $F \notin \F$ such that $G \cong_\F H$ implies that $\hom(F,G) = \hom(F,H)$? The graph $F$ must necessarily have loops (or else $\F$ could not have been h.d.~closed in the no loops allowed setting), and this means that $G \cong_\F H$ implies that either both or neither of $G$ and $H$ contain loops. An example of this happening is with cycles. If $\F$ is the family of (disjoint unions of) loopless cycles, then the relation $\cong_\F$ is cospectrality of adjacency matrices. This is due to the fact that $\hom(C_k,G) = \tr(A_G^k)$, and Newton's identities allow one to compute the eigenvalues of $A_G$ from these numbers\footnote{Strictly speaking we do not have direct access to $\tr(A_G)$ nor $\tr(A_G^2)$, but we can compute the eigenvalues of $A_G^3$ from the numbers $\tr(A^{3k})$ for $k \ge 1$ and the eigenvalues of $A_G$ are just the unique real cube roots of these.}. Let $\F'$ be the homomorphism distinguishing closure (in the setting of loopless graphs) of $\F$. Now for any two graphs (with loops allowed) $G$ and $H$, the relation $G \cong_{\F'} H$ is equivalent to $G$ and $H$ being cospectral, and therefore this implies that $\tr(A_G) = \tr(A_H)$. However, $\tr(A_G)$ is just the number of loops contained in $G$ which is the number of homomorphisms from the graph on a single vertex with a loop to $G$. Thus $\F'$ is not homomorphism distinguishing closed in the loops allowed setting. Note that this implies that if $\F$ is any family of loopless graphs that contains all cycles, then its closure in the loops allowed setting must include the graph on a single vertex with a loop, and thus $\F$ cannot be closed in the loops allowed setting.

On the other hand, there are examples of families that are h.d.~closed in the loopless setting that remain h.d.~closed in the loops allowed setting. For example take the family of forests which we showed is h.d.~closed in Corollary~\ref{cor:forests}. It is easy to see that if $G$ is the graph with two vertices both containing a loop and $H$ is $K_2$, then any forest (in fact any bipartite graph) with $c$ connected components has exactly $2^c$ homomorphisms to both $G$ and $H$. However, any graph with loops will have a positive number of homomorphisms to $G$ but no homomorphisms to $H$. Thus the family of forests is also h.d.~closed in the loops allowed setting.

Given a family of loopless graphs $\F$, let $\mathrm{cl}(\F)$ continue to denote the usual homomorphism distinguishing closure in the no loops allowed setting that we defined in Section~\ref{sec:degree}, and let $\mathrm{cl}^\circ(\F)$ denote the homomorphism distinguishing closure in the loops allowed setting. It is not a priori obvious that $\mathrm{cl}(\F) \subseteq \mathrm{cl}^\circ(\F)$ for every $\F$. This is because it might be the case that there is some $F \notin \F$ such that $G \cong_\F H$ implies $\hom(F,G) = \hom(F,H)$ for all \emph{loopless} graphs $G$ and $H$, but this implication does not hold for some pair $G$ and $H$ where one or both of them have loops. Thus we have the following question:

\begin{question}
If $\F$ is a family of loopless graphs, is $\mathrm{cl}(\F) \subseteq \mathrm{cl}^\circ(\F)$?
\end{question}

Of course we have already seen that it is possible that $\mathrm{cl}^\circ(\F) \not\subseteq \mathrm{cl}(\F)$, and thus it may be that for some families $\F$ neither containment holds. On the other hand we will always have that every loopless graph in $\mathrm{cl}^\circ(\F)$ is contained in $\mathrm{cl}(\F)$. Thus if $\F$ is a family of loopless graphs that is h.d.~closed in the no loops setting, then $\mathrm{cl}^\circ(\F) \setminus \F$ will only contain graphs with loops. A natural question is whether the loopless version of any graph in $\mathrm{cl}^\circ(\F) \setminus \F$ is necessarily contained in $\F$:

\begin{question}
Let $\F$ be a family of loopless graphs that is h.d.~closed in the no loops allowed setting. Can any graph in $\mathrm{cl}^\circ(\F) \setminus \F$ be obtained from a graph in $\F$ by adding loops to some of the vertices?
\end{question}

Another natural question is the following:

\begin{question}
Let $\F$ be a family of graphs possibly with loops that is homomorphism distinguishing closed in the loops allowed setting. If $\F'$ is the family of loopless graphs from $\F$, is $\F'$ homomorphism distinguishing closed in the loopless setting?
\end{question}

A negative answer to the above is possible since for a potential family $\F$ there might be a loopless graph $F$ such that $G \cong_\F H$ implies $\hom(F,G) = \hom(F,H)$ whenever $G$ and $H$ are loopless, but this implication does not hold for all $G$ and $H$ where loops are allowed.

Note that for any family $\F$ of loopless graphs, if we let $\F'$ be the family consisting of the graphs in $\F$ and all graphs containing loops, then the relation $\cong_{\F'}$ will coincide with the relation $\cong_\F$ when both graphs being compared are loopless. From this and Theorem~\ref{thm:uncountable} it follows that there are uncountably many homomorphism indistinguishability relations even in the setting with loops. However, it seems likely that for graphs with loops the relation $\cong_{\F'}$ will differ significantly from $\cong_{\mathrm{cl}^\circ(\F)}$, which is probably the more natural extension of $\cong_{\F}$ to the setting of graphs with loops. This issue arises when considering extending many of the families we considered to the loops allowed setting. For instance, most of the families we considered were of the form given in Theorem~\ref{thm:homclosed2}, i.e., the family $\F$ of graphs all of whose connected components do not admit a weak oddomorphism to any graph in some family $\mathcal{G}$. Since weak oddomorphisms are homomorphisms, no graph with loops can have a weak oddomorphism to a loopless graph. Moreover, even if we allow loops in graphs appearing in $\mathcal{G}$, we have no analog of Theorem~\ref{thm:maindualG} for graphs with loops, and so we would not be able to apply this theorem to e.g.~show that the resulting family $\F$ is h.d.~closed. To partially address this we present an alternative construction of $G_0$ and $G_1$ in the section below.%need to extend our construction of $G_0$ and $G_1$ to graphs $G$ with loops and extend Theorem~\ref{thm:maindualG} to this case.

\subsection{An alternative to $G_0$ and $G_1$}\label{sec:construction2}

There are perhaps several possible ways to extend the construction of $G_0$ and $G_1$ to graphs with loops. Unfortunately there does not appear to be one that perfectly meets our desires. By this we mean that we would like to have an analog of Theorem~\ref{thm:maindualG} for connected graphs $G$ possibly having loops, and we would like the notion of (weak) oddomorphisms used in this analog to be defined precisely the same as in Definition~\ref{def:oddomorphism} but where $N_F(a)$ contains $a$ whenever $a$ has a loop and similarly for $N_G(v)$. Starting from this proposed notion of oddomorphism, one can work backward, first constructing the analog of the linear systems obtained from the Fredholm Alternative in Equation~\eqref{eq:cert}, and then ``unapplying" the Fredholm Alternative to find the linear systems whose solutions we want to be in bijection with the elements of $\Hom_\psi(F,G_0)$ and $\Hom_\psi(F,G_1)$. The problem is that the linear systems you end up with when doing this will impose that a vertex $a \in V(F)$ with a loop must be mapped to a vertex $(v,S)$ of $G_i$ where $v$ has a loop and $S$ does not contain this loop, but will only impose that distinct adjacent vertices $b,c \in V(F)$ with $\psi(b) = \psi(c) = v$ are mapped to $(v,S)$ and $(v,T)$ where $S \triangle T$ does not contain the loop on $v$. The latter condition requires us to put loops on all vertices $(v,S)$ where $v \in V(G)$ has a loop, but the former condition requires us to only put loops on such vertices where $S$ does not contain the loop on $v$. Currently we do not see how to resolve this conflict and derive an analog of Theorem~\ref{thm:maindualG} for graphs with loops meeting the requirements described above. Instead, we give the following definition of a graph $\tilde{G}_U$ for a graph $G$ and $U \subseteq V(G)$ that allows us to obtain analogs of some of our results for graphs with loops. Note that the construction itself only applies to graphs $G$ without loops. %For any graph $G$ the graph $\tilde{G}_U$ will contain loops, but $\tilde{G}_U$ will be independent of any loops appearing in $G$. We will use $\tilde{E}(v)$ to denote the set of non-loop edges incident to a vertex $v$.

\begin{definition}\label{def:construction3}
Let $G$ be a loopless graph and $U \subseteq V(G)$. Define $\tilde{G}_U$ as the graph with vertex set $\{(v,S) : v \in V(G), \ S \subseteq E(v), \ |S| \equiv \delta_{v,U} \ \mathrm{mod} \ 2\}$, where $(v,S)$ is adjacent to $(u,T)$ if $uv \notin E(G)$, or if $uv \in E(G)$ and $uv \notin S \triangle T$.
\end{definition}

Note that by the definition above, all vertices of $\tilde{G}_U$ have a loop. Furthermore, for a fixed $v \in V(G)$, the vertices of the form $(v,S)$ form a complete subgraph in $\tilde{G}_U$ as opposed to an independent set as in $G_U$.

We can now proceed as in Section~\ref{sec:construction}. Since the analysis is so similar, we only provide an overview here. The main difference with this construction is that the projection $\rho: (v,S) \mapsto v$ is no longer a homomorphism from $\tilde{G}_U$ to $G$, it is only a map from $V(\tilde{G}_U)$ to $V(G)$. Thus we will partition $\Hom(F,\tilde{G}_U)$ into the sets $\Hom_\psi(F,\tilde{G}_U)$ where $\psi$ varies over all functions from $V(F)$ to $V(G)$. For a fixed $\psi : V(F) \to V(G)$, using the exact same analysis as in the proof of Lemma~\ref{lem:homs2sols} we again obtain that the elements of $\Hom_\psi(F,\tilde{G}_U)$ are in bijection with the solutions to the following system of equations over $\mathbb{Z}_2$:
\begin{align}
    \sum_{e \in E(\psi(a))} x_e^a &= \delta_{\psi(a),U} \text{ for all } a \in V(F);\label{eq:parity2}\\
    x_e^a + x_e^b &= 0 \text{ for all } ab \in E(F) \text{ s.t. } \psi(a)\psi(b) \in E(G), \text{ where } e = \psi(a)\psi(b).\label{eq:adjacency2}
%    x^a_e &= 0 \text{ if } a \text{ has a loop \& } \psi(a) \text{ has a loop } e.\label{eq:loop3}
\end{align}
Note that if $ab \in E(F)$ but $\psi(a)\psi(b) \notin E(G)$, then Equation~\eqref{eq:adjacency2} does not apply to the edge $ab$. This is because in such a case it is already guaranteed by the construction of $\tilde{G}_U$ that $(\psi(a),S)$ and $(\psi(b),T)$ will be adjacent regardless of $S$ and $T$.

For any $\psi: V(F) \to V(G)$ we can define coefficient matrices $A^\psi$ and $B^\psi$ as in Lemma~\ref{lem:homs2sols}:
\begin{align}
    A^\psi_{b, (a,e)} &= \begin{cases} 1 & \text{if } b = a \\ 0 & \text{o.w.}\end{cases}\\
    B^\psi_{f,(a,e)} &= \begin{cases}1 & \text{if } f \in E(a) \text{ and } e = \psi(f) \in E(G)\\ 0 &\text{o.w.}\end{cases}
\end{align}
where the pairs $(a,e)$ vary over $a \in V(G)$ and $e \in E(\psi(a))$, the row index of $A$ varies over $b \in V(F)$, and the row index of $B$ varies over edges $f$ of $F$ that are mapped to edges of $G$ by $\psi$. Then the linear system appearing in Equations~\eqref{eq:parity2} and~\eqref{eq:adjacency2} is precisely the system
\[\begin{pmatrix}
A^\psi \\ B^\psi
\end{pmatrix}x = \begin{pmatrix}
\chi_{\psi^{-1}(U)} \\ 0
\end{pmatrix}.\]
As in Section~\ref{sec:construction} one can show that the isomorphism class of $\tilde{G}_U$ does depend on and only depends on the parity of $|U|$, and thus we use $\tilde{G}_0$ and $\tilde{G}_1$ to refer to these two distinct graphs. Also as in Section~\ref{sec:construction}, we can apply the Fredholm Alternative and obtain an algebraic certificate for $\hom(F,\tilde{G}_0) \ne \hom(\tilde{G}_1)$ and interpret this combinatorially. This leads to a notion similar to that of (weak) oddomorphism. Before we define this notion recall that for loopless graphs, if $\psi$ is a homomorphism from $F$ to $G$, then $a \in V(F)$ is odd/even with respect to $\psi$ if $|N_F(a) \cap \psi^{-1}(v)|$ is odd/even for all $v \in N_G(\psi(a))$. For a graph $F$ with loops we will take $N_F(a)$ to include $a$ if $a$ has a loop in $F$. For a map $\psi : V(F) \to V(G)$ (with $F$ possibly having loops) we say a vertex $a \in V(F)$ is odd/even with respect to $\psi$ if $|N_F(a) \cap \psi^{-1}(v)|$ is odd/even for all $v \in N_G(\psi(a))$. Since the graph $G$ does not have loops, we will never have that $a \in \psi^{-1}(v)$ for $v \in N_G(\psi(a))$. Thus loops in $F$ do not contribute to the parity of a vertex. We can now define the following:
\begin{definition}
Let $F$ be a graph possibly with loops and let $G$ be a loopless graph. An \emph{oddism} from $F$ to $G$ is a function $\psi : V(F) \to V(G)$ such that
\begin{enumerate}
    \item each vertex of $F$ is either odd or even with respect to $\psi$;
    \item $\psi^{-1}$ contains an odd number of odd vertices for all $v \in V(G)$.
\end{enumerate}
A function $\psi : V(F) \to V(G)$ is a \emph{weak oddism} if there is a subgraph $F'$ of $F$ such that $\psi|_{V(F')}$ is an oddism from $F'$ to $G$.
\end{definition}

Note that the (weak) oddisms from a graph $F$ are precisely the same as the (weak) oddisms from the graph $F'$ obtained from $F$ by removing all of its loops, i.e., loops do not affect (weak) oddisms.

Following the same arguments as in the proof of Theorem~\ref{thm:maindualG} we obtain the following theorem:

\begin{theorem}\label{thm:maindualGloop}
Let $F$ be a graph possibly with loops and let $G$ be a loopless connected graph. Then $\hom(F,\tilde{G}_0) \ge \hom(F,\tilde{G}_1)$ with strict inequality if and only if $F$ has a weak oddism to $G$. Moreover, if such a weak oddism $\psi$ does exist, then there is a connected subgraph $F'$ of $F$ such that $\psi|_{V(F')}$ is an oddism from $F'$ to $G$.
\end{theorem}

It turns out that a graph $F$ has weak oddism to a graph $G$ if and only if it contains a subgraph with a (weak) oddomorphism to $G$:

\begin{lemma}\label{lem:oddism2oddo}
Let $F$ be a graph possibly with loops and let $G$ be a connected graph. Then the following are equivalent:
\begin{enumerate}
    \item $F$ has a weak oddism to $G$;
    \item $F$ contains a subgraph that has an oddomorphism to $G$;
    \item $F$ contains a subgraph that has a weak oddomorphism to $G$.
\end{enumerate}
\end{lemma}
\proof
The equivalence of $(2)$ and $(3)$ is straighforward. Suppose that $F$ contains a subgraph $F'$ that has an oddomorphism $\psi'$ to $G$. Note that $\psi'$ is also an oddism from $F'$ to $G$. Extend $\psi'$ to a function $\psi : V(F) \to V(G)$ arbitrarily. Then $\psi|_{V(F')} = \psi'$ is an oddism from $F'$ to $G$ and thus $\psi$ is a weak oddism from $F$ to $G$.

Now suppose that $\psi$ is a weak oddism from $F$ to $G$. Then $F$ contains a subgraph $F'$ such that $\psi|_{V(F')}$ is an oddism from $F'$ to $G$. Construct $F''$ from $F'$ by removing all edges within fibres of $\psi|_{V(F')}$ and between fibres of $\psi|_{V(F')}$ corresponding to non-adjacent vertices of $G$. Then it is easy to see that $\psi|_{V(F'')}$ is an oddomorphism from $F''$ to $G$.\qeds

Note that the subgraph of $F$ having a (weak) oddomorphism to $G$ in the lemma above is necessarily loopless since a (weak) oddomorphism is a homomorphism and $G$ has no loops.

Using the above we can obtain analogs of some of our results from previous sections. For example, we have the following:

\begin{lemma}\label{lem:loopdegree}
Let $G = K_{1,d}$ be the star on $d+1$ vertices. If $F$ is a graph, then $F$ has a weak oddism to $G$ if and only if $F$ contains a vertex with at least $d$ neighbors distinct from itself.
\end{lemma}
\proof
Let $F$ be a graph, possibly with loops, that has a weak oddism $\psi$ to $G = K_{1,d}$. By Lemma~\ref{lem:oddism2oddo}, $F$ contains a subgraph $F'$ that has an oddomorphism to $K_{1,d}$, and this subgraph must be loopless. Therefore $F'$ has a vertex with at least $d$ neighbors distinct from itself by Lemma~\ref{lem:degree}, and thus so does $F$.%, and let $0,1,2, \ldots, d$ be the vertices of $G$ such that $0$ is adjacent to $1,2, \ldots, d$. By definition we have that there is a subgraph $F'$ of $F$ such that $\psi|_{V(F')}$ is an oddism from $F'$ to $G$. Thus there is an odd vertex $v$ of $F'$ in $\psi^{-1}(0)$ which must have an odd number of neighbors in each $|psi^{-1}(i)$ for $i = 1, \ldots, d$. Therefore $v$ must have at least $d$ neighbors distinct from itself, as desired.

Now suppose that $F$ contains a vertex $v$ with at least $d$ neighbors distinct from itself and let $v_1, \ldots, v_d$ be $d$ such neighbors. Let $F'$ be the subgraph of $F$ consisting of $v$ and $v_1, \ldots, v_d$, and the edges $vv_i$ for $i = 1, \ldots, d$. Then $F' \cong K_{1,d}$ and thus $F'$ has an oddomorphism to $K_{1,d}$. Therefore by Lemma~\ref{lem:oddism2oddo}, the graph $F$ has a weak oddism to $K_{1,d}$.\qeds%Now suppose that $F$ contains a vertex $v$ with at least $d$ neighbors distinct from itself and let $v_1, \ldots, v_d$ be $d$ such neighbors. Let $F'$ be the subgraph of $F$ consisting of $v$ and $v_1, \ldots, v_d$, and the edges $vv_i$ for $i = 1, \ldots, d$. Let the vertices of $G = K_{1,d}$ be $0,1, \ldots, d$ where $0$ is adjacent to $1,\ldots, d$. Define $\psi : V(F) \to V(G)$ on $V(F')$ as $\psi(v) = 0$ and $\psi(v_i) = i$, and define $\psi$ arbitrarily on the rest of $V(F)$. Clearly, $\psi|_{V(F')}$ is an oddism from $F'$ to $G$ and thus $\psi$ is a weak oddism from $F$ to $G$.\qeds

Note that for loopless graphs $F$ the above says that $F$ has a weak oddism to $K_{1,d}$ if and only if the maximum degree of $F$ is at least $d$. From this we obtain the following:

\begin{theorem}
If $G = K_{1,d}$, then $\hom(F,\tilde{G}_0) \ne \hom(F,\tilde{G}_1)$ if and only if $F$ contains a vertex with at least $d$ neighbors distinct from itself. Furthermore, this implies that the class of graphs that do not contain any vertex with at least $d$ neighbors distinct from itself is homomorphism distinguishing closed.
\end{theorem}
\proof
The first claim is immediate from Lemma~\ref{lem:loopdegree} and Theorem~\ref{thm:maindualGloop}. The second claim follows immediately from the first.\qeds

Let $\F$ be the family of graphs with no vertex having at least $d$ neighbors distinct from itself. Then the above shows that there is a single pair of graphs, namely $\tilde{G}_0$ and $\tilde{G}_1$ for $G = K_{1,d}$, such that $\tilde{G}_0 \cong_\F \tilde{G}_1$ but $\hom(F,\tilde{G}_0) \ne \hom(F,\tilde{G}_1)$ for all graphs $F \notin \F$. This is in contrast to the loopless setting where we needed a different such pair of graphs for each $F \notin \F$.

By Theorem~\ref{thm:chordlesscycles} and Lemma~\ref{lem:oddism2oddo}, if $F$ has a weak oddism to the cycle $C_k$, then $F$ must contain a cycle of length at least and of the same parity as $k$. In fact we can also prove the converse:

\begin{lemma}
A graph $F$ has a weak oddism to the cycle $C_k$ if and only if it contains a cycle of length at least and of the same parity as $k$.
\end{lemma}
\proof
We have already proven one direction in the paragraph above. Conversely, suppose that $F$ contains a cycle $C$ of length at least and of the same parity as $k$. Then by Corollary~\ref{cor:cycleoddos}, the cycle $C$ has an oddomorphism to $C_k$ and thus by Lemma~\ref{lem:oddism2oddo} $F$ has a weak oddism to $C_k$.\qeds%, then by Lemma~\ref{lem:asdf}  Suppose then that $F$ contains a cycle $C = v_0,v_1, \ldots, v_{k'-1}$ of length $k'$ where $k' \ge k$ and $k' \equiv k \ \mathrm{mod} \ 2$. Let $0,1, \ldots, k-1$ be the vertices of $C_k$ such $i \sim i+1$ ($\mathrm{mod} \ 2$) for all $i$. Let
%\[\psi(v_i) = \begin{cases}i & \text{if } i \in \{0,1,\ldots, k-1\}\\ 0 & \text{if } i \in \{k, k+1, \ldots, k'-1\} \ \& \ i \equiv k \ \mathrm{mod} \ 2\\ k-1 & \text{if } i \in \{k, k+1, \ldots, k'-1\} \ \& \ i \not\equiv k \ \mathrm{mod} \ 2\end{cases}.\]
%It is straightforward to see that $\psi$ is an oddomorphism from $C$ to $C_k$ where the even vertices with respect to $\psi$ are precisely $\{k,k+1, \ldots, k'-1\}$. Therefore we have that $F$ has a weak oddism to $C_k$ by Lemma~\ref{lem:oddism2oddo}.\qeds

From the above we can extend some of our results about cycles to the setting with loops, though we can no longer require chordlessness. In particular we have the following two theorems whose proofs are analogous to those in Section~\ref{sec:hdclosed}:

\begin{theorem}
Let $k \ge 3$. The family of graphs not containing cycles of length at least and of the same parity as $k$ is homomorphism distinguishing closed in the loops allowed setting.
\end{theorem}

\begin{theorem}
Let $k \ge 3$. The family of graphs of circumference at most $k$ is homomorphism distinguishing closed in the loops allowed setting.
\end{theorem}

\section{Discussion}\label{sec:discussion}

The most important open question addressed in this work is whether Conjectures~(1)--(4) hold. A positive resolution of these conjectures would reveal a beautiful connection between minor-closed classes and homomorphism indistinguishability. Among these four equivalent conjectures, we believe that Conjecture 3 is the most directly approachable. Recall that this conjecture states that for any connected graph $G$, there exist graphs $H$ and $H'$ such that $\hom(G,H) \ne \hom(G,H')$ and $\hom(F,H) \ne \hom(F,H')$ implies that $F$ contains $G$ as a minor. Note that the latter condition is equivalent to saying that $H \cong_{\F_G} H'$ where $\F_G$ is the family of graphs not containing $G$ as a minor. Thus understanding the homomorphism indistinguishability relations of the form $\cong_{\F_G}$ may help resolve this conjecture. It may also be helpful to understand homomorphism indistinguishability over the finite family $\F^G$ of graphs that are minors of $G$, since the conjecture is false if and only if $\cong_{\F_G} \Rightarrow \cong_{\F^G}$ for some connected graph $G$. Thus we pose the following, somewhat vague question:

\begin{question}
Given a connected graph $G$, can we characterize the homomorphism indistinguishability relations $\cong_{\F_G}$ and $\cong_{\F^G}$?
\end{question}

By Corollary~\ref{cor:0iso}, we know that $\hom(G,G_0) \ne \hom(G,G_1)$ for any connected graph $G$. Thus to prove Conjecture 3, it would suffice to show that $\hom(F,G_0) \ne \hom(F,G_1)$ implies that $F$ contains $G$ as a minor. By Theorem~\ref{thm:maindualG} we know that $\hom(F,G_0) \ne \hom(F,G_1)$ is equivalent to $F$ admitting a weak oddomorphism to $G$. Thus we ask the following:

\begin{question}\label{q:oddo2minor}
Does the existence of a weak oddomorphism from $F$ to $G$ imply that $F$ contains $G$ as a minor for any connected graph $G$?
\end{question}

Note that it is both necessary and sufficient to answer the above question in the case where $F$ has an oddomorphism to $G$.

We know that the answer to Question~\ref{q:oddo2minor} is yes for a few connected graphs $G$. Specifically, it follows from our results that the answer is yes for paths, cycles, and stars, as well as for $K_4$. The question remains open for essentially all other graphs. 

We do not currently see a good approach to resolving this question. An optimistic but ultimately doomed approach is to try to prove the stronger statement that if $\psi$ is an oddomorphism from $F$ to $G$, then $F$ contains a subdivision of $G$ such that the vertex of this subdivision corresponding to a given vertex $u \in V(G)$ is contained in $\psi^{-1}(u)$. However, this turns out not to be the case since it is not even true that a graph with an oddomorphism to $G$ must contain a subdivision of $G$. An example of a graph $F$ with an oddomorphism to $K_5$ but not containing any subdivision of $K_5$ is shown in Figure~\ref{fig:no5subd}. This example was found using Sage~\cite{sage}. Unfortunately, though detecting minors can be done in polynomial time, in practice minor finding routines are often very slow and thus it is difficult to test Question~\ref{q:oddo2minor} for even small graphs $F$ and $G$.

\begin{center}
\begin{figure}[!ht]
\begin{center}
\begin{tikzpicture}[thick, scale=2]
\tikzstyle{redgraph}=[fill=red, circle, draw=black, inner sep=0.1cm]
\tikzstyle{greengraph}=[fill=green, circle, draw=black, inner sep=0.06cm]
\tikzstyle{yellowgraph}=[fill=yellow, circle, draw=black, inner sep=0.06cm]
\tikzstyle{bluegraph}=[fill=blue, circle, draw=black, inner sep=0.1cm]
\tikzstyle{magentagraph}=[fill=magenta, circle, draw=black, inner sep=0.06cm]

\node[bluegraph] (v0) at (90:1cm) {};
\node[bluegraph] (v1) at (162:1cm) {};
\node[redgraph] (v00) at (90:1.5cm) {};
\node[redgraph] (v11) at (162:1.5cm) {};
\node[bluegraph] (v2) at (234:1cm) {};
\node[bluegraph] (v3) at (306:1cm) {};
\node[bluegraph] (v4) at (18:1cm) {};
\node[redgraph] (v44) at (18:1.5cm) {};

\draw (v0)--(v11)--(v00)--(v1)--(v2)--(v3)--(v4)--(v00)--(v44)--(v0);
\draw (v0)--(v2)--(v4)--(v1)--(v3)--(v0);

\draw [thick, dashed, black]  (90:1.25cm) ellipse (0.3cm and 0.6cm);

\begin{scope}[rotate=72]
\draw [thick, dashed, black]  (90:1.25cm) ellipse (0.3cm and 0.6cm);
\end{scope}

\begin{scope}[rotate=-72]
\draw [thick, dashed, black]  (90:1.25cm) ellipse (0.3cm and 0.6cm);
\end{scope}

\begin{scope}[rotate=144]
\draw [thick, dashed, black]  (90:1cm) ellipse (0.3cm and 0.3cm);
\end{scope}

\begin{scope}[rotate=-144]
\draw [thick, dashed, black]  (90:1cm) ellipse (0.3cm and 0.3cm);
\end{scope}

\end{tikzpicture}
\end{center}
\caption{A graph with an oddomorphism to $K_5$ but that contains no subdivision of $K_5$. Fibres of the oddomorphism are indicated by the sets of vertices enclosed by dashed lines. Blue vertices are odd and red vertices are even.}
\label{fig:no5subd}
\end{figure}
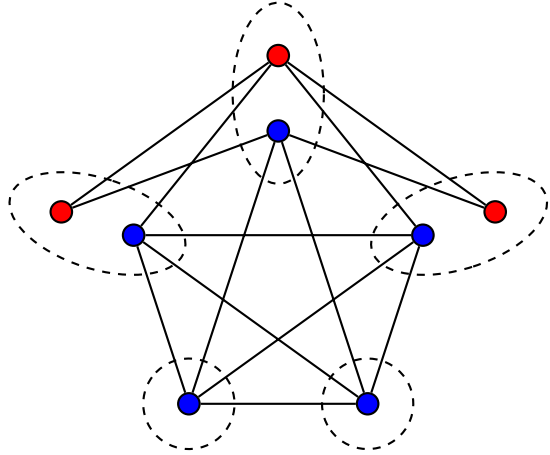
\end{center}

Perhaps a good first test case for Question~\ref{q:oddo2minor} is with $G = K_5$. The graph $K_5$ is the first complete graph for which we do not know the answer to this question, and other than star graphs we do not know the answer to the question for any graphs of maximum degree greater than 3. Thus one might expect that $K_5$ is a good place to look for a counterexample. However in an upcoming work with Richter and Thomassen, we prove that any graph having a weak oddomorphism to $K_5$ must contain $K_5$ as a minor.

One possible idea for answering Question~\ref{q:oddo2minor} is to use induction. For this we need some way of reducing the number of vertices in a graph $F$ that has an oddomorphism to a graph $G$ while still maintaining the property of having such an oddomorphism. Simply deleting vertices does not work, as this does not preserve the property of having an oddomorphism to a given graph. However, there is a simple operation we can perform that does preserve this property. Suppose that $F$ is a graph with an oddomorphism $\psi$ to a graph $G$. Select two vertices $a,b \in V(F)$ that are in the same fibre of $\psi$ and merge these into a single vertex $c$ but make the neighbors of $c$ the symmetric difference of the neighbor sets of $a$ and $b$. Then it is not difficult to see that this graph still admits an oddomorphism to $G$ with all of the same fibres except the fibre of $\psi$ containing $a$ and $b$ now contains $c$ instead. In fact, we use this operation for some inductive proofs in the aforementioned upcoming work with Richter and Thomassen. It is also easy to see that repeated applications of this operations eventually results in a graph $F'$ with an oddomorphism to $G$ whose fibres all have size one and therefore $F' \cong G$. Thus if one can show that the reverse of this operation preserves the property of having $G$ as a minor then Question~\ref{q:oddo2minor} will be answered in the affirmative. However, this seems difficult as it is not at all clear how this operation affects the structure of a potential $G$-minor.

\subsection{Covers and quantum isomorphism}\label{sec:quantum}

As we saw in Lemma~\ref{lem:covers}, if $F$ is an odd cover of a graph $G$, then the covering map from $F$ to $G$ is an oddomorphism. Thus a positive answer to Question~\ref{q:oddo2minor} would imply that any odd cover of a graph $G$ must contain $G$ as a minor. This would be a surprising result on its own, as it would provide an unexpected connection between covers and minors. However, there is at least one known result in this direction. In 1990, Archdeacon and Richter showed that an odd cover of a non-planar graph must be non-planar. In other words, any odd cover of a non-planar graph must contain either $K_{3,3}$ or $K_5$ as a minor. Our results here, along with the result of~\cite{qplanar}, allow us to generalize the result of Archdeacon and Richter: we can show that any graph with an oddomorphism to a non-planar graph must be non-planar. To do this we must make use of the notion of quantum isomorphism and the result of~\cite{qplanar} which showed that quantum isomorphism is equivalent to homomorphism indistinguishability over the class of planar graphs. First we prove the following which is closely related to Theorem 6.3 of~\cite{qiso1}:

\begin{lemma}\label{lem:q2nonplanar}
Let $G$ be a connected graph. Then $G_0$ and $G_1$ are quantum isomorphic if and only if $G$ is non-planar.
\end{lemma}
\proof
The results of~\cite{qiso1} provide a construction of graphs $G'_0$ and $G'_1$ that are quantum isomorphic if and only if $G$ is non-planar. The difference between these $G'_0$ and $G'_1$ and our $G_0$ and $G_1$ is that in the former an edge is put between vertices $(u,S)$ and $(v,T)$ when $uv \in E(G)$ and $uv \in S \triangle T$ rather than when $uv \in E(G)$ and $uv \notin S \triangle T$ (in the former construction one also puts edges between all vertices with the same first coordinate). However it is straightforward to check that the quantum isomorphism between $G'_0$ and $G'_1$ given in~\cite{qiso1} is also a quantum isomorphism between $G_0$ and $G_1$. Thus if $G$ is non-planar then $G_0$ and $G_1$ are quantum isomorphic.

Conversely, since $G$ is connected we have that $\hom(G,G_0) \ne \hom(G,G_1)$. By the titular result of~\cite{qplanar} two graphs are quantum isomorphic if and only if they admit the same number of homomorphisms from any planar graph. Thus if $G$ is planar then $G_0$ and $G_1$ cannot be quantum isomorphic.\qeds

It would be nice to be able to prove the above result without relying on the result of~\cite{qplanar}, however we were not able to adapt the other direction of the proof from~\cite{qiso1} for our graphs.

\begin{lemma}\label{lem:oddo2nonplanar}
Let $F$ be a planar graph. If $F$ has an (weak) oddomorphism to $G$ then $G$ is planar.
\end{lemma}
\proof
We prove that contrapositive, that if $F$ has an (weak) oddomorphism to non-planar $G$, then $F$ is non-planar. Without loss of generality we may assume that $G$ is connected, since otherwise we may restrict to a non-planar connected component of $G$ via Lemma~\ref{lem:connected2connected}. By Lemma~\ref{lem:q2nonplanar} we know that $G_0$ and $G_1$ are quantum isomorphic. Now suppose that $F$ has a weak oddomorphism to $G$. By Theorem~\ref{thm:main} we have that $\hom(F,G_0) \ne \hom(F,G_1)$, and thus $F$ must be non-planar by the main result of~\cite{qplanar}.\qeds

Thus we have obtained our generalization of the result of Archdeacon and Richter. Notably, their proof is much shorter and more direct, whereas ours rests on the result of~\cite{qplanar} which requires 50+ pages and uses the theory of quantum groups. The fact that (weak) oddomorphisms preserve planarity seems a rather interesting property as in general homomorphisms do not preserve planarity and we are not aware of any well-studied type of homomorphism that does. Also note that by Theorem~\ref{thm:homclosed1} this implies that the family of planar graphs is homomorphism distinguishing closed:

\begin{theorem}\label{thm:planarhdclosed}
The family of planar graphs is homomorphism distinguishing closed.
\end{theorem}

Finding interesting examples of non-isomorphic yet quantum isomorphic graphs is difficult as there are lamentably few known constructions of such graphs, and it is known that determining quantum isomorphism is undecidable. There are some interesting tractable necessary conditions for quantum isomorphism~\cite{qiso2} but no nontrivial sufficient conditions are known. If there were some ``natural" equivalence relation sitting strictly between quantum isomorphism and isomorphism and this relation were tractable, then this would provide a tractable sufficient condition and may lead to new examples/constructions of non-isomorphic yet quantum isomorphic graphs. Thus we pose the following question:

\begin{question}
Is there a tractable homomorphism indistinguishability relation strictly between quantum isomorphism and isomorphism, i.e., is there a homomorphism distinguishing closed family strictly between planar graphs and all graphs whose corresponding relation is tractable?
\end{question}

\subsection{Other questions}

Conjecture~\ref{conj:distinct} states that any two distinct minor- and union-closed families give rise to distinct homomorphism indistinguishability relations. However even the following relaxation would be an interesting result:

\begin{conjecture}\label{conj:distinctfromiso1}
Let $\mathcal{F}$ be any minor- and union-closed family of graphs that does not contain all graphs. Then $\cong_\F$ is not isomorphism.
\end{conjecture}

We remark that in the above conjecture the property of being union-closed is not needed since if $\F$ is any proper minor closed class, then closing it under disjoint unions would not yield all graphs.

Notably, the class of 2-degenerate graphs whose homomorphism indistinguishability relation is isomorphism~\cite{dvorak} is not minor-closed. In order to prove the above conjecture it suffices to prove the following:

\begin{conjecture}\label{conj:distinctfromiso2}
For any $n \in \mathbb{N}$, there exist non-isomorphic graphs $H$ and $H'$ such that $\hom(F,H) \ne \hom(F,H')$ implies that $F$ contains $K_n$ as a minor.
\end{conjecture}

In fact, the above two conjectures are equivalent:

\begin{theorem}
Conjectures~\ref{conj:distinctfromiso1} and~\ref{conj:distinctfromiso2} are equivalent.
\end{theorem}
\proof
Suppose that Conjecture~\ref{conj:distinctfromiso2} holds. Let $\mathcal{F}$ be any minor- and union-closed class which is not all graphs. There must be some $n \in \mathbb{N}$ such that $K_n \notin \F$, since any graph is the minor of some complete graph. Let $H$ and $H'$ be the graphs guaranteed by Conjecture~\ref{conj:distinctfromiso2}. This means that if $\hom(F,H) \ne \hom(F,H')$ then $F$ contains $K_n$ as a minor and therefore $F \notin \F$. Thus $H \cong_\F H'$, but $H \not\cong H'$ by assumption.%Then $H$ and $H'$ are not isomorphic and $\hom(F,H) \ne \hom(F,H')$ implies that $F$ contains $K_n$ as a minor and therefore $F \notin \F$. Thus $H \cong_\F H'$.

Conversely, suppose that Conjecture~\ref{conj:distinctfromiso1} holds. Let $n \in \mathbb{N}$ and let $\F$ be the family of graphs not containing $K_n$ as a minor. Note that $\F$ is minor- and union-closed. Therefore, by assumption we have that there exist graphs $H$ and $H'$ such that $H \cong_\F H'$ but $H \not\cong H'$. Therefore, if $\hom(F,H) \ne \hom(F,H')$ then $F \notin \F$ and this is equivalent to $F$ having $K_n$ as a minor.\qeds

Conjectures~\ref{conj:distinctfromiso1} and~\ref{conj:distinctfromiso2} would be implied by an affirmative answer to the following question:

\begin{question}
Given $n \in \mathbb{N}$, does there exist $N \in \mathbb{N}$ such that $F$ admitting a weak oddomorphism to $K_N$ implies that $F$ contains $K_n$ as a minor.
\end{question}

We noted in the introduction that if $\F$ is a minor-closed family, then the relation $\cong_\F$ is preserved under taking complements. We can then ask if the converse is true:

\begin{question}
If $\F$ is a family of graphs such that $\cong_\F$ is preserved under taking complements, then is there a \emph{minor-closed} family $\F'$ such that $\cong_{\F'}$ is equivalent to $\cong_\F$?
\end{question}

If Conjecture~\ref{conj:distinct} holds, then the above question is equivalent to asking whether an h.d.-closed family $\F$ is minor-closed if and only if $\cong_\F$ is preserved under taking complements.

Another interesting question concerns the tractability of homomorphism indistinguishability relations. There are known examples of families $\F$ where $\cong_\F$ can be decided in polynomial time, e.g., $\F$ being the family of cycles, or graphs of bounded treewidth. On the other end of the spectrum, homomorphism indistinguishability over planar graphs is undecidable. Planar graphs contain graphs of unbounded treewidth and so one might wonder if this is related to the undecidability. But of course homomorphism indistinguishability over all graphs is at worst quasi-polynomial time. So it appears that the relationship between a given family $\F$ and the complexity of deciding the relation $\cong_\F$ is not completely straightforward. Thus we pose the following imprecise question:

\begin{question}
What is the relationship between a family $\F$ and the complexity of the relation $\cong_\F$?
\end{question}

In Theorem~\ref{thm:partialorder}, we showed that oddomorphisms and weak oddomorphisms are partial orders. We can thus study the structure of these partial orders. We have seen in Section~\ref{sec:uncountability} that the complete graphs form an infinite antichain in these partial orders, but there are other properties that could be investigated. Actually one can think of Question~\ref{q:oddo2minor} in terms of partial orders. Specifically, it is asking whether the partial order of weak oddomorphisms is a suborder of the (reverse of) the partial order induced by the minor relation. Since weak oddomorphisms are always homomorphisms, an affirmative answer to this question is equivalent to the statement that the weak oddomorphism order is a common suborder of the homomorphism order (which is only a preorder) and the (reverse of) the minor order.

\bibliographystyle{plainurl}
\bibliography{Oddo}

\begin{thebibliography}{10}

\bibitem{arkhipov}
Alex Arkhipov.
\newblock Extending and characterizing quantum magic games.
\newblock 2012.
\newblock \href {http://arxiv.org/abs/1209.3819} {\path{arXiv:1209.3819}}.

\bibitem{qiso1}
Albert Atserias, Laura Man\v{c}inska, David~E. Roberson, Robert \v{S}\'{a}mal,
  Simone Severini, and Antonios Varvitsiotis.
\newblock Quantum and non-signalling graph isomorphisms.
\newblock {\em Journal of Combinatorial Theory, Series B}, 136:289 -- 328,
  2019.
\newblock \href {https://doi.org/https://doi.org/10.1016/j.jctb.2018.11.002}
  {\path{doi:https://doi.org/10.1016/j.jctb.2018.11.002}}.

\bibitem{DGR}
Holger Dell, Martin Grohe, and Gaurav Rattan.
\newblock {Lov{\'a}sz Meets Weisfeiler and Leman}.
\newblock In {\em 45th International Colloquium on Automata, Languages, and
  Programming (ICALP 2018)}, volume 107 of {\em Leibniz International
  Proceedings in Informatics (LIPIcs)}, pages 40:1--40:14, Dagstuhl, Germany,
  2018. Schloss Dagstuhl--Leibniz-Zentrum fuer Informatik.
\newblock \href {https://doi.org/10.4230/LIPIcs.ICALP.2018.40}
  {\path{doi:10.4230/LIPIcs.ICALP.2018.40}}.

\bibitem{dvorak}
Zden{\v{e}}k Dvo{\v{r}}{\'a}k.
\newblock On recognizing graphs by numbers of homomorphisms.
\newblock {\em Journal of Graph Theory}, 64(4):330--342, 2010.
\newblock \href {https://doi.org/10.1002/jgt.20461}
  {\path{doi:10.1002/jgt.20461}}.

\bibitem{FGLSS}
Uriel Feige, Shafi Goldwasser, Laszlo Lov\'{a}sz, Shmuel Safra, and Mario
  Szegedy.
\newblock Interactive proofs and the hardness of approximating cliques.
\newblock {\em J. ACM}, 43(2):268--292, March 1996.
\newblock \href {https://doi.org/10.1145/226643.226652}
  {\path{doi:10.1145/226643.226652}}.

\bibitem{grohetreedepth}
Martin Grohe.
\newblock Counting bounded tree depth homomorphisms.
\newblock In {\em Proceedings of the 35th Annual ACM/IEEE Symposium on Logic in
  Computer Science}, LICS '20, page 507–520. Association for Computing
  Machinery, 2020.
\newblock \href {https://doi.org/10.1145/3373718.3394739}
  {\path{doi:10.1145/3373718.3394739}}.

\bibitem{homtensors}
Martin Grohe, Gaurav Rattan, and Tim Seppelt.
\newblock Homomorphism tensors and linear equations.
\newblock 2021.
\newblock \href {http://arxiv.org/abs/2111.11313} {\path{arXiv:2111.11313}}.

\bibitem{immerman1990}
Neil Immerman and Eric Lander.
\newblock {\em Describing Graphs: A First-Order Approach to Graph
  Canonization}, pages 59--81.
\newblock Springer New York, New York, NY, 1990.
\newblock \href {https://doi.org/10.1007/978-1-4612-4478-3_5}
  {\path{doi:10.1007/978-1-4612-4478-3_5}}.

\bibitem{lovasz}
L{\'{a}}szl{\'{o}} Lov{\'{a}}sz.
\newblock On the shannon capacity of a graph.
\newblock {\em Information Theory, IEEE Transactions on}, 25(1):1--7, Jan 1979.
\newblock \href {https://doi.org/10.1109/TIT.1979.1055985}
  {\path{doi:10.1109/TIT.1979.1055985}}.

\bibitem{qiso2}
Laura Man\v{c}inska, David~E. Roberson, and Antonios Varvitsiotis.
\newblock Semidefinite relaxations of quantum isomorphism.
\newblock In progress, 2017.

\bibitem{qplanar}
Laura Mančinska and David~E. Roberson.
\newblock Quantum isomorphism is equivalent to equality of homomorphism counts
  from planar graphs.
\newblock In {\em 2020 IEEE 61st Annual Symposium on Foundations of Computer
  Science (FOCS)}, pages 661--672, 2020.
\newblock \href {http://arxiv.org/abs/1910.06958} {\path{arXiv:1910.06958}},
  \href {https://doi.org/10.1109/FOCS46700.2020.00067}
  {\path{doi:10.1109/FOCS46700.2020.00067}}.

\bibitem{sage}
{The Sage Developers}.
\newblock {\em {S}ageMath, the {S}age {M}athematics {S}oftware {S}ystem
  ({V}ersion x.y.z)}, YYYY.
\newblock {\tt https://www.sagemath.org}.

\bibitem{tinhofer}
G.~Tinhofer.
\newblock Graph isomorphism and theorems of birkhoff type.
\newblock {\em Computing}, 36(4):285--300, Dec 1986.
\newblock \href {https://doi.org/10.1007/BF02240204}
  {\path{doi:10.1007/BF02240204}}.

\end{thebibliography}
%\bibliographystyle{plainurl}
%
%\bibliography{Qplanar.bbl}

\end{document}